\documentclass[11pt, reqno]{amsart}
\usepackage{charter}
\usepackage{mathrsfs}

\usepackage{graphicx} 
\usepackage[margin=1in]{geometry}
\usepackage{amsmath,amsthm,amssymb,enumerate,fancyhdr,bbm}
\usepackage{natbib}
\usepackage{hyperref} 
\usepackage{mathtools}
\hypersetup{colorlinks=true, linkcolor=blue, citecolor=blue, filecolor=blue, urlcolor=blue}
\DeclareMathOperator{\one}{\mathbbm{1}} 

\numberwithin{equation}{section}
\newtheorem{theorem}{Theorem}[section]

\newtheorem{fact}{Fact}[section]
\newtheorem{lemma}{Lemma}[section]

\newtheorem{remark}{Remark}[section]
\newtheorem{assumption}{Assumption}[section]

\usepackage{mathtools}

\DeclareMathOperator{\Var}{Var}

\newcommand{\E}{\mbox{${\mathbb E}$}}

\usepackage{xcolor}

\newcommand{\reals}{{\mathbb R}}
\newcommand{\bbr}{\reals}

\def\Var{{\rm Var}}
\def\Cov{{\rm Cov}}

\begin{document}
\title[Outliers of Wigner Matrix]{Outlier eigenvalues and eigenvectors of generalized Wigner matrices with finite-rank perturbations}

\author[B. Bhattacharya]{Bishakh Bhattacharya} 
\address{Indian Statistical Institute, 203, B.T. Road, Kolkata.} 
\email{bishakh.rik@gmail.com}
\author[A. Chakrabarty]{Arijit Chakrabarty}
\address{Indian Statistical Institute, 203, B.T. Road, Kolkata.}
\email{arijit.isi@gmail.com}
\author[R. S. Hazra]{Rajat Subhra Hazra}
\address{Mathematical Institute, Leiden University, Einsteinweg 55, 2333 CC Leiden, The Netherlands}
\email{r.s.hazra@math.leidenuniv.nl}

\date{\today}

\begin{abstract}
A generalized Wigner matrix perturbed by a finite-rank deterministic matrix is considered. The fluctuations of the largest eigenvalues, which emerge outside the bulk of the spectrum, and the corresponding eigenvectors, are studied. Under certain assumptions on the perturbation and the matrix structure, we derive the first-order behavior of these eigenvalues and show that they are well separated from the bulk. The fluctuations of these eigenvalues are shown to follow a multivariate Gaussian distribution, and the asymptotic behavior of the associated eigenvectors is also studied. We prove central limit theorems that describe the asymptotic alignment of these eigenvectors with the perturbation's eigenvectors, as well as their Gaussian fluctuations around the origin for non-aligned components. Furthermore, we discuss the convergence of the eigenvector process in a Sobolev space framework. 
\end{abstract}

\keywords{generalized Wigner, empirical spectral distribution, outliers, eigenvector process}
\subjclass[2000]{05C80, 60B20, 60B10, 46L54}
\maketitle

\tableofcontents
\section{Introduction}
Random matrix theory (RMT) has played a pivotal role in understanding the behavior of large complex systems, with applications ranging from quantum physics to data science. One of the fundamental models in this area is the Wigner matrix, which refers to large symmetric matrices with independent and identically distributed (i.i.d.) zero-mean and unit-variance entries. The eigenvalue distribution of such matrices is well understood, with the celebrated Wigner semicircle law describing the bulk behavior of eigenvalues as the matrix size grows. There have been many extensions of the above result to the setting where entries are inhomogeneous, see \cite{anderson2006clt, shlyakhtenko1996random, zhu2020graphon}, for example. In the inhomogeneous setting, the bulk measure satisfies the quadratic vector equations, studied elaborately in \cite{ajanki2017universality, ajanki2019quadratic}. In \cite{chakrabarty2021spectra, zhu2020graphon}, it is shown that in an inhomogeneous Erd\H{o}s-R\'enyi setting, the bulk behaviour of the adjacency matrix is described by a deterministic compactly supported probability measure, which satisfies such an equation. In the classical Wigner setting, the behaviour of the top eigenvalues at the edge of the spectrum is much more complete. Under suitable moment assumptions, the largest eigenvalue converges to $2$ (the right edge of the semicircle) and its fluctuation on the $N^{-2/3}$ scale is governed by the Tracy–Widom law. Analogous edge universality results have been proved for a wide class of Wigner and Wigner-type ensembles, but a fully general theory for inhomogeneous variance profiles is still out of reach. Recently, for a random matrix with a (bounded)variance profile, \cite{ducatez2024large} showed, amongst other things, that the largest eigenvalue, in the first order, converges to the right endpoint of the spectrum.

However, practical systems are often subject to structured perturbations. In particular, finite-rank perturbations of Wigner matrices have garnered significant attention, especially in connection with the Baik–Ben Arous–P\'eché (BBP) transition \cite{baik2005phase, peche2006largest, feral2007largest}. The BBP phenomenon describes a phase transition where, under a rank-one perturbation, the largest eigenvalue separates from the bulk and exhibits new behavior. This result has been extended to various settings, including spiked covariance models in principal component analysis (PCA), where the largest eigenvalue corresponds to the signal and the bulk reflects noise. Some notable references are \cite{benaych2011eigenvalues,capitaine2009largest, baik2006eigenvalues,bassler2009eigenvalue, benaych2011fluctuations}.

 In the rank-one case, \cite{feral2007largest} showed that if the perturbation strength $\theta$ is below a critical value, then the largest eigenvalue sticks to the edge and has Tracy–Widom fluctuations, whereas above the said threshold, it separates from the bulk and fluctuates on an $O(1)$ scale with a non-universal limit law. Building on this, \cite{capitaine2009largest}, and later \cite{pizzo2013finite} analysed general finite-rank perturbations $M_N = W_N/\sqrt{N} + A_N$, where $A_N$ is a deterministic Hermitian matrix with finitely many non-zero eigenvalues $\alpha_1 \ge \cdots \ge \alpha_k$. They proved that for each fixed $i$,
\[
  \mu_i \xrightarrow{\text{a.s.}}
  \begin{cases}
    2, & |\alpha_i| \le 1, \\
    \alpha_i + \alpha_i^{-1}, & |\alpha_i| > 1,
  \end{cases}
\]
and established the following transition in the fluctuation behaviour: below the threshold one recovers the Tracy–Widom statistics at the edge, while above the threshold the outliers have asymptotically Gaussian fluctuations, with a variance that depends on the entry distribution of $W_N$. Related results in a more general framework, including a detailed description of the associated eigenvectors, were obtained by \cite{benaych2011eigenvalues} via free-probability techniques.

 The finite rank perturbations of the symmetric matrix also play a crucial role in the study of random graphs; see, for example,  (\cite{furedi1981eigenvalues}). The properties of the outlier eigenvalues and the corresponding eigenvectors for the Erd\H{o}s-R\'enyi was first studied by \cite{erdHos2013spectral} where it was shown that in a homogeneous Erd\H{o}s-R\'enyi setting, subject to a rank one perturbation, the centred largest eigenvalue under suitable scaling exhibits Gaussian fluctuation. The study was later extended to other eigenvalues and their properties. For example, it was shown in \cite{erdHos2012spectral} that the second largest eigenvalue under suitable centring and scaling exhibits Tracy-Widom Fluctuation. Studies involving random perturbations can be found in \cite{ji2020gaussian}. The above results can be extended to the fully connected regime in the Erd\H{o}s-R\'enyi random graphs and this was shown recently in \cite{tikhomirov2021outliers}. There has also been various extensions to the inhomogeneous Erd\H{o}s-R\'enyi random graph and matrices with variance profile (\cite{chakrabarty2020eigenvalues, boursier2024large, ajanki2017universality}).

In this paper, we consider a generalized Wigner matrix $W_N$ with a bounded variance profile, perturbed by a deterministic finite-rank matrix of the form
\[
  A_N = W_N + \theta_N \sum_{i=1}^k \alpha_i e_i e_i',
\]
where $\theta_N$ is a scalar that may grow with $N$, the $\alpha_i>0$ are fixed distinct signal strengths, and $e_1,\dots,e_k$ are deterministic unit vectors arising from discretizations of orthonormal functions on $[0,1]$. Our primary interest is in the \emph{outlier} eigenvalues of $A_N$ that separate from the $O(\sqrt{N})$ bulk due to the perturbation, and in the finer properties of the associated eigenvectors. Under mild assumptions on the variance profile and on the growth rate of $\theta_N$, we prove central limit theorems for the outlier eigenvalues, establish Gaussian fluctuation results for the alignment of the leading eigenvectors with the planted directions $e_i$.  Finally, we consider a functional eigenvector process by viewing the eigenvector as a function on $[0,1]$. In a Sobolev space framework, we prove a functional CLT, i.e., weak convergence of this eigenvector process to a Gaussian limit (which can be interpreted as a form of Gaussian random field on $[0,1]$) on the said space.

\subsection*{ Main contributions:}

This work provides a detailed analysis of eigenvector behavior in a class of dense random matrices with non-homogeneous variance profile, extending the scope of classical random matrix theory beyond the Wigner regime. Motivated by models in statistical physics and network theory, our setting captures inhomogeneity through a structured variance profile $f(x,y)$, enabling new phenomena that are absent in translation-invariant ensembles.

The principal contributions of this paper are the following.

\begin{itemize}
  \item Eigenvalue fluctuations (Theorem~\ref{Th.EigenValueCLT}).  
    We show that the top \(k\) outlier eigenvalues of a perturbed generalized Wigner matrix, where \(k\) is the rank of the perturbation, fluctuate on an \(O(1)\) scale and jointly converge to a multivariate Gaussian distribution.

  \item Eigenvector alignment and delocalization (Theorems \ref{Th.EigenVectorAlignment.CLT} and \ref{Decolalization}). 
    Under mild regularity of the variance profile and the planted perturbation, the leading eigenvectors align strongly with the deterministic signal directions while all other components remain negligible, and each eigenvector satisfies a sharp sup‐norm delocalization bound.

  \item Functional CLT for the eigenvector process (Theorem~\ref{EVecProMain}).
    Viewing the top eigenvectors as functions on \([0,1]\), we establish a Sobolev‐space central limit theorem for their fluctuation field, proving convergence to a centered Gaussian process with explicitly characterized covariance.

  \item New technical estimates.  
   In the results for the eigenvalues, the Gaussian limit arises from quadratic forms of the type $e_iWe_i^\prime$. In the eigenvector case, especially when it comes to understanding fluctuations, more nonlinearity enters the picture, and terms of the form $e_iW^2e_i^\prime$ become the key building block. To obtain these  results, we develop novel high‐moment and concentration bounds for quadratic and higher‐order forms of the random matrix, as well as precise resolvent expansions and delocalization controls that may be of independent interest. In fact, we give an application of the martingale central limit theorem for a quadratic form of the Wigner matrix (Theorem \ref{Th.MartingaleCLT}). 
\end{itemize}

\paragraph{\bf Outline:} In Section \ref{section:setup}, we formalize the model setup, assumptions, and state the main results. Section \ref{section:eigenvalue fluctuations} contains an outline of the proof of the eigenvalue fluctuation results, including preliminary estimates and the eigenvalue CLT. Since the ideas of the eigenvalue fluctuations are similar to \cite{chakrabarty2020eigenvalues}, the details of the proof of eigenvalue fluctuations are postponed to Appendix \ref{appendixA}. Section \ref{section:evseries} gives some technical estimates and series representations for the eigenvectors. Section \ref{section:alignment} is devoted to the eigenvector alignment results: we prove the asymptotic alignment and CLTs for eigenvector entries, the eigenvector delocalization. The functional CLT for the eigenvector process in the Sobolev space setting is proved in Section \ref{section:EVprocess}. The martingale CLT for a quadratic form of $W_N^2$ is proved in Section \ref{section:MartingaleCLT}.   Important technical lemmas and calculations are deferred to Appendix \ref{appendix:B}. 

\section{The Setup and the Main Results}\label{section:setup}
In this section, we describe the model of a perturbed generalized Wigner matrix  and present the main results on the behavior of its top eigenvalues and eigenvectors. The focus is on the fluctuations of these
quantities around their deterministic limits.

\paragraph{\bf Setup and Assumptions:} 
For each $N \ge 1$, let $W_N$ be a symmetric Wigner-type random matrix of size $N\times N$, defined by
\begin{equation}\label{eq:WignerTypeMatrix} W_N(j,\ell) \;=\; X_{\,j \wedge \ell,\, j \vee \ell}\,,
\end{equation} where $\{X_{j,\ell}: 1 \le j \le \ell \le N\}$ is a collection of independent (not necessarily
identically distributed) random variables with mean 0 and a variance profile given by
\begin{align*}
f\left(\frac{j}{N},\frac{\ell}{N}\right)\,,
\end{align*} 
for some fixed bounded Riemann-integrable function $f: [0,1]^2 \to [0,\infty)$. We impose a mild moment condition on the entries of $W_N$: there exists a constant $B>0$ such
that for all $m \ge 1$, \begin{equation}\label{eq:momentcond} \mathbb{E}[\,|X_{j,\ell}|^m\,] \;\le\; m^{\,B
m}\,. \end{equation} This condition (sometimes called a sub-exponential tail bound) is used to ensure
concentration of measure for various random quadratic forms of $W_N$.
Next, we define a deterministic finite-rank perturbation to $W_N$. Let $k \ge 1$ be fixed (not depending
on $N$). We consider \begin{equation}\label{eq:perturndefn} A_N \;=\; W_N \;+\; \theta_N \sum_{i=1}^{k}
\alpha_i\, e_i\,e_i'\,, \end{equation} where $\theta_N > 0$ is a scalar that may grow with $N$, and $\alpha_1
> \alpha_2 > \cdots > \alpha_k > 0$ are fixed distinct constants (the signal strengths). The vectors $e_1,\dots,e_k$ are deterministic unit vectors in $\mathbb{R}^N$ (thus $e_i$ is $N\times 1$ and $e_i'$ is its
transpose). These vectors represent the directions of the finite-rank perturbation. We assume that $e_i$ is
the discretization of some function $h_i$ on $[0,1]$. In particular, for each $1 \le i \le k$,

$$e_i^{\prime}=\frac{1}{\sqrt{N}}\left(h_i(1 / N), h_i(2 / N), \ldots, h_i(1)\right).$$
Equivalently, the $\ell$-th component of $e_i$ is $e_i(\ell) = h_i(\ell/N)/\sqrt{N}$. We refer to $h_i(x)$ (for
$x\in[0,1]$) as the underlying signal function corresponding to the perturbation vector $e_i$.
We will work under the following assumptions on the growth of $\theta_N$ and the regularity of the
functions $h_i$. In many of our results, only the first two assumptions \ref{assump:A1} and \ref{assump:A2} are needed; the third will
be required for some finer results (notably the functional CLT)
\begin{assumption}[{\bf Growth of \(\theta_N\)}]\label{assump:A1}
There exists $\xi > 4(B+1) \vee 8$ such that
\[
\sqrt{N}(\log N)^{\xi} \preceq \theta_N \preceq N
\]
for all large $N$.
\end{assumption}
Here $a_N \preceq b_N$ means $a_N = O(b_N)$ as $N\to\infty$. Intuitively, $
\theta_N$ grows at least on the order of $\sqrt{N}\,(\log N)^\xi$ (so that the perturbation is
sufficiently strong to create outliers), but at most on the order of $N$.

\begin{assumption}[{\bf Orthogonality of \(h_i\)}]\label{assump:A2}
    The collection ${h_1,\ldots,h_k}$ is an orthonormal set of bounded functions in $L^2[0,1]$. 
\end{assumption}
Equivalently, one can say the vectors $e_1,\ldots,e_k$
defined above are asymptotically orthonormal in $\mathbb{R}^N$ (indeed one can check $e_i' e_j =
\langle h_i, h_j\rangle_{L^2[0,1]} + o(1)$ as $N\to\infty$, which equals $\delta_{ij}$ by
orthonormality).

\begin{assumption}[{\bf Regularity of $h_i$}]\label{assump:A3} Each $h_i: [0,1]\to\mathbb R$ is Lipschitz continuous: there exists $L_i < \infty$ such that

\[
        |h_i(x) - h_i(y)| \leq L_i |x - y|, \quad \forall x, y \in [0,1].
    \]

\end{assumption}
This assumption will only be needed for some results where we require strong regularity, for example in analyzing the eigenvector as a function on $[0,1]$.\\

\medskip
\paragraph{\bf Notation:} We use the notation $\lambda_1(M)\ge \lambda_2(M)\ge \cdots \ge \lambda_N(M)$ for the ordered eigenvalues of an $N\times N$ symmetric matrix $M$. In particular, $\lambda_1(A_N)$ will denote the largest eigenvalue of $A_N$, $\lambda_2(A_N)$ the second largest, and so on. Similarly, for each $1\le i\le k$ we denote by $v_i$ a unit eigenvector of $A_N$ corresponding to $\lambda_i(A_N)$ (we will often call $v_i$ the $i$th eigenvector, understanding that it is determined up to a sign $\pm1$ which we will fix arbitrarily or via a convenient convention). Throughout, we focus on the outlier eigenvalues $\lambda_1(A_N),\ldots,\lambda_k(A_N)$ and their eigenvectors $v_1,\ldots,v_k$ and these are the $k$ eigenvalues that will be shown to escape the bulk of $W_N$’s spectrum due to the perturbation. We will use the following standard probabilistic notations to describe convergence as $N\to\infty$:
\begin{itemize}
 \item   A sequence of events $E_N$ is said to occur with high probability (w.h.p.) if
    \[
        \mathbb{P}(E_N^c) = O\left(e^{-(\log N)^\eta}\right), \quad \text{for some } \eta > 1.
    \]
   In words, the probability that $E_N$ fails to occur decays super-polynomially fast as $N$ grows.

 \item   For random variables (or random sequences) $Y_N$ and $Z_N$, we write
    \[
        Y_N = O_{hp}(Z_N)
    \]
    to mean that there exists a constant $C>0$ such that $|Y_N| \le C |Z_N|$ w.h.p. (for all large $N$).\item Similarly,
   \[
        Y_N = o_{hp}(Z_N)\]
    means that for every $\delta>0$, $|Y_N| \le \delta|Z_N|$ w.h.p. as $N\to\infty$. Intuitively, $O_{hp}$ and $o_{hp}$ are high-probability versions of the usual $O(\cdot)$ and $o(\cdot)$ notations.

\item    Similarly, \( Y_N = O_p(Z_N) \) means
    \[
        \lim_{x \to \infty} \sup_{N \geq 1} \mathbb{P}(|Y_N| > x |Z_N|) = 0,
    \]
    and \( Y_N = o_p(Z_N) \) means
    \[
        \lim_{N \to \infty} \mathbb{P}(|Y_N| > \delta |Z_N|) = 0, \quad \text{for all } \delta > 0.
    \]
Note that if \( Z_N \neq 0 \) a.s., then \( Y_N = O_p(Z_N) \) is equivalent to the stochastic tightness of \((Z_N^{-1}Y_N : N \geq 1)\), while \( Y_N = o_p(Z_N) \) implies \( Z_N^{-1}Y_N \xrightarrow{P} 0 \). Further, \( O_{\mathrm{hp}} \) and \( o_{\mathrm{hp}} \) are stronger than their \( O_p \) and \( o_p \) counterparts.
\end{itemize}
Throughout the paper, we omit the dependence on $N$ in notation when there is no risk of confusion. For example, we often write $W$ and $A$ instead of $W_N$ and $A_N$. All asymptotic notation ($O(\cdot), o(\cdot)$, etc.) is with respect to $N \to \infty$ unless specified otherwise.

\subsection{Main Results.} \label{subsec:mainresults}

We first summarize the first-order (law of large numbers) behavior of the outlier eigenvalues. Since the bulk of the eigenvalues of $A_N$ lie around the scale $\|W_N\| = O(\sqrt{N})$, the perturbation causes $k$ outliers at the larger scale of order $\theta_N$ (we drop the subscript $N$ from $\theta$ subsequently). Indeed, under Assumptions \ref{assump:A1}-\ref{assump:A2} we will show that the outliers separate from the bulk for large $N$. More precisely, for each fixed $i \in \{1,\ldots,k\}$, the $i$-th largest eigenvalue $\lambda_i(A)$ is approximately $\theta \alpha_i$. The deviation of $\lambda_i(A)$ from $\theta \alpha_i$ vanishes with high probability as $N\to\infty$. In particular, one obtains
\[
\lambda_i(A) = \theta \alpha_i \left(1 + o_{hp}(1)\right),
\]
for each $1\le i\le k$. This shows that $\lambda_i(A)$ is eventually bounded away from 0 w.h.p., so dividing by $\lambda_i(A)$ in asymptotic formulas will be justified. 

We now turn to the fluctuations of the outliers. Our first main theorem is a Central Limit Theorem (CLT) for the outlier eigenvalues, which describes their asymptotic distribution after proper centering. Let $\mathbb{E}[\lambda_i(A)]$ denote the expected value of $\lambda_i(A)$ (with respect to the randomness of $W$). We will show that $\mathbb{E}[\lambda_i(A)]$ is close to $\theta \alpha_i$ as $N$ grows (more precisely, an expansion is given in Theorem \ref{Th.EigenValueMean} below). The CLT is stated in two parts: first an expansion of $\lambda_i(A)$ around its mean, and then a description of the joint Gaussian law in the limit.
\begin{theorem}\label{Th.EigenValueCLT} ({\bf Central Limit Theorem for Eigenvalues}) Under Assumptions \eqref{assump:A1} and \eqref{assump:A2}, for all $i$ such that $1 \leq i \leq k$, we have the expansion
\[
\lambda_i(A) = \mathbb E \left(\lambda_i(A)\right) + \frac{\theta \alpha_i}{\lambda_i(A)} e_i^\prime We_i + o_p(1)\,.
\]

As a consequence, the vector of fluctuations of the top $k$ eigenvalues converges in distribution to a multivariate Gaussian:
\[
\bigg (\lambda_i(A) - \mathbb E(\lambda_i(A)) : 1 \leq i \leq k\bigg ) \Rightarrow (G_i : 1 \leq i \leq k)\,,
\]
where $(G_1,\ldots,G_k)$ is a centered Gaussian $\mathbb R^k$-valued random vector with covariance
\[
\mathrm{Cov}(G_i,G_j) = 2 \int_0^1 \int_0^1 h_i(x) h_i(y) h_j(x) h_j(y) f(x,y)\,dx\,dy \,,
\]
for all $1 \le i,j \le k$. In particular, for each fixed $1\le i\le k$, $\lambda_i(A)-\E[\lambda_i(A)]$ is asymptotically $\mathcal{N}( 0, \sigma_i^2 )$ with
$$\sigma_i^2=2 \int_0^1 \int_0^1 h_i(x)^2 f(x, y) h_i(y)^2 d x d y.$$

\end{theorem}

\medskip
\begin{remark} In the expansion above, $\theta \alpha_i\lambda_i(A)^{-1} e_i' W e_i$ is the leading random term in $\lambda_i(A) - \E\lambda_i(A)$. Since $\lambda_i(A) \approx \theta \alpha_i$ for large $N$, this leading term is roughly $e_i' W e_i$ (up to a factor of order 1). Thus heuristically the fluctuation of $\lambda_i(A)$ is governed by the random quadratic form $e_i' W e_i$. These forms for different $i$ are correlated if the $h_i$ have overlapping support or $f(x,y)$ couples them, which results in the non-zero covariance $\Cov(G_i,G_j)$ when $i\neq j$. 
\end{remark}

The next result provides an asymptotic description of the centering $\mathbb{E}[\lambda_i(A)]$ of the outlier eigenvalues. In other words, it refines the first-order approximation $\mathbb{E}[\lambda_i(A)] \approx \theta \alpha_i$ by giving the leading correction term. Define a $k\times k$ deterministic matrix $B = B_N$ by
\begin{equation}\label{eq:Bmatrix}
B(j,l) = \sqrt{\alpha_j \alpha_{\ell}} \theta e_j^\prime e_{\ell} + \frac{\sqrt{\alpha_j \alpha_{\ell}} \mathbb E(e_j^\prime W^2 e_{\ell})}{\theta \alpha_i^2}\, \text{ for } 1\le j,\ell \le k.
\end{equation}

Intuitively, $B$ is like the “finite-rank part” of $\E(A)$ plus an $O(1/\theta_N)$ correction coming from the $W^2$ term.

For the next result a bit stronger assumption on $\theta_N$ than \ref{assump:A1} is needed.

\begin{theorem}\label{Th.EigenValueMean} ({\bf Asymptotic Behavior of Eigenvalue Mean}) Suppose the  Assumptions \ref{assump:A1} and \ref{assump:A2} hold.  Additionally, assume \begin{equation}\label{assump:A4}
N^{2/3} \ll \theta_N.
\end{equation} 
Then we have
\[
\mathbb E (\lambda_i(A)) = \lambda_i(B)\left(1 + O \left(\frac{N^2}{\theta^3}\right)\right)\,,
\]
where $B$ is as defined in \eqref{eq:Bmatrix} and $\lambda_i(B)$ is the $i$-th eigenvalue of $B$.
\end{theorem}

Since the proofs of Theorem \ref{Th.EigenValueCLT} and \ref{Th.EigenValueMean} are modifications of the proof from \cite{chakrabarty2020eigenvalues} we provide the details in the Appendix of this article.

We now summarize the main findings regarding the eigenvectors corresponding to the outlier eigenvalues. Fix $1 \le i \le k$, and let $v_i$ be the (unit-norm) eigenvector of $A$ associated with $\lambda_i(A)$. We are interested in the $N$-dimensional vector $v_i = (v_i(1), \ldots, v_i(N))'$ and, in particular, its projections onto the perturbation directions $e_1,\ldots,e_k$. Since ${e_1,\ldots,e_k}$ form the signal directions, an intuitive expectation is that $v_i$ should be mostly aligned with $e_i$ and approximately orthogonal to $e_j$ for $j\ne i$. The following results confirm this picture and quantify the fluctuations.

First, it turns out that $v_i$ is indeed asymptotically aligned with $e_i$ and asymptotically orthogonal to $e_j$ ($j\ne i$). In fact, one can show that
\[
e_i^{\prime} v_i=1+o_{hp}\left(\frac{\sqrt{N}}{\theta}\right), \quad e_j^{\prime} v_i=o_{hp}\left(\frac{\sqrt{N}}{\theta}\right) \quad(j \neq i)
\]

which means that $e_i$ captures the vast majority of the weight of $v_i$, up to a small fraction on the order of $(\sqrt{N}/\theta)$. Note that $\sqrt{N}/\theta \to 0$ by Assumption \ref{assump:A1}, so indeed $e_i'v_i \to 1$ and $e_j'v_i \to 0$ in probability.

Our next theorem describes the Gaussian fluctuations of the alignment of $v_i$ with $e_i$. In other words, we consider the inner product $e_i' v_i$ and show that it concentrates near $1$ and that its fluctuations around that value are Gaussian in the $N\to\infty$ limit. Moreover, for the inner products with non-matching directions ($e_j' v_i$ with $j\ne i$), we establish that they fluctuate around 0 and also follow an approximate Gaussian law. This can be viewed as a CLT for the coordinates of the eigenvectors in the basis of perturbation vectors ${e_i}$.

\begin{theorem}[{\bf Gaussian fluctuation of eigenvector alignment}]\label{Th.EigenVectorAlignment.CLT} Under Assumptions \ref{assump:A1}-\ref{assump:A3}, 
\begin{itemize}
    \item[(a)] for each $i\in{1,\ldots,k}$, we have
\[
\frac{\left(\theta \alpha_i\right)^2}{\sqrt{N}}\left(e_i^{\prime} v_i-\mathbb{E}\left[e_i^{\prime} v_i\right]\right) \xrightarrow{d} \mathcal{N}\left(0, \sigma^2\right)
\]

as $N\to\infty$, where
\[
\sigma_i^2=\frac{1}{2} \int_0^1 \int_0^1 h_i(x)^2 f(x, y) f(y, z) h_i(z)^2 d x d y d z
\]

\noindent
In particular, $e_i' v_i$ itself converges to 1 in probability, since $\E[e_i' v_i]\to 1$ and the fluctuations vanish as $N\to\infty$.  

\item[(b)] For any $j \ne i$,
\[
\theta \left(e_j^{\prime} v_i-\E[e_j^{\prime} v_i]\right) \xrightarrow{d} \mathcal{N}\left(0, \tau_{i j}^2\right), \]

for some $\tau_{ij}^2$ given by the following expression 
\begin{align*}
\tau_{ij}^2 = \frac{1}{(\alpha_i-\alpha_j)^2 \alpha_j}\int_{[0,1]^2} \left[h_j(x) h_i(y) + h_i(x) h_j(y)\right]^2 f(x,y)\, dx\, dy.
\end{align*}
 In particular, $e_j' v_i = o_p(1)$ for $j\ne i$, showing asymptotic orthogonality of $v_i$ and $e_j$.
\end{itemize}
\end{theorem}

In the above result the main fluctuation in $e_i'v_i$ comes from certain quadratic forms of $W$ (specifically $e_i' W^2 e_i$); we prove a martingale CLT for those quadratic forms to deduce the above result. For $e_j'v_i$ with $j\ne i$, the leading contribution comes from a linear form $e_j' W e_i$, which is a sum of approximately $N^2$ independent terms and hence by the Lindeberg CLT is asymptotically normal of order $N$. The $\theta$ scaling above normalizes it to $O(1)$.

We now state an eigenvector delocalization result. Delocalization means that an eigenvector has no significant mass on any single coordinate. In our context, $v_i$ is largely supported on the direction $e_i$, but we want to ensure it is not too concentrated even within $e_i$. The following theorem shows that all coordinates of $v_i$ are $O_{hp}(1/\sqrt{N})$, up to some logarithmic corrections. In particular, the $\ell^\infty$ norm of $v_i$ goes to zero as $N\to\infty$. This is similar in spirit to the delocalization of eigenvectors in an unperturbed Wigner matrix (where typically $\|v_i\|_\infty = O(N^{-1/2})$), but here we must account for a slower decay due to the perturbation (hence the logarithmic factor appears).

\begin{theorem}[{\bf Eigenvector delocalization}]\label{Decolalization}  Under Assumptions \ref{assump:A1}-\ref{assump:A3}, for each $1\le i\le k$, as $N\to\infty$,
\begin{align*}
\|v_i - e_i\|_{\infty} = O_{hp}\left(\frac{(\log N)^{\xi/4}}{\theta}\right) = o_{hp}\left(\frac{1}{\sqrt{N}}\right).
\end{align*}
\end{theorem}

Thus every coordinate $v_i(\ell)$ of the eigenvector satisfies $|v_i(\ell)| = O_{hp}(N^{-1/2})$, and $\max_{1\le \ell\le N}|v_i(\ell)| = o_{hp}(N^{-1/2})$. This shows that $v_i$ is delocalized across of order $N$ coordinates.

We now turn to a more functional perspective by considering each eigenvector $v_i$ as a
function on $[0,1]$. For $1 \le i \le k$, define the piecewise-constant function
$V_N^{(i)} : [0,1] \to \mathbb{R}$ by
\begin{equation}\label{eq:Processfn}
  V_N^{(i)}(t) =
  \begin{cases}
    \sqrt{N}\, v_i(\ell), & \text{if } \frac{\ell-1}{N} < t \le \frac{\ell}{N},\quad \ell = 1,\dots,N,\\[1ex]
    0, & t = 0.
  \end{cases}
\end{equation}
 This is a convenient way to interpret $v_i$ as a random element of $L^2[0,1]$. Since $v_i$ is a unit vector, $V_N^{(i)}$ is an $L^2$-normalized function (indeed, $\int_0^1 V_N^{(i)}(t)^2 dt = \sum_{\ell=1}^N v_i(\ell)^2=1$).  For large $N$, we expect $V_N^{(i)}(t)$ to concentrate near the function $h_i(t)$.

For $1 \le i \le k$ and $g \in C_0^\infty(0,1)$, we set
\begin{equation}\label{EVecproAction}
  \langle \mathcal V_N^{(i)}, g\rangle
  := \theta_N \int_0^1 \big( V_N^{(i)}(t) - \mathbb{E}[V_N^{(i)}(t)] \big) g(t)\,dt.
\end{equation}

 The scaling $\theta_N $ is chosen so that $\langle \mathcal V_N^{(i)}, g \rangle$ converges to a non-degenerate limit; it comes from the fact that $V_N^{(i)} - \E[V_N^{(i)}$ is of order $1/\theta_N$ in the $h_i$-direction (from Theorem \ref{Th.EigenVectorAlignment.CLT}). We will consider $\mathcal V_N^{(i)}$ as an element of the dual of the Sobolev space $H^1_0[0,1]$ (the space of $C_0^\infty$ functions with norm given by the Dirichlet energy, see below), which is a subspace of distributions on $[0,1]$. We call $\mathcal V_N^{(i)}$ the eigenvector process (fluctuation field) associated to the top eigenvector.

To describe the scaling limit of $\mathcal V_N^{(i)}$, it is useful to recall some facts about Sobolev spaces on $[0,1]$. Let $H^1_0[0,1]$ be the Sobolev space of (equivalence classes of) functions on $[0,1]$ that vanish at the boundary and have square-integrable first derivative. This is a Hilbert space with inner product $$\langle g_1,g_2\rangle_{H^1} = \int_0^1 g_1'(x) g_2'(x),dx.$$ Let $\{\phi_\ell(x)\}_{\ell\ge1}$ be an orthonormal basis of $L^2[0,1]$ consisting of eigenfunctions of the Dirichlet Laplacian $-d^2/dx^2$ on $[0,1]$. We have $\phi_\ell(0)=\phi_\ell(1)=0$, and the corresponding eigenvalues $0<\lambda_1 < \lambda_2 < \cdots$ tend to infinity as $\ell\to\infty$. It is well known that $\{\phi_\ell\}_{\ell\ge 1}$ is also an orthogonal basis for $H^1_0[0,1]$, with $\|\phi_\ell\|_{H^1}^2 = \lambda_{\ell}$. Moreover, $\{\lambda_\ell^{-1/2}\phi_\ell\}$ forms an orthonormal basis of $H^1_0[0,1]$. The dual space $H^{-1}[0,1]$ is defined to be the space of continuous linear functionals on $H^1_0[0,1]$ or, equivalently, the completion of $L^2[0,1]$ under the norm $\|h\|_{H^{-1}} = \sup_{|g|_{H^1}\le1} \langle h,\, g\rangle$. It can be characterized in terms of the basis $\{\phi_{\ell}\}_{\ell\ge 1}$: any $h\in H^{-1}$ can be expanded as $h = \sum_{\ell \ge 1} a_\ell \phi_\ell$ in the sense of distributions, and
\[
\|h\|_{H^{-1}}^2=\sum_{\ell \geq 1} \frac{a_{\ell}^2}{\lambda_{\ell}}.
\]

More generally, for any $c>0$, one can define $H^c_0[0,1]$ as the closure of $C_0^\infty(0,1)$ under the norm $\|g\|_{H^c}^2 := \sum_{\ell\ge1} a_\ell^2 \lambda_\ell^c$ (where $g = \sum a_\ell \phi_\ell$); its dual $H^{-c}$ has norm $\|h\|_{H^{-c}}^2 = \sum a_\ell^2 \lambda_{\ell}^{-c}$. Intuitively, larger $c$ means a smaller dual space (because $\lambda_\ell^{-c}$ decays faster), so convergence in $H^{-c}$ is a stronger mode of convergence for larger $c$.

Using this framework, we can now state the convergence of the eigenvector process $\mathcal V_N^{(i)}$. Roughly, $\mathcal V_N^{(i)}$ converges to a limit $\mathcal V^{(i)}$ in $H^{-d}$ for any $d>1/2$. The limit $\mathcal V^{(i)}$ is a random element of $H^{-d}$ which is Gaussian (a centered Gaussian field). In fact, $\mathcal V^{(i)}$ can be viewed as a variant of Gaussian white noise on $(0,1)$ with Dirichlet boundary conditions (zero average), as we explain after the theorem.

\begin{theorem}({\bf Functional CLT for eigenvector process})\label{EVecProMain} Fix $i \in \{1,\dots,k\}$ and let $\mathcal V_N^{(i)}$ be the eigenvector fluctuation field defined in \eqref{EVecproAction}. Under Assumptions \ref{assump:A1}, \ref{assump:A2} and \ref{assump:A3}, the sequence ${\mathcal V_N^{(i)}}$ is tight in $H^{-d}[0,1]$ for every $d > \frac12$. Moreover, $\mathcal V_N^{(i)}$ converges weakly to a random element $\mathcal V^{(i)}$ in $H^{-d}[0,1]$ (for any $d>1/2$). The limit $\mathcal V^{(i)}$ is a mean-zero Gaussian random distribution on $(0,1)$ characterized by its covariance: for any test functions $g_1, g_2 \in C_0^\infty(0,1)$,
\begin{equation}\label{Eq:covariance}
\mathbb{E}\left[\left\langle \mathcal V^{(i)}, g_1\right\rangle\left\langle \mathcal V^{(i)}, g_2\right\rangle\right]= C_k^{(i)}(g_1, g_2),
\end{equation}
where $C_k^{(i)}(g_1, g_2)$ is an explicit function given in terms $\{h_j, f, \alpha_j\}_{1\le j\le k}$ in \eqref{covarianceform}. 
\end{theorem}
We shall provide an expression for the general case but the case $i=1$ and $k=1$ seem to be more tractable and for $k\ge 1$, the expression is much more complicated.
Theorem \ref{EVecProMain} states that the fluctuations of the eigenvector $v_1$ (viewed as a function on $[0,1]$) converge to a Gaussian random distribution whose law is that of a mean-zero Gaussian process with covariance given by an explicit function. This process can be viewed as transformation of the white noise. As a sanity check, one can verify that in the simplest setting (say $k=1$ and a flat variance profile $f(x,y)\equiv 1$), the previous results imply that $V_N^{(1)}(t) \approx h_1(t)$ and $\mathcal V_N^{(1)}$ converges to a white noise with unit intensity on $(0,1)$, projected to the subspace of zero-mean distributions. Indeed, in that case $h_1(t)$ is constant since $e_1$ is proportional to the all-1’s vector, $e_1'v_1$ fluctuates like Gaussian random variable around its mean (Theorem \ref{Th.EigenVectorAlignment.CLT}) with $$\int h_1^2(x) f(x,y) f(y,z) h_1^2(z)dxdy dz = 1.$$
 Putting these together, one recovers that $\langle \mathcal V_N^{(1)}, g\rangle$ converges to $\mathcal{N}(0, C_1(g)^2)$, where $C_1(g)$ is given as below:
 \[
 C_1(g)^2=\int_0^1 g^2(x)\, dx - \left(\int_0^1 g(x)\,dx\right)^2.
\]

\begin{remark}[Identification of the Limit in the Homogeneous Case]
In the special case where \( f \equiv 1 \) and \( h \equiv 1 \), the limiting fluctuation field \( \mathcal{V}^{(1)}(t) \) admits a concrete representation as projected white noise:
\[
\mathcal{V}^{(1)}(t) = \xi(t) - \int_0^1 \xi(x)\,dx,
\]
where \( \xi(t) \) denotes standard white noise on \( [0,1] \). This is the projection of white noise onto the orthogonal complement of the constant function \( \mathbf{1} \in L^2[0,1] \). In other words, \( \mathcal{V}(t) \) is the unique generalized process satisfying
\[
\langle \mathcal{V}^{(1)}, g \rangle = \langle \xi, g \rangle - \left( \int_0^1 g(x)\,dx \right) \left( \int_0^1 \xi(x)\,dx \right), \quad \text{for all } g \in L^2[0,1].
\]

The covariance kernel of \( \mathcal{V}^{(1)} \) is given by
\[
\mathbb{E}[\mathcal{V}^{(1)}(t)\mathcal{V}^{(1)}(s)] = \delta(t-s) - 1,
\]
which reflects the fact that the total mass \( \int_0^1 \mathcal{V}^{(1)}(t)\,dt \equiv 0 \) almost surely. In particular, for any test function \( g \in L^2[0,1] \), the variance of the linear functional \( \langle \mathcal{V}^{(1)}, g \rangle \) is
\[
\mathrm{Var}(\langle \mathcal{V}^{(1)}, g \rangle) = \int_0^1 g^2(t)\,dt - \left( \int_0^1 g(t)\,dt \right)^2.
\]
 This shows that the eigenvector fluctuation field converges to white noise with the constant mode removed.
\end{remark}

The proof of Theorem \ref{EVecProMain} leverages known results about the convergence of random fields to Gaussian Free Field-like distributions. See for example \cite{camia2015planar} and \cite{sheffield2007gaussian} for related techniques.

\section{Eigenvalue Fluctuations – Proofs of Main Results}\label{section:eigenvalue fluctuations}

In this section, we present the proofs of the results concerning the outlier eigenvalues of $A_N$. We will repeatedly use the assumptions and notation introduced earlier. First, we establish some useful estimates about the base random matrix $W_N$. These estimates will be used to control error terms and justify approximations in the proofs of the theorems. The proofs of these lemmas are technical but not too complicated; we defer them to the Appendix. Throughout this section, we assume that Assumptions \ref{assump:A1} and \ref{assump:A2} hold. 

\begin{lemma}({\bf Norm bound for $W$})\label{lemma:NormBound} Let $W$ be the Wigner-type matrix defined in \eqref{eq:WignerTypeMatrix} satisfying \eqref{eq:momentcond} and with variance profile $f(x,y)$. Then \[
\|W\| = O_{hp}(\sqrt{N})\,.
\]
\end{lemma}
This is a standard result since $\E[W_N(j,\ell)^2] \le C$ uniformly in $j,\ell$, the semicircle law suggests $\|W_N\| \leq 2\sqrt{N C}$, and more refined arguments give the above concentration. A proof is provided in the Appendix for completeness.

Lemma \ref{lemma:NormBound} implies that the bulk of the eigenvalues of $A_N = W_N + (\text{finite-rank perturbations})$ lie in the interval $[-O(\sqrt{N}), O(\sqrt{N})]$ with high probability. In particular, any eigenvalues that are of order $\theta \gg \sqrt{N}$ must separate from the bulk and these will be the outliers we study.

The next two lemmas give bounds on certain quadratic forms involving $W$. They will be useful for controlling the higher-order terms in expansions. Let $$L := \lfloor \log N \rfloor.$$

\begin{lemma}({\bf Bound on moment of quadratic forms})\label{ExpectationBound} For any two $N\times 1$ deterministic vectors $u$ and $v$ whose entries are bounded by $N^{-1/2}$ in absolute value, there exists a constant $C_1>0$ such that for all integers $n$ with $2 \le n \le L$,
\[
|\mathbb E(u^\prime W^n v)| \le (C_1\sqrt{N})^n.
\]
\end{lemma}
\begin{lemma}({\bf Concentration of quadratic forms})\label{ConcentrationBound} For $u,v$ as above, there exists $\eta>1$ such that
\begin{align}
\max _{1 \le n \le L}\mathbb P\left(\left|u^\prime W^n v - \mathbb E(u^\prime W^n v)\right| > N^{\frac{n-1}{2}}(\log N)^{\frac{n\xi}{4}}\right) = O(e^{-(\log N)^\eta})
\end{align}
\end{lemma}
This is an adaptation of a result from \cite{erdHos2013spectral} to our setup with a variance profile.

Combining Lemmas \ref{ExpectationBound} and \ref{ConcentrationBound}, we see that for any fixed power $1\le n \le \log N$, $u' W^n v = \E(u'W^n v) + O_{hp}(N^{(n-1)/2} (\log N)^{n\xi/4})$. In particular, taking $n=1$, we have $u' W v = O_{hp}((\log N)^{\xi/4}))$ (since $\E(u' W v)=0$ for mean-zero entries in $W$). For $n=2$, $u' W^2 v = \E(u'W^2 v) + O_{hp}(\sqrt{N} (\log N)^{\xi/2}) = O_{hp}(N)$ since $\E(u' W^2 v) = O(N)$ by Lemma \ref{ExpectationBound}. These will be frequently used estimates.

We now prove the first-order behavior of the outliers and this is formalized in the next theorem.
\begin{theorem}\label{Th.probedge}\textbf{First-Order Behavior of Eigenvalues.} Under Assumptions \eqref{assump:A1} and \eqref{assump:A2}, we have for every $1 \leq i \leq k$,
\[
\lambda_i(A) = \theta \alpha_i \left(1 + o_{\mathrm{hp}}(1)\right)\,.
\]
\end{theorem}

\begin{proof}
 Observe that $\E[A_N] = \theta uu'$ where $u := [\sqrt{\alpha_1}e_1, \ldots, \sqrt{\alpha_k}e_k]'$ is an $N\times k$ matrix whose columns are the perturbation vectors scaled by $\sqrt{\alpha_i}$. The non-zero eigenvalues of $uu'$ coincide with those of the $k\times k$ matrix $u' u$, and by Assumption \ref{assump:A2} we have
\[
u^\prime u \to \mathrm{Diag} \left(\alpha_1,\alpha_2,\ldots,\alpha_k\right)
\]
as $N\to\infty$. Thus the eigenvalues of $\frac{1}{\theta}\E[A_N] = u u'$ converge to $\alpha_1,\ldots,\alpha_k$ (and the rest $N-k$ eigenvalues are 0). In particular, for large $N$, $\lambda_i(\E[A_N]) = \theta \alpha_i + o(\theta)$. Now, by the Weyl inequality, the difference between $\lambda_i(A_N)$ and $\lambda_i(\E[A_N])$ can be bounded by the operator norm of the noise matrix:
\begin{align*}
\left|\lambda_i(A_N)-\lambda_i\left(\mathbb E(A_N)\right)\right|\le\|W_N\|=O_{hp}\left(\sqrt{N}\right)\,,
\end{align*}
 by Lemma \ref{lemma:NormBound}. Therefore, for each fixed $\varepsilon>0$, with high probability we have
\[
\frac{\lambda_i\left(A_N\right)}{\theta\alpha_i} \in\left(1-\frac{C \sqrt{N}}{\theta \alpha_i}, 1+\frac{C \sqrt{N}}{\theta \alpha_i}\right),\]
for some constant $C>0$. As $N\to\infty$, the quantity $\frac{\sqrt{N}}{\theta} \to 0$ by Assumption \ref{assump:A1}, so the right-hand side interval shrinks to 1. Thus $\lambda_i(A_N)/\theta \alpha_i \to 1$ in probability, completing the proof. 
\end{proof}

Having established the location of the outliers to first order, we proceed to analyze their fluctuations. For notational convenience, in the remainder of this section we fix an index $i \in {1,\ldots, k}$ and focus on the corresponding eigenvalue $\lambda_i(A)$. 

Define $\mu := \lambda_i(A)$, and for brevity denote $\mu_0 := \E[\lambda_i(A)]$ which is a deterministic number that is close to $\theta \alpha_i$ by the previous theorem. We will eventually show that $\mu - \mu_0 = O_p(1)$. Let us also introduce a useful $k\times k$ matrix $K = K_N$ that depends on $W$ and $\mu$:
\begin{equation}\label{eq:VMatrix}
K(j,\ell) =
\begin{cases}
\displaystyle \theta\sqrt{\alpha_j \alpha_\ell} e_j' \Big(I - \frac{1}{\mu}W\Big)^{-1} e_\ell, &\text{if $\|W\| < \mu$},\\
0, &\text{otherwise}.
\end{cases}
\end{equation}
Note that $\|W\|<\mu$ holds with high probability for large $N$. $K$ is an effective matrix capturing the interaction between different signal directions $e_j$ and $e_\ell$ through the resolvent $(I - \frac{1}{\mu}W)^{-1}$. The diagonal entries $K(j,j)$ will turn out to be close to $\theta \alpha_j$, reflecting that the $j$-th outlier eigenvalue is near $\theta \alpha_j$.

Before proceeding, let us note an immediate consequence of Theorem \ref{Th.probedge} and Lemma \ref{lemma:NormBound}.

\begin{remark}\label{Bound.W/mu}
 From Theorem \ref{Th.probedge} we have $\mu = \lambda_i(A) \sim \theta \alpha_i$ for large $N$, with high probability. Meanwhile, Lemma \ref{lemma:NormBound} says $\|W\| = O_{hp}(\sqrt{N}) = o_{hp}(\theta)$ by Assumption \ref{assump:A1}. Combining these, we get
\begin{align*}
\left\|\frac{W}{\mu}\right\|=O_{hp}\left(\frac{\sqrt{N}}{\theta}\right) = o_{hp}(1)
\end{align*}
In particular, for large $N$, the operator $I - \frac{W}{\mu}$ is invertible with high probability, which justifies the expansion of the resolvent in a Neumann series.
\end{remark}

For the proof of Theorem \ref{Th.EigenValueCLT}, the strategy is to express $\mu = \lambda_i(A)$ in terms of the matrix $K$ defined above, and then to carefully analyze the fluctuations of $K$. Recall that $\lambda_i(A)$ is the $i$th largest eigenvalue of $A$. 

\begin{lemma}\label{EVa.combined}
As \( N \to \infty \), the matrix \( K \) satisfies
\[
K(j, l) = \theta \alpha_j \left( \one(j = l) + o_{hp}(1) \right), \quad 1 \le j, l \le k,
\]
and with high probability,
\[
\mu = \lambda_i(K).
\]
\end{lemma}

\begin{proof}
Fix \( 1 \le j, l \le k \). On the high-probability event \( \left\| \frac{W}{\mu} \right\| < 1 \), Remark~\ref{Bound.W/mu} implies
\[
\left\| \sum_{n \ge 1} \frac{W^n}{\mu^n} \right\| = O_{hp}\left( \frac{\|W\|}{\mu} \right) = o_{hp}(1),
\]
so that
\[
\left( I - \frac{W}{\mu} \right)^{-1} = I + o_{hp}(1).
\]
Assumption \eqref{assump:A2} gives \( \lim_{N \to \infty} e_j^\prime e_l = \one(j = l) \). Therefore,
\[
\frac{K(j, l)}{\theta} = \sqrt{\alpha_j \alpha_l} \, e_j^\prime \left(I - \frac{W}{\mu}\right)^{-1} e_l = \sqrt{\alpha_j \alpha_l} \left( \one_{\{j = l\}} + o_{hp}(1) \right),
\]
which simplifies to
\[
V(j, l) = \theta \alpha_j \left( \one(j = l) + o_{hp}(1) \right).
\]

For the second claim, the equality \( \mu = \lambda_i(V) \) holds with high probability by an argument identical to that of \cite[Lemma 5.2]{chakrabarty2020eigenvalues} using the Gershgorin Circle theorem. 
\end{proof}

\subsection{Proof of Theorem \ref{Th.EigenValueCLT}: a sketch}
Let us denote $Y_n$ matrix as follows:
\begin{equation}\label{eq:Y}
 Y_n(j, \ell):=\sqrt{\alpha_j \alpha_{\ell}} \theta e_j^{\prime} W^n e_{\ell}.
\end{equation}

We briefly sketch the main steps, deferring technical details to the Appendix \ref{appendixA}.

\paragraph{{\bf Step 1:   Resolvent expansion.}} By Theorem \ref{Th.probedge}, the outlier eigenvalue $\mu=\lambda_i(A)$ is of order $\theta$, while a norm bound on the noise matrix $W$ gives $\|W\| = O(\sqrt{N})$. Thus $\big\|\frac{W}{\mu}\big\| = o_{hp}(1)$, and we can invert $I - \frac{W}{\mu}$ as a Neumann series. This yields an infinite series representation
   \begin{equation}\label{K:inf}
   K=\sum_{n=0}^{\infty} \mu^{-n} Y_n
   \end{equation}

    for an appropriately defined $k\times k$ matrix $K$ (with entries involving $W^n$) such that $\mu = \lambda_i(K)$.

\paragraph{\bf Step 2: Truncation (Lemma \ref{Le.S1}).} We truncate the above series at $n=L=\lfloor \log N\rfloor$. Lemma \ref{Le.S1} shows that the tail contribution from $n>L$ is negligible, giving
    \[
    \mu=\lambda_i\left(\sum_{n=0}^L \mu^{-n} Y_n\right)+o_{h p}(1)
    \]

\paragraph{\bf Step 3: Replacing higher-order terms (Lemma \ref{Le.S2}).} Next we replace each random matrix $Y_n$ in the truncated sum for $2 \le n \le L$ by its expectation. A concentration bound (Lemma \ref{ConcentrationBound}) ensures
    \[
    \sum_{n=2}^L \mu^{-n}\left(Y_n-\E[Y_n]\right)=o_{h p}(1)
    \]

    which yields
  \[
  \mu=\lambda_i\left(Y_0+\frac{Y_1}{\mu}+\sum_{n=2}^L \mu^{-n} \E\left[Y_n\right]\right)+o_{h p}(1).
  \]
\paragraph{\bf Step 4: Deterministic approximation (Lemma \ref{Le.S3}).} We now introduce a deterministic approximation for $\mu$. Consider the equation
    \[
    x=\lambda_i\left(Y_0 + \sum_{n=2}^L x^{-n} \E\left[Y_n\right]\right),
    \]
    which (by Lemma \ref{Le.S3}) has a solution $x=\tilde{\mu}$ for large $N$. Moreover, $\tilde{\mu}/\theta \to \alpha_i$ as $N\to\infty$. Intuitively, $\tilde{\mu}$ is the eigenvalue obtained by replacing all random quantities in the truncated sum with their deterministic limits or expectations.

\paragraph{\bf Step 5: Bounding the difference (Lemma \ref{Le.S4}).} We then show that $\mu$ is close to $\tilde{\mu}$. In particular, Lemma \ref{Le.S4} gives
    \[
    \mu-\tilde{\mu}=O_{h p}\left(\frac{\left\|Y_1\right\|}{\theta}+1\right)
    \]
    implying that $|\mu-\tilde{\mu}|= O_{hp}((\log N)^{\xi/4})$.

\paragraph{\bf Step 6: Extracting the main fluctuation (Lemma \ref{Le.S5}).} We next isolate the leading random term. Lemma \ref{Le.S5} shows that
   \[
   \mu=\bar{\mu}+\frac{1}{\mu} Y_1(i, i)+o_{h p}\left(\frac{\left\|Y_1\right\|}{\theta}+1\right),
   \]

    for some deterministic $\bar{\mu}$ in the vicinity of $\tilde{\mu}$ (in fact, $\bar{\mu}=\lambda_i(X)$ for a certain $k\times k$ matrix $X$ built from $\E[Y_n]$, and one can check $\bar{\mu}/\theta \to \alpha_i$). Here $Y_1(i,i) = \alpha_i\theta e_i' W e_i$ is the key random component driving the fluctuation of $\mu$.

\paragraph{\bf Step 7: Final expression.} Finally, since $\bar{\mu}$ differs from $\mu_0 := \E[\mu]$ by an $o(1)$ term (one can show $\E[\mu] = \bar{\mu} + o(1)$ as $N\to\infty$), we may replace $\bar{\mu}$ with $\mu_0$ in the previous equation without affecting the $o_p(1)$ error. This yields
   \begin{equation}
    \mu-\mu_0=\frac{1}{\mu} Y_1(i, i)+o_p(1)
    \end{equation}
Also what we have an useful bound which will be frequently used throughout the article.
\begin{equation}\label{eq:mu-mu0}
    \mu-\mu_0=O_{hp}\left(\frac{\|Y_1\|}{\theta}+1\right).
    \end{equation}
    Recalling that $\mu = \lambda_i(A)$, $Y_1(i,i)=\alpha_i\theta e_i' W e_i$, and $\mu_0=\E[\lambda_i(A)]$, we obtain
   \[
   \lambda_i(A)=\E\left[\lambda_i(A)\right]+\frac{\theta \alpha_i}{\lambda_i(A)} e_i^{\prime} W e_i+o_p(1),
   \]

    as claimed in Theorem \ref{Th.EigenValueCLT}.
For the joint CLT, we rely on the asymptotic normality of the family $(e_i' W e_i)_{i=1}^k$. First, by a direct covariance computation one finds that for any $1\le i,j \le k$,
\[
\mathrm{Cov}\left(e_i^\prime We_i,e_j^\prime We_j\right)= \frac{2}{N^2}\sum_{1\le a, b\le N}h_i\left(\frac {a}{N}\right)h_i\left(\frac {b}{N}\right)h_j\left(\frac {a}{N}\right)h_j\left(\frac {b}{N}\right)f\left(\frac{a}{N},\frac{b}{N}\right)
\]

which converges, as $N\to\infty$, to
\[
2\int_0^1\int_0^1h_i(x)h_i(y)h_j(x)h_j(y)f(x,y)\,dx\,dy
\]

Moreover, since each $e_i' W e_i$ is a (normalized) sum of many independent random variables (essentially a sum of $W$’s entries on the $i$th row/column), the multivariate Lindeberg–Lévy CLT applies. Hence as $N\to\infty$,
\[
\left(e_1^{\prime} W e_1, e_2^{\prime} W e_2, \ldots, e_k^{\prime} W e_k\right) \Rightarrow\left(G_1, G_2, \ldots, G_k\right),
\]
where $(G_1,\ldots,G_k)$ is a mean-zero Gaussian vector with the above covariance structure. Finally, by Theorem \ref{Th.probedge} we know $\lambda_i(A)$ is w.h.p. nonzero (and of order $\theta$), and Theorem \ref{Th.EigenValueCLT} provides the linear approximation linking $\lambda_i(A) - \E[\lambda_i(A)]$ to $e_i' W e_i$. By Slutsky’s theorem, we can therefore deduce that
\[
\left(\lambda_1(A)-\E\left[\lambda_1(A)\right], \ldots, \lambda_k(A)-\E\left[\lambda_k(A)\right]\right) \Rightarrow\left(G_1, \ldots, G_k\right),
\]
as required for Theorem \ref{Th.EigenValueCLT}.

\section{Technical Estimates and Series representations for the Eigenvector}\label{section:evseries}

In this section, we provide some technical estimates which are crucial for the eigenvector results stated in Section~\ref{subsec:mainresults}, namely
Theorem~\ref{Th.EigenVectorAlignment.CLT} (Gaussian fluctuation of the aligned component),
and Theorem \ref{Decolalization} (eigenvector
delocalization). Throughout, we assume that Assumptions~\ref{assump:A1}--\ref{assump:A3} hold. We
begin with two auxiliary lemmas on the eigenvector alignment and on a bound for powers of $W$.

\noindent\textbf{Proof Outline of Theorem \ref{Th.EigenVectorAlignment.CLT}.} We begin by expressing the deviation $e_i'v_i - \E[e_i'v_i]$ in terms of the random matrix $W$.  Using the eigenvector perturbation expansion (as in the proof of Theorem \ref{Th.probedge}), one shows that the leading random term is proportional to the quadratic form $e_i'W^2e_i - \E[e_i'W^2e_i]$.  All other terms are of smaller order under Assumption \eqref{assump:A1}.  Since $e_i'W^2e_i$ is a sum of weakly dependent terms, a martingale central limit theorem (or concentration lemmas) imply that $(e_i'W^2e_i - \E[e_i'W^2e_i])/\sqrt{N}$ converges to a Gaussian law.  This yields the asymptotic normality of $(\theta\alpha_i)^2(e_i'v_i - \E[e_i'v_i])/\sqrt{N}$.  For $j\neq i$, one similarly expresses $e_j'v_i - \E[e_j'v_i]$ in terms of the linear form $e_j'We_i$ plus negligible error; then the classical Lindeberg CLT gives the Gaussian limit for $\theta(e_j'v_i-\E[e_j'v_i])$.  The detailed argument is given below. 

\subsection{Preliminary lemmas for the proof of Theorem \ref{Th.EigenVectorAlignment.CLT}}
We begin with a few essential estimates from the eigenvalue analysis. 
\begin{remark}\label{Error.negpow,mumu0}
Following the line of the proof of Lemma \ref{Le.S4}, one can see that for any fixed $m \ge 1$,
\begin{equation}
\mu^{-m} - \mu_1^{-m} = O_{hp}\left(\frac{|\mu-\mu_1|}{\theta^{m+1}}\right)
\end{equation}
for any $\mu_1$ such that $\mu_1 \sim \theta$ (in particular, this holds for $\mu_1=\bar{\mu}$ or $\mu_1=\E[\mu]$).
\end{remark}
\begin{remark}
Similar to Remark \ref{Error.negpow,mumu0}, denoting $\mathbb E(\mu)$ by $\mu_0$, Lemma \ref{Le.S4} along with \eqref{lim.mu0-mubar} imply  
\begin{equation}\label{Error.mu-mu0}
\mu - \mu_0 = O_{hp}\left(\frac{\|Y_1\|}{\theta} + 1\right).   
\end{equation}
\end{remark}
\begin{lemma}\label{EigenVector.1}
With Assumptions \ref{assump:A1}-\ref{assump:A3} in place, as $N \to \infty$ we have 
\[
e_j^\prime\left(I - \frac{W}{\mu}\right)^{-m}e_l = \one(j =l) + o_{hp}(\frac{\sqrt{N}}{\theta}), \text{ for } m\in \{1,2\}.
\]
\end{lemma}

\begin{proof} Following the arguments in the proof of Theorem \ref{Th.EigenValueCLT} , we expand the resolvent (on the high probability event) as
$$\Bigl(I - \frac{W}{\mu}\Bigr)^{-1} = I + \frac{W}{\mu} + \frac{W^2}{\mu^2} + \sum_{n=3}^L \frac{W^n}{\mu^n} + \sum_{n=L}^\infty \frac{W^n}{\mu^n}.$$
By arguments leading up to \eqref{eq:mu-mu0}, we have for fixed $1 \le j,l \le k$, this yields the expansion
\begin{equation}\label{BasicExpansion}
e_j^\prime\Bigl(I-\frac{W}{\mu}\Bigr)^{-1}e_l = e_j^\prime e_l + \frac{e_j^\prime W e_l}{\mu} + \frac{e_j^\prime W^2 e_l}{\mu^2} + \sum_{n=3}^L \frac{\mathbb E\bigl(e_j^\prime W^n e_l\bigr)}{\mu_0^n} + o_{hp}\Bigl(\frac{N|\mu - \mu_0|}{\theta^3} + \frac{\sqrt{N}}{\theta^2}\Bigr).
\end{equation}
Here we have used the fact that $\mu = \mu_0 + o_{hp}(1)$. 

Next, by Lemma \ref{ConcentrationBound} and Lemma \ref{ExpectationBound} along with Assumption \ref{assump:A3}, we can bound each term in \eqref{BasicExpansion}.  In particular, using these bounds and that $e_j^\prime e_l = \one(j=l) + O(1/N)$, we obtain
\begin{align}
&\label{BasicApprox1}e_j^\prime\left(I-\frac{W}{\mu}\right)^{-1}e_l \\ & = e_j
^\prime e_l + O_{hp}\left(\frac{(\log N)^{\xi/4}}{\theta}\right) + O
_{hp}\left(\frac{N}{\theta^2}\right) + O\left(\frac{N^{3/2}}{\theta^3}\right) + o_{hp}\left[ \frac{N|\mu - \mu_0|}{\theta^3} + \frac{\sqrt{N}}{\theta^2}\right]\\
\nonumber & = \one(j =l) + O\left(\frac{1}{N}\right) + o_{hp}\left(\frac{\sqrt{N}}{\theta}\right) + o_{hp}\left[ \frac{N}{\theta^3}(\log N)^{\xi/4} + \frac{\sqrt{N}}{\theta^2}\right]\\
&\label{BasicApprox2} = \one(j =l) + o_{hp}\left(\frac{\sqrt{N}}{\theta}\right)
\end{align}
Hence
\[
e_j^\prime\Bigl(I-\frac{W}{\mu}\Bigr)^{-1}e_l = \one(j=l) + o_{hp}\Bigl(\frac{\sqrt{N}}{\theta}\Bigr),
\]
as claimed.  The case $m=2$ follows by a similar argument. 
\end{proof}

In what follows we define 
\begin{equation}\label{def:s_m}
S_m^{ij} = \sum_{n = m}^L \frac{\mathbb E(e_i^\prime W^n e_j)}{\mu_0^n} \text{ and } \tilde S_{m}^{ij} = \sum_{n = m}^L \frac{\mathbb E(e_i^\prime W^n e_j)(n+1)}{\mu_0^n}
\end{equation}

\begin{lemma}\label{EigenVector.2}

Under assumptions \ref{assump:A1}-\ref{assump:A3}, as $N \to \infty$ we have for $j,l \neq i$
\begin{align*}
\Bigl[e_j^\prime\Bigl(I - \frac{W}{\mu}\Bigr)^{-1}e_i\Bigr]\Bigl[e_l^\prime\Bigl(I - \frac{W}{\mu}\Bigr)^{-1}e_i\Bigr]
&= \frac{\mathbb E(e_j^\prime W^2 e_i)\,\mathbb E(e_l^\prime W^2 e_i)}{\mu_0^4} + \frac{\mathbb E(e_j^\prime W^2 e_i)}{\mu_0^2}\,S_3^{li}
\\
&+ \frac{\mathbb E(e_l^\prime W^2 e_i)}{\mu_0^2}\,S_3^{ji} + S_3^{ji}\,S_3^{li} 
+ o_{hp}\left( \frac{\sqrt{N}}{\theta^2}\right)
\end{align*}
where $S_3^{ji}$ and $S_3^{li}$ are as defined in \eqref{def:s_m}.
\end{lemma}
\begin{proof}
Using \eqref{BasicExpansion} and \eqref{eq:mu-mu0}, we obtain
\begin{align}
\nonumber &\left[e_j^\prime\left(I - \frac{W}{\mu}\right)^{-1}e_i\right]\left[e_l^\prime\left(I - \frac{W}{\mu}\right)^{-1}e_i\right] \\
\nonumber &\quad= \frac{e_j^\prime W^2 e_i \cdot e_l^\prime W^2 e_i}{\mu^4} + \frac{e_j^\prime W^2 e_i}{\mu^2} S_3^{li} + \frac{e_l^\prime W^2 e_i}{\mu^2} S_3^{ji}+ S_3^{ji}S_3^{li} \\
&\qquad\quad+o_{hp}\left(\frac{N|\mu - \mu_0|}{\theta^3} + \frac{\sqrt{N}}{\theta^2}\right) + o_{hp}\left[\frac{N|\mu - \mu_0|}{\theta^3} + \frac{\sqrt{N}}{\theta^2}\right]^2.\label{eq:mid_expansion}
\end{align}

Using Lemma \ref{ConcentrationBound} and Remark \ref{Error.negpow,mumu0}, we further estimate the first term on the right-hand side of \eqref{eq:mid_expansion} as
\begin{align}
\frac{e_j^\prime W^2 e_i \cdot e_l^\prime W^2 e_i}{\mu^4} &= \frac{\mathbb E(e_j^\prime W^2 e_i) \mathbb E(e_l^\prime W^2 e_i)}{\mu_0^4} + O_{hp}\left(\frac{N\sqrt{N}(\log N)^{\xi/2}}{\theta^4}\right) + O_{hp}\left(\frac{N^2|\mu - \mu_0|}{\theta^5}\right) \\
&= \frac{\mathbb E(e_j^\prime W^2 e_i) \mathbb E(e_l^\prime W^2 e_i)}{\mu_0^4} + o_{hp}\left(\frac{N}{\theta^3}\left(1 + \frac{\|Y_1\|}{\theta}\right)\right).\label{eq:first_term_estimate}
\end{align}

Similarly, we have for the second term,
\begin{align}
\frac{e_j^\prime W^2 e_i}{\mu^2} S_3^{li} &= \frac{\mathbb E(e_j^\prime W^2 e_i)}{\mu_0^2} S_3^{li} + O_{hp}\left(\frac{N\sqrt{N}(\log N)^{\xi/2}}{\theta^4}\right) + O_{hp}\left(\frac{N^2|\mu - \mu_0|}{\theta^5}\right) \\
&= \frac{\mathbb E(e_j^\prime W^2 e_i)}{\mu_0^2} S_3^{li} + o_{hp}\left(\frac{N}{\theta^3}\left(1 + \frac{\|Y_1\|}{\theta}\right)\right).\label{eq:second_term_estimate}
\end{align}

An analogous calculation applies to the third term $\frac{e_l^\prime W^2 e_i}{\mu^2} S_3^{ji}$, yielding
\begin{equation}\label{eq:third_term_estimate}
\frac{e_l^\prime W^2 e_i}{\mu^2} S_3^{ji} = \frac{\mathbb E(e_l^\prime W^2 e_i)}{\mu_0^2} S_3^{ji} + o_{hp}\left(\frac{N}{\theta^3}\left(1 + \frac{\|Y_1\|}{\theta}\right)\right).
\end{equation}

Substituting the estimates \eqref{eq:first_term_estimate}, \eqref{eq:second_term_estimate}, and \eqref{eq:third_term_estimate} into \eqref{eq:mid_expansion}, we obtain
\begin{align*}
& \left[e_j^\prime\left(I - \frac{W}{\mu}\right)^{-1}e_i\right]\left[e_l^\prime\left(I - \frac{W}{\mu}\right)^{-1}e_i\right]  \\
&\quad = \frac{\mathbb E(e_j^\prime W^2 e_i) \mathbb E(e_l^\prime W^2 e_i)}{\mu_0^4} + \frac{\mathbb E(e_j^\prime W^2 e_i)}{\mu_0^2} S_3^{li} + \frac{\mathbb E(e_l^\prime W^2 e_i)}{\mu_0^2} S_3^{ji} + S_3^{ji}S_3^{li} + o_{hp}\left( \frac{\sqrt{N}}{\theta^2}\right).
\end{align*}
This completes the proof.
\end{proof}

\begin{lemma}\label{EigenVector.3}
Under Assumptions \ref{assump:A1}-\ref{assump:A3}, as $N \to \infty$, for $j,l,p \ne i$, we have
\begin{align*}
\left[e_j^\prime\left(I - \frac{W}{\mu}\right)^{-1}e_i\right]
\left[e_l^\prime\left(I - \frac{W}{\mu}\right)^{-1}e_i\right]
\left[e_p^\prime\left(I - \frac{W}{\mu}\right)^{-2}e_p\right] 
= D_{lpp}S_2^{ji} + o_{hp}\left( \frac{\sqrt{N}}{\theta^2}\right),
\end{align*}
where
\begin{align*}
D_{lpp} &= \frac{3\,\mathbb E(e_j^\prime W^2 e_i) \mathbb E(e_l^\prime W^2 e_i)}{\mu_0^4}
+ \frac{\mathbb E(e_l^\prime W^2 e_i)}{\mu_0^2}\tilde S_3^{pp}
+ \frac{3\,\mathbb E(e_p^\prime W^2 e_p)}{\mu_0^2}S_3^{li} \\
&\quad + e_l^\prime e_i e_p^\prime e_p + e_p^\prime e_p S_2^{li} + S_3^{li}\tilde S_3^{pp}.
\end{align*}
\end{lemma}

\begin{proof}
Using \eqref{BasicExpansion} we have
\begin{align}
&\left[e_l^\prime\left(I - \frac{W}{\mu}\right)^{-1}e_i\right]\left[e_p^\prime\left(I - \frac{W}{\mu}\right)^{-2}e_p\right] \nonumber \\
&= \frac{e_l^\prime W e_i}{\mu}(e_p^\prime e_p)
+ \frac{e_l^\prime W^2 e_i}{\mu^2}(e_p^\prime e_p)
+ (e_l^\prime e_i)(e_p^\prime e_p)
+ \frac{3(e_l^\prime W^2 e_i)(e_p^\prime W^2 e_p)}{\mu^4} \\ &\quad + \frac{e_l^\prime W^2 e_i}{\mu^2}\tilde S_3^{pp}
+ (e_p^\prime e_p) S_3^{li}
+ \frac{3(e_p^\prime W^2 e_p)}{\mu^2}S_3^{li}
+ S_3^{li}\tilde S_3^{pp} \nonumber\\
&\quad + o_{hp}\left(\frac{N|\mu - \mu_0|}{\theta^3} + \frac{\sqrt{N}}{\theta^2}\right)
+ o_{hp}\left(\frac{N|\mu - \mu_0|}{\theta^3} + \frac{\sqrt{N}}{\theta^2}\right)^2.\label{eq:eigenv3-1}
\end{align}

 Using Lemma \ref{ConcentrationBound} and \ref{eq:mu-mu0} we simplify:
\begin{align}
&\left[e_l^\prime\left(I - \frac{W}{\mu}\right)^{-1}e_i\right]\left[e_p^\prime\left(I - \frac{W}{\mu}\right)^{-2}e_p\right] \nonumber\\
&= D_{lpp} + \frac{e_l^\prime W e_i}{\mu}(e_p^\prime e_p)
+ o_{hp}\left(\left(1 + \frac{\|Y_1\|}{\theta} + \frac{\|Y_1\|^2}{\theta^2}\right)\frac{N}{\theta^3} + \frac{\sqrt{N}}{\theta^2}\right) \nonumber\\
&\quad + O_{hp}\left(\frac{N(1+\frac{\|Y_1\|}{\theta})}{\theta^3}\right)
+ O_{hp}\left(\frac{\sqrt{N}(\log N)^{\xi/2}}{\theta^2}\right),\label{EigenVector.3.intermd}
\end{align}
where
\[
D_{lpp} = O\left(\frac{N}{\theta^2}\right).
\]

Multiplying by the third factor and using Lemma \ref{EigenVector.1} again,
\begin{align*}
&\left[e_j^\prime\left(I - \frac{W}{\mu}\right)^{-1}e_i\right]\left[e_l^\prime\left(I - \frac{W}{\mu}\right)^{-1}e_i\right]\left[e_p^\prime\left(I - \frac{W}{\mu}\right)^{-2}e_p\right] \\
&= D_{lpp}\left[e_j^\prime e_i + \frac{e_j^\prime W e_i}{\mu} + \frac{e_j^\prime W^2 e_i}{\mu^2} + \sum_{n=3}^L \frac{\mathbb E(e_j^\prime W^n e_i)}{\mu_0^n} + o_{hp}\left(\frac{N|\mu-\mu_0|}{\theta^3}+\frac{\sqrt{N}}{\theta^2}\right)\right]\\
&\quad + o_{hp}\left(\left(1 + \frac{\|Y_1\|}{\theta} + \frac{\|Y_1\|^2}{\theta^2}\right)\frac{N}{\theta^3}+\frac{\sqrt{N}}{\theta^2}\right).
\end{align*}
Since $e_j^\prime e_i=o(1/N)$ for $j\neq i$, this simplifies to:
\begin{align*}
&= D_{lpp}S_2^{ji} + o_{hp}\left( \frac{\sqrt{N}}{\theta^2}\right).
\end{align*}
This completes the proof.
\end{proof}

\begin{remark}\label{Remark:djlm}
It is clear that for $j,l,m \ne i$, our estimate in \eqref{EigenVector.3.intermd} simplifies as  
\begin{align}
 & \left[e_j^\prime\left(I - \frac{W}{\mu}\right)^{-1}e_i\right]\left[e_l^\prime\left(I - \frac{W}{\mu}\right)^{-2}e_m\right]\nonumber \\
 = &  D_{jlm} + o_{hp}\left(\left(1 + \frac{\|Y_1\|}{\theta} + \frac{\|Y_1\|^2}{\theta^2}\right)\frac{N}{\theta^3} + \frac{\sqrt{N}}{\theta^2}\right) + O_{hp}\left(\left(1 + \frac{\|Y_1\|}{\theta} \right)\frac{1}{\theta^3}\right) + O_{hp}\left(\frac{\sqrt{N}(\log N)^{\xi/2}}{\theta^3}\right) \nonumber\\
&= D_{jlm} + o_{hp}\left( \frac{\sqrt{N}}{\theta^2}\right).\label{EigenVector.3.exten} 
\end{align}
\end{remark}

\begin{lemma}\label{EVec.pivot1}
Let \(1 \leq i \leq k\), and let \(v_i\) denote the normalized eigenvector corresponding to the eigenvalue \(\lambda_i(A)\). Then, as \(N \to \infty\), we have the expansion
\begin{equation}\label{finding}
(e_i^\prime v_i)^{-2} 
= \theta^2\alpha_i^2\, e_i^\prime (\mu I - W)^{-2}e_i \;+\; \Gamma_i \;+\; \Upsilon_i \;+\; \mathcal{E}_N,
\end{equation}
where
\begin{align*}
\Gamma_i &:= 2\theta^3 \mu_0^{-3} \sum_{\substack{1 \le l \le k \\ l \ne i}} (\alpha_l)^{1/2}(\alpha_i)^{3/2}\, \left[Y^{-1} d_{li}\right](l), \\
\Upsilon_i &:= \theta^4 \mu_0^{-4} \sum_{\substack{1 \le l,m \le k \\ l \ne i,\, m \ne i}} (\alpha_l \alpha_m)^{1/2} \alpha_i\, \left[Y^{-1} \tilde d_{lm}\right](l)\, \left[Y^{-1} d_{lm}\right](m), \\
\text{ and, } 
\mathcal{E}_N &:= o_{hp}\left(\frac{\sqrt{N}}{\theta^2}\right)
\end{align*}
denotes an error term which is negligible with high probability.

Here, \(Y\) is a deterministic matrix converging to 
\[
\mathrm{Diag}\left(1-\frac{\alpha_1}{\alpha_i},\ldots,1-\frac{\alpha_{i-1}}{\alpha_i},1-\frac{\alpha_{i+1}}{\alpha_i},\ldots,1-\frac{\alpha_k}{\alpha_i}\right),
\]
and the vectors \(d_{jl}, \tilde d_{jl}\) are deterministic with each component of order \(O\left(\frac{N}{\theta^2}\right)\).
\end{lemma}

\begin{proof}
Let us fix $N\ge1$ and a sample point for which $\|W\|<\mu$, which is a high probability event, as affirmed by Remark~\ref{Bound.W/mu}. The proof proceeds in the following steps.

\textbf{Step 1:}
We begin with the eigenvector equation:
\begin{equation}\label{proof.l1.eq2}
\mu v_i = A v_i = W v_i + \theta \sum_{l=1}^k \alpha_l (e_l^\prime v_i) e_l.
\end{equation}
Since $\|W\|<\mu$, the matrix $\mu I - W$ is invertible, giving
\begin{equation}\label{eq.master1}
v_i = \theta \sum_{l=1}^k \alpha_l (e_l^\prime v_i)(\mu I - W)^{-1} e_l.
\end{equation}

\textbf{Step 2:}
Define the vector
\begin{equation}\label{V.EVec}
u := [\sqrt{\alpha_1}(e_1^\prime v_i), \ldots, \sqrt{\alpha_k}(e_k^\prime v_i)]^\prime.
\end{equation}
Premultiplying \eqref{eq.master1} with $\sqrt{\alpha_j} \mu e_j^\prime$ for $1 \le j \le k$, we find that
\begin{equation}\label{proof.l1.eqnew2}
V u = \mu u,
\end{equation}
where $V(j,l) := \theta \sqrt{\alpha_j \alpha_l} e_j^\prime (\mu I - W)^{-1} e_l$.

\textbf{Step 3:}
Using \eqref{eq.master1}, and normalization of \(v_i\), we expand
\begin{align}
(e_i^\prime v_i)^{-2} &= \theta^2 \alpha_i^2\, e_i^\prime (\mu I - W)^{-2} e_i \label{eq.master2} \\
&\quad + \theta^2 \sum_{\substack{l=i \text{ or } m=i \\ (l,m)\ne(i,i)}} \alpha_l \alpha_m \frac{(e_l^\prime v_i)(e_m^\prime v_i)}{(e_i^\prime v_i)^2} e_l^\prime (\mu I - W)^{-2} e_m \nonumber \\
&\quad + \theta^2 \sum_{\substack{l,m \ne i}} \alpha_l \alpha_m \frac{(e_l^\prime v_i)(e_m^\prime v_i)}{(e_i^\prime v_i)^2} e_l^\prime (\mu I - W)^{-2} e_m. \nonumber
\end{align}

\textbf{Step 4:}
Let $\tilde u$ be the column vector obtained by removing the $i$-th entry from the vector $u$ defined in \eqref{V.EVec}. Let $\tilde V_i$ denote the $i$-th column of $V$ with the $i$-th entry removed, and let $\tilde V$ be the $(k-1) \times (k-1)$ matrix obtained by deleting the $i$-th row and $i$-th column from $V$. Then, equation \eqref{proof.l1.eqnew2} implies that
\begin{equation}\label{Vtruncated1}
\mu \tilde u = \tilde V \tilde u + u(i)\tilde V_i, \qquad \text{w.h.p.}
\end{equation}

By Lemma \ref{EigenVector.1}
\[
\left\| I_{k-1} - \mu^{-1}\tilde V - D \right\| = o_{hp}\left( \frac{\sqrt{N}}{\theta} \right),
\]
where $D = \mathrm{Diag}\left(1 - \frac{\alpha_1}{\alpha_i}, \ldots, 1 - \frac{\alpha_{i-1}}{\alpha_i}, 1 - \frac{\alpha_{i+1}}{\alpha_i}, \ldots, 1 - \frac{\alpha_k}{\alpha_i} \right)$.

Since $D$ is invertible under our ordering assumption $\alpha_1 > \alpha_2 > \cdots > \alpha_k > 0$, it follows that $I_{k-1} - \mu^{-1}\tilde V$ is invertible with high probability. Thus, we may rewrite \eqref{Vtruncated1} as
\begin{equation}\label{Vtruncated2}
\tilde u = u(i)\mu^{-1}\left(I_{k-1} - \mu^{-1} \tilde V\right)^{-1}\tilde V_i, \qquad \text{w.h.p.}
\end{equation}

Substituting \eqref{Vtruncated2} into \eqref{V.EVec}, we find that for any $j \ne i$,
\begin{equation}\label{EigenSp.dim1}
e_j^\prime v_i = \frac{\sqrt{\alpha_i}}{\sqrt{\alpha_j}} e_i^\prime v_i \mu^{-1} \left[ \left(I_{k-1} - \mu^{-1} \tilde V \right)^{-1} \tilde V_i \right](j).
\end{equation}

\textbf{Step 5:}

Define $E := I_{k-1} - \mu^{-1}\tilde V$. The entries of $E$ can be explicitly written in block form as:
\begin{align*}
E_{lm} = 
\begin{cases}
1 - \frac{\theta}{\mu} \alpha_m\, e_m^\prime\left(I - \frac{W}{\mu}\right)^{-1}e_m & \text{if } l = m, \\
- \frac{\theta}{\mu} \sqrt{\alpha_l \alpha_m}\, e_l^\prime\left(I - \frac{W}{\mu}\right)^{-1}e_m & \text{if } l \ne m,
\end{cases}
\end{align*}
with appropriate index shifts for $l,m \ge i$ due to the deletion of the $i$-th row and column.

Each entry $E_{lm}$ can be approximated as:
\begin{align}
E_{lm} 
&= \delta_{lm} - \frac{\theta}{\mu_0} \sqrt{\alpha_l \alpha_m}\, e_l^\prime (\mu I - W)^{-1} e_m \notag \\
&= \delta_{lm} - \frac{\theta}{\mu_0} \sqrt{\alpha_l \alpha_m}(S_2^{lm} + e_l^\prime e_m) + Z_{lm}, \label{Det.Approx1}
\end{align}
where
\[
Z_{lm} = O_{hp}\left( \frac{N|\mu - \mu_0|}{\theta^3} + \frac{\sqrt{N}}{\theta^2} + \frac{\sqrt{N}(\log N)^{\xi/2}}{\theta^2} + \frac{|\mu - \mu_0|}{\theta} + \frac{(\log N)^{\xi/4}}{\theta} \right).
\]

Let $Y$ denote the deterministic matrix with entries
\begin{equation}\label{eq:Ymatrix}
Y_{lm} := \delta_{lm} - \frac{\theta}{\mu_0} \sqrt{\alpha_l \alpha_m} (S_2^{lm} + e_l^\prime e_m).
\end{equation}
Then $E = Y + Z$, and moreover,
\begin{equation}\label{Det.Approx2}
\|Y - D\| = O\left( \frac{N}{\theta^2} \right).
\end{equation}

Let $B_\delta$ be the closed ball (in operator norm) of radius $\delta > 0$ centered at $D$, such that every matrix in $B_\delta$ is invertible. Define the constants
\begin{equation}\label{InverseBound}
C_1 := \sup_{F \in B_\delta} \|F^{-1}\| < \infty, \quad
C_2 := \sup_{\substack{F_1, F_2 \in B_\delta}} \frac{\|F_1^{-1} - F_2^{-1}\|}{\|F_1 - F_2\|} < \infty.
\end{equation}
Then for large $N$, both $Y$ and $I_{k-1} - \mu^{-1}\tilde V$ lie in $B_\delta$ with high probability, and
\begin{equation}
\| (Y + Z)^{-1} - Y^{-1} \| \le C_2 \|Z\|.
\end{equation}
This establishes the deterministic control required to analyze the subsequent terms in equation \eqref{eq.master2}.

\textbf{Step 7:}

Since this analysis occurs on a high-probability event, we define \( \bar V_i := \tilde V_i/\theta \). Using Lemma \ref{EigenVector.1} and the expansion in \eqref{Det.Approx1}, we obtain for \( j \ne i \),
\begin{align}
& \left\| \left[ (I_{k-1} - \mu^{-1} \tilde V)^{-1} \bar V_i - Y^{-1} \bar V_i \right] \cdot e_j^\prime \left(I - \frac{W}{\mu}\right)^{-2} e_i \right\| \notag \\
& \le \left\| (Y + Z)^{-1} - Y^{-1} \right\| \cdot \|\bar V_i\| \cdot \left| e_j^\prime \left(I - \frac{W}{\mu} \right)^{-2} e_i \right| \notag \\
& = o_{hp} \left( \left(1 + \frac{\|Y_1\|}{\theta} \right) \frac{N}{\theta^3} + \frac{\sqrt{N}}{\theta^2} \right)=:\widetilde{\mathcal{E}_N}, \label{EVec.KeyApprox0}
\end{align}
where the final bound follows from \eqref{eq:mu-mu0}.

Thus, for \( j \ne i \), using \eqref{EVec.KeyApprox0} and \eqref{EigenSp.dim1} we deduce for the first term in the series \eqref{eq.master2}:
\begin{align}
\theta^2 \frac{e_j^\prime v_i}{e_i^\prime v_i} \cdot e_j^\prime (\mu I - W)^{-2} e_i 
= \frac{\sqrt{\alpha_i}}{\sqrt{\alpha_j}}\theta^3 \mu^{-3} (Y^{-1} \bar V_i)(j) \cdot e_j^\prime \left(I - \frac{W}{\mu}\right)^{-2} e_i + \widetilde{\mathcal{E}_N}. \label{EVec.KeyApprox1}
\end{align}

\textbf{Step 8:}

Applying the above analysis again for \( l,m \ne i \), we get:
\begin{align}
& \theta^2 \frac{e_l^\prime v_i}{e_i^\prime v_i} \cdot \frac{e_m^\prime v_i}{e_i^\prime v_i} \cdot e_l^\prime (\mu I - W)^{-2} e_m \notag \\
= & \frac{\alpha_i}{\sqrt{\alpha_l}\sqrt{\alpha_m}}\theta^4 \mu^{-4} (Y^{-1} \bar V_i)(l) (Y^{-1} \bar V_i)(m) \cdot e_l^\prime \left(I - \frac{W}{\mu}\right)^{-2} e_m + \widetilde{\mathcal{E}_N}. \label{EVec.KeyApprox2}
\end{align}

\textbf{Step 9:}

Let \( d_{ji} \in \mathbb{R}^{k-1} \) denote the deterministic vector defined via the entries \( D_{mji} \) from \eqref{EigenVector.3.intermd}. Using \eqref{Det.Approx2} and \eqref{EigenVector.3.exten}, we have \( \|Y^{-1}\| = O(1) \), and Lemma \ref{EigenVector.3} gives
\begin{align}
\left\| Y^{-1} \left( \bar V_i e_j^\prime (\mu I - W)^{-2} e_i - d_{ji} \right) \right\| 
= o_{hp} \left( \left(1 + \frac{\|Y_1\|}{\theta} + \frac{\|Y_1\|^2}{\theta^2} \right) \frac{N}{\theta^3} + \frac{\sqrt{N}}{\theta^2} \right).
\end{align}

Thus, \eqref{EVec.KeyApprox1} becomes
\begin{align}
\theta^2 \frac{e_j^\prime v_i}{e_i^\prime v_i} \cdot e_j^\prime (\mu I - W)^{-2} e_i 
= \frac{\sqrt{\alpha_i}}{\sqrt{\alpha_j}}\theta^3 \mu_0^{-3} (Y^{-1} d_{ji})(j) + \mathcal{E}_N. \label{EVec.KeyApprox3}
\end{align}

\textbf{Step 10:}

Define \( \tilde d_{lm} \in \mathbb{R}^{k-1} \) with $(\tilde d_{lm})_{j}= D_{jlm} S_2^{j^\prime i}$ for $j\neq i$ as in Remark \ref{Remark:djlm}. Arguing as before, we find
\begin{align}
\left\| Y^{-1}(\bar V_i - \tilde d_{lm}) Y^{-1} d_{lm} \right\| = \mathcal{E}_N.
\end{align}
Hence, from \eqref{EVec.KeyApprox2} we conclude:
\begin{align}
\theta^2 \frac{e_l^\prime v_i}{e_i^\prime v_i} \cdot \frac{e_m^\prime v_i}{e_i^\prime v_i} \cdot e_l^\prime (\mu I - W)^{-2} e_m 
=\frac{\alpha_i}{\sqrt{\alpha_l}\sqrt{\alpha_m}} \theta^4 \mu_0^{-4} (Y^{-1} \tilde d_{lm})(l) (Y^{-1} d_{lm})(m) + \mathcal{E}_N. \label{EVec.KeyApprox4}
\end{align}

Hence substituting the approximations \eqref{EVec.KeyApprox3} and \eqref{EVec.KeyApprox4} into equation \eqref{eq.master2}, we conclude the desired expansion:
\[
(e_i^\prime v_i)^{-2} = \theta^2 \alpha_i^2 e_i^\prime (\mu I - W)^{-2} e_i + \Gamma_i + \Upsilon_i + \mathcal{E}_N,
\]
as stated in Lemma \ref{EVec.pivot1}.
\end{proof}

\begin{lemma}\label{Align.bounds}
Let $v_i$ be the normalized eigenvector corresponding to $\lambda_i(A)$\,. Then as $N\to \infty$,
\begin{align}
\label{Align.Bound.along}e_i^\prime v_i = 1 + o_{hp}\bigg(\frac{\sqrt{N}}{\theta}\bigg)\,,
\end{align}and for $j \ne i$
\begin{align}
\label{Align.Bound.ortho}e_j^\prime v_i = o_{hp}\bigg(\frac{\sqrt{N}}{\theta}\bigg)\,.
\end{align}
\end{lemma}

The proof of Lemma \ref{Align.bounds} is a straightforward computation and the details are provided in the Appendix \ref{appendix:B}.
\begin{lemma}\label{EVec.pivot2} 
Let \(1 \leq i \leq k\), and let \(v_i\) denote the normalized eigenvector corresponding to \(\lambda_i(A)\). Then, as \(N \to \infty\), we have
\begin{align}
(e_i^\prime v_i)^{-2} - (1 + \tilde \Delta_i) 
=\, & \mathcal A_i\cdot\left(e_i^\prime W^2 e_i - \mathbb{E}(e_i^\prime W^2 e_i)\right)
 + \mathcal{E}_N,
\label{EVec.pivot2.target}
\end{align}
where:
\begin{itemize}
    \item[(a)] \(\tilde \Delta_i = O\left(\frac{N}{\theta^2}\right)\) is a deterministic correction term,
    \item[(b)] \(\mathcal{E}_N = o_{hp}\left(\frac{\sqrt{N}}{\theta^2}\right)\) is a high-probability negligible error term.
    \item[(c)] $\mathcal A_i=e_i^\prime e_i\, \mu_0^{-4}\theta\alpha_i(2\theta\alpha_i - \mu_0)$.
\end{itemize}
\end{lemma}

\begin{proof}[Proof of Lemma \ref{EVec.pivot2}]
From Lemma \ref{EVec.pivot1}, we know
\begin{equation}\label{eq:pivot1_summary}
(e_i'v_i)^{-2} = \theta^2\alpha_i^2 e_i'(\mu I - W)^{-2} e_i + \Gamma_i + \Upsilon_i + \mathcal{E}_N,
\end{equation}
where $\Gamma_i = O\left(\frac{N}{\theta^2}\right)$, $\Upsilon_i = O\left(\frac{N}{\theta^2}\right)$, and $\mathcal{E}_N = o_{hp}\left(\frac{\sqrt{N}}{\theta^2}\right)$.
Define the deterministic correction term
\[ \Delta_i := \Gamma_i + \Upsilon_i = O\left(\frac{N}{\theta^2}\right). \]
Subtracting $\Delta_i$ and $\mathcal{E}_N$ from both sides of \eqref{eq:pivot1_summary}, we obtain
\begin{equation}\label{eq:pivot1_reduced}
(e_i'v_i)^{-2} - \Delta_i - \mathcal{E}_N = \theta^2\alpha_i^2 e_i'(\mu I - W)^{-2} e_i.
\end{equation}

\begin{align*}
e_i'(\mu I - W)^{-2}e_i &= \mu^{-2} e_i' (I - \mu^{-1} W)^{-2} e_i \\
&= \mu^{-2} \sum_{n=0}^\infty (n+1) \mu^{-n} e_i' W^n e_i \\
&= \mu^{-2} \left( e_i'e_i + 2 \mu^{-1} e_i'We_i + 3 \mu^{-2} e_i'W^2e_i + \sum_{n=3}^\infty (n+1) \mu^{-n} e_i'W^n e_i \right).
\end{align*}

Group the expansion into three parts:
Using the moment expansion and techniques developed in Lemmas \ref{Le.S1}, \ref{Le.S2}, and \ref{Le.S4}, we expand:
\begin{align}
\nonumber (e_i'v_i)^{-2} - \Delta_i - \mathcal{E}_N  - e_i'e_i = &\mu^{-2}e_i'e_i\left((\theta\alpha_i)^2 - \mu^2\right) +  2\mu^{-3}(\theta\alpha_i)^2 e_i'We_i \\
\label{EVec.pivot2.interm1}+ & 3\mu^{-4}(\theta\alpha_i)^2 e_i'W^2e_i + (\theta\alpha_i)^2 \sum_{n=3}^L \frac{\mathbb E(e_i' W^n e_i)(n+1)}{\mu_0^{n+2}} \\
\nonumber & + o_{hp}\left(\frac{N|\mu - \mu_0|}{\theta^3} + \frac{\sqrt{N}}{\theta^2}\right)
\end{align}

Premultiplying \eqref{eq.master1} by $e_i'$ and seperating the terms involving $i$ and $\ell\neq i$ we get
\begin{align}
\nonumber \mu - \theta\alpha_i &= \theta\alpha_i(e_i'e_i - 1) + \theta\alpha_i\left(\frac{e_i'We_i}{\mu} + \frac{e_i'W^2e_i}{\mu^2} + \sum_{n=3}^{\infty} \frac{e_i'W^n e_i}{\mu^n} \right) \\
\label{EVec.pivot2.interm2} & + \theta \sum_{l \ne i} \alpha_l \frac{e_l'v_i}{e_i'v_i} e_i' \left(I - \frac{W}{\mu}\right)^{-1} e_l.
\end{align}

By Assumption \ref{assump:A3}
\begin{equation}
\label{Lipchitz.ei}
e_i' e_i = 1 + O\left(\frac{1}{N}\right).
\end{equation}

Multiplying \eqref{EVec.pivot2.interm2} by $\mu + \theta\alpha_i$ and plugging in \eqref{Lipchitz.ei}, we deduce:
\begin{align}
\nonumber & (e_i'v_i)^{-2} - \Delta_i - \mathcal{E}_N - e_i'e_i - (\theta\alpha_i)^2 \sum_{n=3}^L \frac{\mathbb E(e_i' W^n e_i)(n+1)}{\mu_0^{n+2}} \\
\nonumber = & \mu^{-2}e_i'e_i(\theta\alpha_i + \mu)\theta\alpha_i(1 - e_i'e_i) \\
&+ \mu^{-3}\theta\alpha_i\left(2\theta\alpha_i  - (\mu + \theta\alpha_i) e_i'e_i \right) e_i'We_i \label{EVec.pivot2.interm2.3}\\
\nonumber &+ \mu^{-4}\theta\alpha_i\left(3\theta\alpha_i - (\mu + \theta\alpha_i) e_i'e_i \right) e_i'W^2e_i \\
&- \mu^{-2}e_i'e_i(\mu + \theta\alpha_i)\theta\alpha_i\sum_{n=3}^\infty \frac{e_i'W^n e_i}{\mu^n}\label{EVec.pivot2.interm2.1}\\
 & \label{EVec.pivot2.interm2.2} - \mu^{-2}e_i'e_i(\mu + \theta\alpha_i)\theta \sum_{l \ne i} \alpha_l \frac{e_l'v_i}{e_i'v_i} e_i'\left(I - \frac{W}{\mu}\right)^{-1} e_l + o_{hp}\left(\frac{N|\mu - \mu_0|}{\theta^3} + \frac{\sqrt{N}}{\theta^2}\right)
\end{align}

For \eqref{EVec.pivot2.interm2.1} using approximations from \eqref{EVec.pivot2.interm1}, we have:
\begin{align*}
\nonumber &\mu^{-2}e_i'e_i(\mu + \theta\alpha_i)\theta\alpha_i \sum_{n=3}^\infty \frac{e_i'W^n e_i}{\mu^n} \\
= & \mu_0^{-2}e_i'e_i(\mu_0 + \theta\alpha_i)\theta\alpha_i \sum_{n=3}^L \frac{\mathbb E(e_i'W^n e_i)}{\mu_0^n} \\
 &+ o_{hp}\left(\frac{N|\mu - \mu_0|}{\theta^3} + \frac{\sqrt{N}}{\theta^2}\right)
\end{align*}

Now consider the term in \eqref{EVec.pivot2.interm2.2}. Let $\bar{D}_{jl}$ denote the deterministic term in Lemma \ref{EigenVector.2}, and define $\bar{d}_l$ as the vector of these quantities for $j \ne i$. Then by \eqref{EVec.KeyApprox3}, for $l \ne i$:
\begin{align}
 \mu^{-2}e_i'e_i(\mu + \theta\alpha_i)\theta \alpha_l \frac{e_l'v_i}{e_i'v_i} e_i'\left(I - \frac{W}{\mu}\right)^{-1}e_l 
\label{EVec.pivot2.interm3} = \mu_0^{-3}e_i'e_i(\mu_0 + \theta\alpha_i)\theta^2\sqrt{\alpha_i \alpha_l} [Y^{-1}\bar{d}_l](l) + \mathcal E_N
\end{align}

Consider the term in \eqref{EVec.pivot2.interm2.3}. By Lemma \ref{ConcentrationBound} we have
\begin{align}
\mu^{-3}(\theta\alpha_i)^2 O\left(\frac{1}{N}\right)e_i'We_i
= O_{hp}\left(\frac{(\log N)^{\xi/4}}{\theta^2}\right)
= o_{hp}\left(\frac{\sqrt{N}}{\theta^2}\right).
\end{align}

Equation \eqref{EVec.pivot2.interm2.2}, along with \eqref{Align.Bound.ortho}, implies
\begin{align}
\label{EVec.pivot2.interm2.1'}
\mu - \theta\alpha_i
= O\left(\frac{\theta}{N}\right) + O_{hp}\left((\log N)^{\xi/4}\right) + o_{hp}\left(\frac{N}{\theta}\right).
\end{align}

Thus,
\begin{align}
\label{EVec.pivot2.interm5}
\mu^{-3}\theta\alpha_i(\theta\alpha_i -\mu)\, e_i'e_i\,e_i'We_i
&= O_{hp}\left(\frac{(\log N)^{\xi/4}}{\theta^2}\right)
+ O_{hp}\left(\frac{(\log N)^{\xi/2}}{\theta^2}\right)
+ o_{hp}\left(\frac{N}{\theta^3}\right) \\
&= o_{hp}\left(\frac{\sqrt{N}}{\theta^2}\right).
\end{align}

Similarly, by Lemma \ref{lemma:NormBound},
\begin{align}
\mu^{-4}(\theta\alpha_i)^2 O\left(\frac{1}{N}\right)e_i'W^2 e_i
= O_{hp}\left(\frac{1}{\theta^2}\right)
= o_{hp}\left(\frac{\sqrt{N}}{\theta^2}\right).
\end{align}

Using these estimates and absorbing the error terms into $\mathcal E_N$, remaining deterministic terms by $\Delta_i = O\left(\frac{N}{\theta^2}\right)$ and using Remark \ref{Error.negpow,mumu0}, we have
\begin{align}
\nonumber
(e_i'v_i)^{-2} - \Delta_i - \mathcal{E}_N - e_i'e_i
&= \mu_0^{-2}e_i'e_i(\theta\alpha_i + \mu_0)\theta\alpha_i(1 - e_i'e_i)
+ \mu_0^{-4}\theta\alpha_i(2\theta\alpha_i - \mu_0)\, e_i'e_i\, e_i'W^2e_i \label{EVec.pivot2.interm4'} \\
&\quad + o_{hp}\left(\frac{\sqrt{N}}{\theta^2}\right).
\end{align}

Subtracting \(\mu_0^{-4}\theta\alpha_i(2\theta\alpha_i - \mu_0)\, e_i'e_i\, \mathbb{E}(e_i'W^2e_i)\) from both sides and absorbing the error terms into \( Er \) again, we get
\begin{align}
\nonumber
&(e_i'v_i)^{-2} - \Delta_i - \mathcal{E}_N - e_i'e_i
- \mu_0^{-2}e_i'e_i(\theta\alpha_i + \mu_0)\theta\alpha_i(1 - e_i'e_i) \\
&\quad - \mu_0^{-4}\theta\alpha_i(2\theta\alpha_i - \mu_0)\, e_i'e_i\, \mathbb{E}(e_i'W^2e_i) \\
&= \mu_0^{-4}\theta\alpha_i(2\theta\alpha_i - \mu_0)\, e_i'e_i\left(e_i'W^2e_i - \mathbb{E}(e_i'W^2e_i)\right).
\end{align}

Using \eqref{Lipchitz.ei} and absorbing the remaining deterministic terms in $\Delta_i$, thus implying $\Delta_i=O\left(\frac{N}{\theta^2}\right)$ we have by defining $\tilde \Delta_i=\Delta_i+1$ the required lemma.
\end{proof}

\begin{lemma}\label{EVec.pivot3}
Let \(1 \le i \le k\). Then, as \(N \to \infty\), the eigenvector mass satisfies the refined approximation
\[
e_i^\prime v_i-(1+\zeta)
= \frac{1}{2} (1 + \tilde\Delta_i)^{-3/2} \cdot \mathcal A_i\cdot \left(e_i^\prime W^2 e_i - \mathbb{E}(e_i^\prime W^2 e_i)\right) + \mathcal{E}_N,
\]
where $$\mathcal A_i= \mu_0^{-4} \theta \alpha_i (2\theta \alpha_i - \mu_0) e_i^\prime e_i$$
and \(\zeta\) and \(\Delta_i\) are \(O\left(\frac{N}{\theta^2}\right)\) and \(\mathcal{E}_N = o_{hp}\left(\frac{\sqrt{N}}{\theta^2}\right)\).

\end{lemma}

\begin{proof}

We start from the expansion derived in Lemma~\ref{EVec.pivot2}, which gives
\begin{align}
\nonumber (e_i^\prime v_i)^{-2} - (1 + \tilde\Delta_i) = \ & \mathcal A_i\cdot (e_i^\prime W^2 e_i - \mathbb{E}(e_i^\prime W^2 e_i)) +\mathcal E_N
\end{align}
where \( \tilde\Delta_i = O\left(\frac{N}{\theta^2}\right) \).

\medskip

Let \( g(x) = x^{-1/2} \). Then using the second-order Taylor expansion around \( 1 + \tilde\Delta_i \), we get:
\begin{align*}
(e_i^\prime v_i) - (1 + \tilde\Delta_i)^{-1/2} =  ((e_i^\prime v_i)^{-2} - (1 + \tilde\Delta_i)) g^\prime(1 +\tilde\Delta_i) 
 + \frac{1}{2} ((e_i^\prime v_i)^{-2} - (1 + \tilde\Delta_i))^2 g^{\prime\prime}(r),
\end{align*}
where \( r \) lies between \( (e_i^\prime v_i)^{-2} \) and \( 1 + \tilde\Delta_i \). Since from Lemma~\ref{EVec.pivot2} we have
\[ (e_i^\prime v_i)^{-2} - (1 + \tilde\Delta_i) = o_{hp}(1), \]
we conclude that \( r = 1 + \tilde\Delta_i + o_{hp}(1) \), hence \(g^{\prime\prime}(r) = O_{hp}(1)\). Moreover, \( g^\prime(1 + \tilde\Delta_i) = -\tfrac{1}{2}(1 + \tilde\Delta_i)^{-3/2} = O(1) \). Also recall that we fix the sign of $v_i$, that is, $e_i^\prime v_i>0$. Thus,
\begin{align}
(e_i^\prime v_i) - (1 + \tilde\Delta_i)^{-1/2}&= ((e_i^\prime v_i)^{-2} - (1 + \tilde\Delta_i)) g^\prime(1 + \tilde\Delta_i)\nonumber\\
&=\mathcal A_i\cdot(e_i^\prime W^2 e_i - \mathbb{E}(e_i^\prime W^2 e_i)) 
\label{EVec.pivot4.1} g^\prime(1 + \tilde\Delta_i) + \mathcal{E}_N.
\end{align}

We now control the second-order term:
\begin{align}
& ((e_i^\prime v_i)^{-2} - (1 + \tilde\Delta_i))^2 \cdot \frac{g^{\prime\prime}(r)}{2} \nonumber\\
=\ & O_{hp}\left(\left[ e_i^\prime e_i\, \mu_0^{-4} \theta \alpha_i (2\theta \alpha_i - \mu_0)(e_i^\prime W^2 e_i - \mathbb{E}(e_i^\prime W^2 e_i)) \right]^2\right) + \text{lower order terms} \nonumber\\
=\ & o_{hp}\left(\left(1 + \frac{\|Y_1\|}{\theta} + \frac{\|Y_1\|^2}{\theta^2}\right)\frac{N}{\theta^3} + \frac{\sqrt{N}}{\theta^2}\right).\label{EVec.pivot4.2}
\end{align}
Thus the remainder is negligible.
Since for $\tilde\Delta_i$ small, we define $\zeta$ which satisfies the following:
\[
(1+ \tilde\Delta_i)^{-1/2} = 1 + \zeta = 1 + O\left(\frac{N}{\theta^2}\right) 
\]the proof follows in conjunction with \eqref{EVec.pivot4.1} and \eqref{EVec.pivot4.2}.
\end{proof}

\section{Eigenvector alignment: Proof of Theorem \ref{Th.EigenVectorAlignment.CLT}}\label{section:alignment}
\begin{proof}[Proof of Theorem \ref{Th.EigenVectorAlignment.CLT} (a)]
From Lemma \ref{EVec.pivot3}, we have the expansion
 \begin{align}
 (e_i' v_i)-(1+\zeta)=\frac{1}{2} (1 + \tilde\Delta_i)^{-3/2} \cdot \mathcal A_i\cdot \left( e_i' W^2 e_i - \mathbb{E}[e_i' W^2 e_i] \right) + \mathcal{E}_N.\label{eq:evecalign.1}
 \end{align}

Let us define the centered quantity
\[
\widetilde{R} := (e_i' v_i) - (1 + \zeta) - \frac{1}{2} (1 + \tilde\Delta_i)^{-3/2} \cdot\mathcal{A}_i \left( e_i' W^2 e_i - \mathbb{E}[e_i' W^2 e_i] \right).
\]
Then
\[
\mathbb{E}[\widetilde{R}] = \mathbb{E}[e_i' v_i] - (1 + \zeta).
\]
Using \eqref{eq:evecalign.1}, for any $\delta > 0$, there exists $\eta > 1$ such that
\[
\mathbb{E}[|\widetilde{R}|] \le \delta \left( \left(1 + \frac{\mathbb{E}\|Y_1\|}{\theta} + \frac{\mathbb{E}\|Y_1\|^2}{\theta^2} \right) \frac{N}{\theta^3} + \frac{\sqrt{N}}{\theta^2} \right) + \sqrt{\mathbb{E}[\widetilde{R}^2]} \cdot \mathcal{O}(e^{-(\log N)^\eta}).
\]
Using the fact that $\mathrm{Var}(e'W^2e)= O(N)$, the second term is negligible, and so we conclude
\begin{equation}
\mathbb{E}[e_i' v_i] = 1 + \zeta + o\left( \frac{\sqrt{N}}{\theta^2} \right).
\end{equation}

Substituting back into \eqref{eq:evecalign.1} gives
\begin{align}
(e_i' v_i) - \mathbb{E}[e_i' v_i]
= \frac{1}{2} (1 + \tilde\Delta_i)^{-3/2}\mathcal{A}_i \cdot \left( e_i' W^2 e_i - \mathbb{E}[e_i' W^2 e_i] \right) + \mathcal{E}_N. \label{eq:evecalign.2}
\end{align}


Finally, from Theorem \ref{Th.MartingaleCLT}, we know that
\[
\frac{e_i' W^2 e_i - \mathbb{E}[e_i' W^2 e_i]}{\sqrt{N}} \Rightarrow \mathcal{N}(0, \sigma_i^2),
\]
where
\[
\tilde\sigma_i^2 := 2 \int_0^1 \int_0^1 \int_0^1 h_i^2(x) f(x,y) f(y,z) h_i^2(z) \, dx\,dy\,dz.
\]
Recall that $\frac{1}{2} (1 + \tilde\Delta_i)^{-3/2}= 1+ \mathrm{O}(\frac{N}{\theta^2})$. Using the fact that $\mu_0/\theta\to\alpha_i$ and multiplying both sides by $(\theta \alpha_i)^2/\sqrt{N}$ gives the final result:
\[
\frac{(\theta \alpha_i)^2}{\sqrt{N}} \left( (e_i' v_i) - \mathbb{E}[e_i' v_i] \right) \Rightarrow \mathcal{N}(0, \frac{1}{4}\tilde\sigma_i^2),
\]

\end{proof}

\begin{proof}[Proof of Theorem \ref{Th.EigenVectorAlignment.CLT} (b)]
From equation~\eqref{Vtruncated2}, we recall the key approximation:
\begin{equation}\label{Vtruncated3}
\left(I_{k-1} - \frac{1}{\mu} \widetilde{V} \right)\widetilde{u} = \frac{u(i)}{\mu} \widetilde{V}_i, \qquad \text{w.h.p.}
\end{equation}
Using $u(i)= e_i' v_i$, and $\bar{V}=\tilde{V}/\theta$ and replacing $u(i)$ and $\mu$ respectively by $\E[u(i)]$ (using Theorem \ref{Th.EigenVectorAlignment.CLT}(a))  and $\mu_0$, we incur an error which we gather in a term $\mathcal F_N$:
\begin{align}
\label{Evec.Ortho.S1}
\left(I_{k-1} - \frac{1}{\mu} \widetilde{V} \right)\widetilde{u}
= \mathbb{E}[u(i)]\,\frac{\theta}{\mu_0} \overline{V}_i + \mathcal{F}_N,
\end{align}
where the total error is bounded by
\[
\mathcal{F}_N = o_{hp} \left( \frac{\sqrt{N}(\log N)^{\xi/2}}{\theta^2} \right).
\]

Next, from Lemma~\ref{Det.Approx1} and the definition of $Y$ matrix in \eqref{eq:Ymatrix}, we approximate:
\begin{align}
Y \widetilde{u}
= \mathbb{E}[u(i)]\,\frac{\theta}{\mu_0} \overline{V}_i + \mathcal{F}_N.
\end{align}

Hence, for \(j \neq i\), we extract the \(j^{\text{th}}\) coordinate:
\begin{align}
\label{Evec.Ortho.S2}
(Y\widetilde{u})(j) = \mathbb{E}[u(i)]\,\frac{\theta}{\mu_0} \overline{V}_i(j) + \mathcal{F}_N.
\end{align}

Now observe that the left-hand side of \eqref{Evec.Ortho.S2} is given by
\[
(Y\widetilde{u})(j) = \left[1 - \frac{\theta}{\mu_0}\alpha_j\left(S_2^{jj}+e_j^\prime e_j\right)\right]\sqrt{\alpha_j}e_j^\prime v_i + \sum_{l \ne j,i}\left[- \frac{\theta}{\mu_0}\sqrt{\alpha_l\alpha_j}(S_2^{lj} + e_l^\prime e_j)\right]\sqrt{\alpha_l}e_l^\prime v_i
\]

and the right-hand side of \eqref{Evec.Ortho.S2} is given by
\[
\mathbb{E}[u(i)]\,\frac{\theta}{\mu_0} \overline{V}_i(j)=\mathbb E[u(i)]\mu_0^{-1}\theta\sqrt{\alpha_j\alpha_i}e_j^\prime\left(I-\frac{W}{\mu}\right)^{-1}e_i.
\]

From the structure of \(Y\), and using Lemma~\ref{Det.Approx1} (step 7), we compute:
\begin{align}
\label{Evec.Ortho.S3}\sum_{l \ne j,i}Y_{lj}\sqrt{\alpha_l}e_l^\prime v_i =  \sum_{l \ne j,i}Y_{lj}\mathbb E[u(i)]\mu_0^{-1}\theta \left[Y^{-1}\bar V_i\right](l) + o_{hp}\left(\frac{\sqrt{N}(\log N)^{\xi/2}}{\theta^2}\right)
\end{align}

where the matrix entries
\[
Y_{lj} := \delta_{jl} - \frac{\theta}{\mu_0} \sqrt{\alpha_l \alpha_j} (S_2^{lj} + e_l^\prime e_j)
\]

Now, from Lemmas~\ref{EigenVector.1} and~\ref{ConcentrationBound}, and the approximation~\eqref{Error.negpow,mumu0}, we have:
\begin{align}
\label{Evec.Ortho.S4}
\theta e_j^\prime \left(I - \frac{1}{\mu} W\right)^{-1} e_i 
= \theta S_2^{ji} + \theta e_j^\prime e_i + \frac{\theta}{\mu_0} e_j^\prime W e_i +  o_{hp}\left(\frac{\sqrt{N} (\log N)^{\xi/2}}{\theta}\right).
\end{align}

Substituting equations~\eqref{Evec.Ortho.S3} and~\eqref{Evec.Ortho.S4} into~\eqref{Evec.Ortho.S2} (and absorbing the error in $\mathcal F_N$) gives:
\begin{align}
\label{Evec.Ortho.detcent}
e_j^\prime v_i = \frac{1}{Y_{jj} \sqrt{\alpha_j}} \left[ S_2^{ji} + e_j^\prime e_i - \sum_{l \ne j,i} Y_{lj} \mathbb{E}[u(i)]\,\frac{\theta}{\mu_0} \left[Y^{-1} \overline{V}_i \right](l) \right] + \frac{e_j^\prime W e_i}{\mu_0 Y_{jj} \sqrt{\alpha_j}} + \mathcal{F}_N.
\end{align}

Now define the centered quantity:
\[
\bar{R} := e_j^\prime v_i - \rho - \widetilde{\Sigma}, \quad \text{where}
\]
\[
\rho := \frac{1}{Y_{jj} \sqrt{\alpha_j}} \left[ S_2^{ji} + e_j^\prime e_i - \sum_{l \ne j,i} Y_{lj} \mathbb{E}[u(i)]\,\frac{\theta}{\mu_0} \left[Y^{-1} \overline{V}_i \right](l) \right],
\]
\[
\widetilde{\Sigma} := \frac{e_j^\prime W e_i}{\mu_0 Y_{jj} \sqrt{\alpha_j}}.
\]

Clearly, \(\mathbb{E}[\bar{R}] = \mathbb{E}[e_j^\prime v_i] - \rho\). Using the same high-probability truncation trick as in Lemma~\ref{ExpectationBound}:
\begin{align*}
\mathbb{E}|\bar{R}| 
\le o\left( \frac{\sqrt{N}(\log N)^{3\xi/4}}{\theta^2} \right) + \mathbb{E}^{1/2}[\bar{R}^2] \cdot O\left(e^{-(\log N)^\eta}\right).
\end{align*}

Since \(\rho\) and \(e_j^\prime v_i\) are uniformly bounded and
\[
\mathbb{E}[\widetilde{\Sigma}^2] = {O}\left( \frac{1}{\theta^2} \right), 
\quad 
\mathbb{E}[\rho \widetilde{\Sigma}] = {O}\left( \frac{1}{\theta^2} \right),
\]
we deduce:
\[
\mathbb{E}[\bar{R}^2] = {O}\left( \frac{1}{\theta^2} \right) \Rightarrow \mathbb{E}|\bar{R}| = o\left( \frac{1}{\theta} \right).
\]

Thus, the mean is approximated as:
\begin{equation}
\label{Evec.Ortho.ExpReplace}
\mathbb{E}[e_j^\prime v_i] = \rho + o\left( \frac{1}{\theta} \right).
\end{equation}

Finally, from~\eqref{Evec.Ortho.detcent} and~\eqref{Evec.Ortho.ExpReplace}, we conclude
\begin{equation}
\label{Evec.Ortho.final}
e_j^\prime v_i - \mathbb{E}[e_j^\prime v_i] = \frac{e_j^\prime W e_i}{\mu_0 Y_{jj} \sqrt{\alpha_j}} + o_{hp}\left(\frac{1}{\theta}\right).
\end{equation}

We will now analyze the leading contribution, namely the first-order term:
\[
e_j^\prime W e_i.
\]
By definition of $e_j$ and $e_i$ in terms of the function $h_j$ and $h_i$, we have
\begin{align}
e_j^\prime W e_i &= \sum_{1 \le a \ne b \le N} W(a,b) \left[\frac{h_j\left(\frac{a}{N}\right) h_i\left(\frac{b}{N}\right)}{N} + \frac{h_i\left(\frac{a}{N}\right) h_j\left(\frac{b}{N}\right)}{N}\right] \notag \\
&\quad + \sum_{1 \le a \le N} \frac{W(a,a) h_j\left(\frac{a}{N}\right) h_i\left(\frac{a}{N}\right)}{N}. \label{eq:offdiag-ejWei}
\end{align}
The diagonal contribution vanishes in probability by the weak law of large numbers:
\[
\sum_{1 \le a \le N} \frac{W(a,a) h_j\left(\frac{a}{N}\right) h_i\left(\frac{a}{N}\right)}{N} \xrightarrow{\mathbb{P}} 0.
\]

Now consider the variance of the off-diagonal sum in \eqref{eq:offdiag-ejWei}. Define
\begin{align*}
S_N^2 := \mathrm{Var}(e_j^\prime W e_i) &= \sum_{1 \le a \ne b \le N} \mathrm{Var}(W(a,b)) \left[ \frac{h_j\left(\frac{a}{N}\right) h_i\left(\frac{b}{N}\right)}{N} + \frac{h_i\left(\frac{a}{N}\right) h_j\left(\frac{b}{N}\right)}{N} \right]^2 \\
&= \sum_{1 \le a \ne b \le N} \left[ \frac{h_j^2\left(\frac{a}{N}\right) h_i^2\left(\frac{b}{N}\right)}{N^2} + \frac{h_i^2\left(\frac{a}{N}\right) h_j^2\left(\frac{b}{N}\right)}{N^2} \right. \\
&\quad \left. + \frac{2 h_j\left(\frac{a}{N}\right) h_i\left(\frac{b}{N}\right) h_i\left(\frac{a}{N}\right) h_j\left(\frac{b}{N}\right)}{N^2} \right] f\left(\frac{a}{N}, \frac{b}{N}\right).
\end{align*}

As $N \to \infty$, we may pass to the limit using Riemann sums and the boundedness of $f$ and $h_i, h_j$, to conclude:
\begin{align}
\lim_{N \to \infty} S_N^2 &= \int_{[0,1]^2} \left[h_j(x) h_i(y) + h_i(x) h_j(y)\right]^2 f(x,y)\, dxdy. \label{eq:var-ejWei}
\end{align}

The sum in \eqref{eq:offdiag-ejWei} consists of a finite sum of independent mean-zero variables with variances given above. Since the fourth moments of $W(a,b)$ are uniformly bounded (by assumption A2), Lyapunov’s condition is satisfied. Therefore, by Lyapunov’s central limit theorem, we conclude:
\[
e_j^\prime W e_i \xrightarrow{d} \mathcal{N}(0, \sigma_{ij}^2),
\]
where $\sigma_{ij}^2$ is given by the right-hand side of \eqref{eq:var-ejWei}.

Observe that
\[
Y_{jj}  \to \frac{\alpha_i}{\alpha_i - \alpha_j}, \qquad \text{and} \qquad \frac{\mu_0}{\theta} \to \alpha_i.
\]
Hence multiplying both sides by $\theta$ gives:
\[
\theta \left(e_j^\prime v_i - \mathbb{E}[e_j^\prime v_i]\right) \xrightarrow{d} \mathcal{N}\left(0,\, \tau_{ij}^2\right),
\]
where
\begin{align*}
\tau_{ij}^2 = \frac{1}{(\alpha_i-\alpha_j)^2 \alpha_j}\int_{[0,1]^2} \left[h_j(x) h_i(y) + h_i(x) h_j(y)\right]^2 f(x,y)\, dx\, dy.
\end{align*}
\end{proof}

\subsection{Proof of Theorem \ref{Decolalization}}
We begin by an estimate whose proof will be presented in the Appendix.

\begin{lemma}\label{Deloc.est}
Let \( n \le L \). Then, for each \( j \) and for fixed \( i \), with high probability we have
\[
\left| (W^n e_j)(i) \right| = O_{hp}\left( \sqrt{N^{n-1}} (\log N)^{n\xi/4} \right).
\]
\end{lemma}

\begin{proof}[Proof of Theorem \ref{Decolalization}]
We begin by recalling the representation of the eigenvector \( v_i \) from equation \eqref{eq.master1}:
\begin{align*}
v_i = \theta \sum_{l=1}^k \alpha_l (e_l^\prime v_i) (\mu I - W)^{-1} e_l.
\end{align*}

Now, fix a coordinate \( j \in \{1, \dots, N\} \). Taking the \( j \)-th component on both sides, we obtain:
\[
v_i(j) = \theta \sum_{l=1}^k \alpha_l (e_l^\prime v_i) \left[ (\mu I - W)^{-1} e_l \right](j).
\]

We now apply the alignment estimate from Lemma \ref{Align.bounds} to see that the dominant contribution comes from the \( l = i \) term. Thus,
\begin{align*}
v_i(j) 
&= \theta \alpha_i \left[1 + o_{hp}\left( \frac{\sqrt{N}}{\theta} \right)\right] \left[ (\mu I - W)^{-1} e_i \right](j) + \text{negligible terms}.
\end{align*}

Using series expansion and applying Lemma \ref{Deloc.est} with analysis similar to Lemma \ref{Le.S1}, \ref{Le.S2} and \ref{Le.S4}, we obtain
\[
\left[ (\mu I - W)^{-1} e_i \right](j) 
= \frac{1}{\mu} \left[ e_i(j) + O_{hp}\left( \frac{(\log N)^{\xi/4}}{\theta} \right) + O_{hp}\left( \frac{1}{(\log N)^{\log N}} \right) \right].
\]

Substituting this back into the expression for \( v_i(j) \), we get:
\begin{align*}
v_i(j) 
&= \theta \alpha_i \left[1 + o_{hp}\left( \frac{\sqrt{N}}{\theta} \right) \right]
\cdot \frac{1}{\mu} \left[ e_i(j) + O_{hp}\left( \frac{(\log N)^{\xi/4}}{\theta} \right) \right] \\
&= \left[1 + o_{hp}\left( \frac{\sqrt{N}}{\theta} \right)\right] \left[ e_i(j) + O_{hp}\left( \frac{(\log N)^{\xi/4}}{\theta} \right) \right].
\end{align*}

Thus, for each \( j \in \{1, \dots, N\} \), we find
\begin{align}
\label{conclusion}
|v_i(j) - e_i(j)| 
= o_{hp}\left( \frac{1}{\theta} \right) + O_{hp}\left( \frac{(\log N)^{\xi/4}}{\theta} \right)
= O_{hp}\left( \frac{(\log N)^{\xi/4}}{\theta} \right).
\end{align}

This completes the proof.
\end{proof}

\section{Eigenvector process and Proof of Theorem \ref{EVecProMain}}\label{section:EVprocess}
\begin{proof}[Proof of Theorem \ref{EVecProMain}]
Let \( V_N^{(i)} : [0,1] \to \mathbb{R} \) be the piecewise-constant function defined in \eqref{eq:Processfn}, and define the fluctuation process \( \mathcal{V}_N \) via
\[
\langle \mathcal{V}_N^{(i)}, g \rangle := \theta\sqrt{N} \left( \underline{g}' v_i - \mathbb{E}(\underline{g}' v_i) \right), \qquad \text{for } g \in C_0^\infty(0,1).
\]
Here, the discretized test function \( \underline{g} \in \mathbb{R}^N \) is defined by
\begin{equation}\label{def:lowerg}
\underline{g}(j) := \int_{(j-1)/N}^{j/N} g(t)\,dt.
\end{equation}

\vspace{1em}
\noindent\textbf{Step 1: Resolvent expansion and representation of $\underline{g}'v_i$.} Using \eqref{eq.master1}, we expand
\begin{align*}
\underline{g}' v_i 
&= \frac{\theta}{\mu} \alpha_i (e_i' v_i) \, \underline{g}'\left(I - \frac{W}{\mu}\right)^{-1} e_i 
+ \frac{\theta}{\mu} \sum_{\substack{1 \le l \le k \\ l \ne i}} \alpha_l (e_l' v_i)\, \underline{g}'\left(I - \frac{W}{\mu}\right)^{-1} e_l \\
&=: A_i + \sum_{\substack{1 \le l \le k \\ l \ne i}} B_l,
\end{align*}
where we isolate the dominant term \( A_i \) corresponding to \( l = i \), and denote the cross terms by \( B_l \).

By applying the resolvent expansion and using Lemmas \ref{Le.S1}, \ref{Le.S2}, and \ref{Le.S4}, we obtain:
\begin{align*}
\frac{A_i}{\alpha_i} 
&= \frac{\theta}{\mu}(e_i' v_i)(\underline{g}' e_i) 
+ \frac{\theta}{\mu^2}(e_i' v_i)(\underline{g}' W e_i) 
+ \frac{\theta}{\mu}(e_i' v_i)\left[\sum_{n=2}^L \frac{\mathbb{E}(\underline{g}' W^n e_i)}{\bar{\mu}^n} 
+ o_{hp}\left(\frac{1}{\sqrt{N}\theta}\right)\right],
\end{align*}
where \( L = \lfloor \log N \rfloor \), and \( \bar{\mu} \) is the deterministic approximation to \( \mu \) defined in Lemma \ref{Le.S5}.

We next approximate the third term using the expansion of \( (e_i' v_i) \) around its mean as in \eqref{eq:evecalign.2}, together with Remark \ref{Error.negpow,mumu0} and Lemma \ref{Le.S4}:
\begin{align*}
\frac{\theta}{\mu}(e_i' v_i) \sum_{n=2}^L \frac{ \mathbb{E}(\underline{g}' W^n e_i) }{ \bar{\mu}^n }
&= \frac{\theta}{\mu} \mathbb{E}(e_i' v_i) \sum_{n=2}^L \frac{ \mathbb{E}(\underline{g}' W^n e_i) }{ \bar{\mu}^n } 
+ O_{hp}\left( \frac{N(\log N)^{\xi/2}}{\theta^4} \right) \\
&= \frac{\theta}{\bar{\mu}} \mathbb{E}(e_i' v_i) \sum_{n=2}^L \frac{ \mathbb{E}(\underline{g}' W^n e_i) }{ \bar{\mu}^n } 
+ O_{hp}\left( \frac{|\mu - \bar{\mu}| \sqrt{N}}{\theta^3} \right) 
+ O_{hp}\left( \frac{N(\log N)^{\xi/2}}{\theta^4} \right) \\
&= \frac{\theta}{\bar{\mu}} \mathbb{E}(e_i' v_i) \sum_{n=2}^L \frac{ \mathbb{E}(\underline{g}' W^n e_i) }{ \bar{\mu}^n } 
+ o_{hp}\left( \frac{1}{\sqrt{N} \theta} \right).
\end{align*}

We now obtain the approximation:
\begin{equation}\label{EVecproTh1.2}
\frac{A_i}{\alpha_i} - \frac{\theta}{\bar{\mu}} \mathbb{E}(e_i' v_i) \sum_{n=2}^L \frac{ \mathbb{E}(\underline{g}' W^n e_i) }{ \bar{\mu}^n }
= \frac{\theta}{\mu}(e_i' v_i)(\underline{g}' e_i) 
+ \frac{\theta}{\mu^2}(e_i' v_i)(\underline{g}' W e_i) 
+ o_{hp}\left( \frac{1}{\sqrt{N} \theta} \right).
\end{equation}

Using again \eqref{eq:evecalign.2}, Remark \ref{Error.negpow,mumu0}, and Lemma \ref{Le.S4} on the first two terms, we upgrade to:
\begin{equation}\label{EVecproTh1.3}
\frac{A_i}{\alpha_i} - \frac{\theta}{\bar{\mu}} \mathbb{E}(e_i' v_i) \sum_{n=0}^L \frac{ \mathbb{E}(\underline{g}' W^n e_i) }{ \bar{\mu}^n }
= \theta \mathbb{E}(e_i' v_i)(\underline{g}' e_i) \left( \frac{1}{\mu} - \frac{1}{\bar{\mu}} \right) 
+ \frac{\theta}{\bar{\mu}^2} \mathbb{E}(e_i' v_i)(\underline{g}' W e_i) 
+ o_{hp}\left( \frac{1}{\sqrt{N} \theta} \right).
\end{equation}

Finally, multiplying both sides by \( \alpha_i \), we obtain:
\begin{align}\label{EVecproTh1.4}
&A_i - \frac{\theta}{\bar{\mu}} \alpha_i \mathbb{E}(e_i' v_i) \sum_{n=0}^L \frac{ \mathbb{E}(\underline{g}' W^n e_i) }{ \bar{\mu}^n }\nonumber\\
&=- \frac{\theta^2}{\bar{\mu}^3} \alpha_i \mathbb{E}(e_i' v_i)(\underline{g}' e_i)(e_i' W e_i) + \frac{\theta}{\bar{\mu}^2} \alpha_i \mathbb{E}(e_i' v_i)(\underline{g}' W e_i) 
+ o_{hp}\left( \frac{1}{\sqrt{N} \theta} \right).
\end{align}

Now, fix \( l \ne i \). Applying similar calculations (along with the results from Lemmas~\ref{Le.S1}, \ref{Le.S2}, \ref{Le.S4} and \eqref{Evec.Ortho.final}), we obtain:
\begin{align*}
\frac{B_l}{\alpha_l} 
&= \frac{\theta}{\mu} (e_l' v_i)\, (\underline{g}' e_l) 
+ \frac{\theta}{\mu^2} (e_l' v_i)\, (\underline{g}' W e_l) 
+ \frac{\theta}{\bar{\mu}} \mathbb{E}(e_l' v_i) \sum_{n=2}^L \frac{ \mathbb{E}(\underline{g}' W^n e_l) }{ \bar{\mu}^n } 
+ o_{hp}\left( \frac{1}{\sqrt{N} \theta} \right).
\end{align*}

Using Lemma~\ref{ConcentrationBound} for \( n = 1 \), Lemma~\ref{Align.bounds}, and Lemma~\ref{Le.S4}, we further approximate:
\begin{align*}
\frac{B_l}{\alpha_l} 
&= \frac{\theta}{\bar{\mu}} (e_l' v_i)\, (\underline{g}' e_l) 
+ \frac{\theta}{\bar{\mu}} \mathbb{E}(e_l' v_i) \sum_{n=2}^L \frac{ \mathbb{E}(\underline{g}' W^n e_l) }{ \bar{\mu}^n } 
+ o_{hp}\left( \frac{1}{\sqrt{N} \theta} \right).
\end{align*}

Thus, by applying \eqref{Evec.Ortho.final} again, we get:
\begin{align}
\nonumber
& \sum_{\substack{1 \le l \le k \\ l \ne i}} B_l 
- \sum_{\substack{1 \le l \le k \\ l \ne i}} \alpha_l \cdot \frac{\theta}{\bar{\mu}} \mathbb{E}(e_l' v_i) \sum_{n=0}^L \frac{ \mathbb{E}(\underline{g}' W^n e_l) }{ \bar{\mu}^n } \\
\nonumber
&\quad = \sum_{\substack{1 \le l \le k \\ l \ne i}} \frac{\theta}{\bar{\mu}} \alpha_l\, (\underline{g}' e_l) \cdot \left( e_l' v_i - \mathbb{E}(e_l' v_i) \right) + o_{hp}\left( \frac{1}{\sqrt{N} \theta} \right) \\
\label{EVecProTh1.Orthopart}
&\quad = \sum_{\substack{1 \le l \le k \\ l \ne i}} \frac{\theta}{\bar{\mu}^2} \sqrt{\alpha_l}\, (\underline{g}' e_l) \cdot \frac{e_l' W e_i}{Y_{ll}} + o_{hp}\left( \frac{1}{\sqrt{N} \theta} \right),
\end{align}
where \( Y_{ll} \) is as defined \eqref{eq:Ymatrix}. Note that like before we replace $\mu_0$ by $\bar{\mu}$ with effecting the $o(1)$ error.

Combining the expansions \eqref{EVecproTh1.4} and \eqref{EVecProTh1.Orthopart}, we reorganize the expression for \( \underline{g}' v \) into the form:
\begin{align}
\underline{g}' v 
=\, &\underbrace{
\frac{\theta}{\bar{\mu}^2} \alpha_i\, \mathbb{E}(e_i' v_i)\, \underline{g}' W e_i 
\;-\; \frac{\theta^2}{\bar{\mu}^3} \alpha_i\, \mathbb{E}(e_i' v_i)\, \underline{g}' e_i \cdot e_i' W e_i 
\;+\; \sum_{\substack{1 \le l \le k \\ l \ne i}} \frac{\theta}{\bar{\mu}^2} \sqrt{\alpha_l}\, (\underline{g}' e_l)\, \frac{e_l' W e_i}{Y_{ll}}
}_{\text{\bf Random part: }\mathcal{R}_N(g)}
\\
&\quad +\; \underbrace{
\sum_{l=1}^k \sum_{n=0}^L \alpha_l\, \frac{\theta}{\bar{\mu}}\, \mathbb{E}(e_l' v_i)\, \frac{ \mathbb{E}(\underline{g}' W^n e_l) }{ \bar{\mu}^n }
}_{\text{\bf Deterministic part: }\mathcal{D}_N(g)}
\;+\; \underbrace{o_{hp}\left( \frac{1}{\sqrt{N} \theta} \right)}_{\text{\bf Error term: }\mathcal{E}_N(g)}.
\label{eq:gprimev-RDE}
\end{align}

We now justify subtracting the deterministic centering term from the expansion \eqref{eq:gprimev-RDE}. Define the residual
\begin{equation}
\label{EVecproTh1.Rem}
\mathbf{R} := \underline g'v_i -\mathcal D_N(g)- \mathcal{R}_N(g).
\end{equation}

Clearly, the expectation satisfies:
\[
\mathbb{E}(\mathbf{R}) = \mathbb{E}(\underline{g}' v_i) - \sum_{l=1}^k\sum_{n=0}^L \alpha_l \frac{\theta}{\bar \mu} \mathbb{E}(e_l' v_i)  \frac{\mathbb{E}(\underline{g}' W^n e_l)}{\bar \mu^n}.
\]

To control the fluctuation of \( \mathbf{R} \), we invoke high-probability bounds. For any \( \delta > 0 \), there exists \( \eta > 1 \) such that
\[
\mathbb{E}|\mathbf{R}| \le o\left( \frac{1}{\sqrt{N} \theta} \right) + \mathbb{E}(\mathbf{R}^2) \cdot \exp\left( - (\log N)^\eta \right).
\]

We now estimate \( \mathbb{E}(\mathbf{R}^2) \) using moment bounds from Lemmas~\ref{Le.S4}, \ref{ExpectationBound}, and related facts. Since \( \mathbf{R} \) involves quadratic forms in \( W \), we compute their variances explicitly.

From symmetry and independence of \( W(a,b) \), we obtain:
\begin{align}
\label{EVecproTh1.5}
\mathbb{E}(e_i' W e_i)^2 
= \frac{2}{N^2} \sum_{1 \le i_1, i_2 \le N} h_i^2\left( \frac{i_1}{N} \right) h_i^2\left( \frac{i_2}{N} \right) f\left( \frac{i_1}{N}, \frac{i_2}{N} \right) 
\to 2 \int_0^1 \int_0^1 h_i^2(x) h_i^2(y) f(x,y)\, dxdy.
\end{align}

For the linear-quadratic term, we compute:
\begin{align}
\nonumber
N \cdot \mathbb{E}(\underline g' W e_i)^2 
&= \frac{1}{N^2} \bigg[
\sum_{i_1, i_2} g^2(\frac{i_1}{N}) h_i^2\left( \frac{i_2}{N} \right) f\left( \frac{i_1}{N}, \frac{i_2}{N} \right) \\
&\quad + \sum_{i_1, i_2} g(\frac{i_1}{N}) g(\frac{i_2}{N}) h_i\left( \frac{i_1}{N} \right) h_i\left( \frac{i_2}{N} \right) f\left( \frac{i_1}{N}, \frac{i_2}{N} \right) \bigg] \notag \\
\label{EVecproTh1.6}
&\to \int_0^1 \int_0^1 g^2(x) h_i^2(y) f(x,y)\, dxdy 
+ \int_0^1 \int_0^1 g(x) g(y) h_i(x) h_i(y) f(x,y)\, dxdy.
\end{align}

Next, for the cross term:
\begin{align}
\nonumber
\sqrt{N} \cdot \mathbb{E}(\underline g' W e_i \cdot e_i' W e_i) 
&= \frac{1}{N^2} \bigg[
\sum_{i_1,i_2} g(\frac{i_1}{N}) h_i\left( \frac{i_1}{N} \right) h_i^2\left( \frac{i_2}{N} \right) f\left( \frac{i_1}{N}, \frac{i_2}{N} \right) \\
&\quad + \sum_{i_1,i_2} h_i^2\left( \frac{i_1}{N} \right) h_i\left( \frac{i_2}{N} \right) g(\frac{i_2}{N}) f\left( \frac{i_1}{N}, \frac{i_2}{N} \right) \bigg] \notag \\
\label{EVecproTh1.7}
&\to 2 \int_0^1 \int_0^1 g(x) h_i(x) h_i^2(y) f(x,y)\, dxdy.
\end{align}

Similarly, for the cross terms involving \( l \ne i \), we have:
\begin{align}
\label{EVecproTh1.9}
\mathbb{E}(e_l' W e_i)^2 
\to \int_0^1 \int_0^1 h_l^2(x) h_i^2(y) f(x,y)\, dxdy 
+ \int_0^1 \int_0^1 h_l(x) h_l(y) h_i(x) h_i(y) f(x,y)\, dxdy.
\end{align}

\begin{align}
\label{EVecproTh1.10}
\mathbb{E}(e_l' W e_i \cdot e_i' W e_i) 
\to 2 \int_0^1 \int_0^1 h_l(x) h_i(x) h_i^2(y) f(x,y)\, dxdy.
\end{align}

\begin{align}
\nonumber
\sqrt{N} \cdot \mathbb{E}(\underline g' W e_i \cdot e_l' W e_i) 
&\to \int_0^1 \int_0^1 g(x) h_l(x) h_i^2(y) f(x,y)\, dxdy \\
\label{EVecproTh1.11}
&\quad + \int_0^1 \int_0^1 g(x) h_i(x) h_i(y) h_l(y) f(x,y)\, dxdy.
\end{align}

Finally, using that \( \|\underline g\|_2 = O(1/\sqrt{N}) \) and collecting all bounds, we obtain:
\[
\mathbb{E}^{1/2}(\mathbf{R}^2) \cdot \exp(-(\log N)^\eta) = o\left( \frac{1}{\sqrt{N} \theta} \right),
\]
and hence
\begin{align}
\label{EVecproTh1.8}
\mathbb{E}|\mathbf{R}| = o\left( \frac{1}{\sqrt{N} \theta} \right).
\end{align}

\begin{equation}\label{eq:VN.fluctuation.decomp}
\langle \mathcal V_N^{(i)}, g \rangle 
= \theta \sqrt{N} \left( \underline{g}' v_i - \mathbb{E}(\underline{g}' v_i) \right) 
= \mathbf{Fluct}_i + \mathbf{Fluct}_{\text{cross}} + \mathcal{E}_N,
\end{equation}
where
\begin{align*}
\mathbf{Fluct}_i 
&:= \frac{\theta^2 \sqrt{N}}{\bar{\mu}^2} \alpha_i\, \mathbb{E}(e_i' v_i)\, \underline{g}' W e_i 
- \frac{\theta^3 \sqrt{N}}{\bar{\mu}^3} \alpha_i\, \mathbb{E}(e_i' v_i)\, \underline{g}' e_i\, e_i' W e_i, \\
\mathbf{Fluct}_{\text{cross}} 
&:= \sum_{\substack{1 \le l \le k \\ l \ne i}} \frac{\theta^2 \sqrt{N}}{\bar{\mu}^2} \sqrt{\alpha_l}\, (\underline{g}' e_l)\, \frac{e_l' W e_i}{Y_{ll}}, \\
\mathcal{E}_N 
&:= o_{hp}\left(1\right).
\end{align*}

\noindent\textbf{Step 2: Functional convergence.}
Again one can apply the Lyapunov CLT to derive the following for each $g \in C_0^\infty(0,1)$,
\[
\langle \mathcal V^{(i)}_N, g \rangle \Rightarrow \mathcal{N}(0, \sigma^2(g)).
\]

The asymptotic variance of $\langle \mathcal V^{(i)}_N, g \rangle$ can be written as:
\begin{small}
\begin{align*}
\sigma^2(g) &= 
2\alpha_i^2 \left( \int_0^1 g(x) h_i(x)\, dx \right)^2 
\int_0^1 \int_0^1 h_i^2(x) h_i^2(y) f(x,y)\, dxdy \\
&\quad + \alpha_i^2 \int_0^1 \int_0^1 g^2(x) h_i^2(y) f(x,y)\, dxdy 
+ \alpha_i^2 \int_0^1 \int_0^1 g(x) h_i(x) g(y) h_i(y) f(x,y)\, dxdy  \\
&\quad - 4\alpha_i^2 \left( \int_0^1 g(x) h_i(x)\, dx \right) 
\int_0^1 \int_0^1 h_i(x) g(x) h_i^2(y) f(x,y)\, dxdy  \\
&\quad + \sum_{\substack{1 \le l \le k \\ l \ne i}} \frac{2 \alpha_i^2 \sqrt{\alpha_l}}{\alpha_i - \alpha_l}
\left( \int_0^1 g(x) h_l(x)\, dx \right)
\left[
\int_0^1 \int_0^1 g(x) h_l(x) h_i^2(y) f(x,y)\, dxdy \right. \nonumber \\
&\qquad\left. + \int_0^1 \int_0^1 g(x) h_i(x) h_i(y) h_l(y) f(x,y)\, dxdy
\right]  \\
&\quad - 4\sum_{\substack{1 \le l \le k \\ l \ne i}} \frac{ \alpha_i^2 \sqrt{\alpha_l} }{\alpha_i - \alpha_l}
\left( \int_0^1 g(x) h_l(x)\, dx \right)
\left( \int_0^1 g(x) h_i(x)\, dx \right)
\int_0^1 \int_0^1 h_l(x) h_i(x) h_i^2(y) f(x,y)\, dxdy  \\
&\quad + \sum_{\substack{1 \le l_1, l_2 \le k \\ l_1, l_2 \ne i}} 
\frac{\alpha_i^2 \sqrt{\alpha_{l_1}} \sqrt{\alpha_{l_2}}}{(\alpha_i - \alpha_{l_1})(\alpha_i - \alpha_{l_2})}
\left( \int_0^1 g(x) h_{l_1}(x)\, dx \right)
\left( \int_0^1 g(x) h_{l_2}(x)\, dx \right) \nonumber \\
&\qquad \times 
\left[
\int_0^1 \int_0^1 h_{l_1}(x) h_{l_2}(x) h_i^2(y) f(x,y)\, dxdy 
+ \int_0^1 \int_0^1 h_{l_1}(x) h_i(x) h_i(y) h_{l_2}(y) f(x,y)\, dxdy
\right]. 
\end{align*}
\end{small}
\vspace{1em}

\noindent\textbf{Step 3: Finite dimensional convergence.}
Let $g_1, g_2, \ldots, g_l \in C[0,1]$, and define(similar to \eqref{def:lowerg}) the discretized vectors $\underline{g}_j \in \mathbb{R}^N$ by
\begin{align}
\underline{g}_j(p) := \int_{(p-1)/N}^{p/N} g_j(t)\, dt.
\end{align}
Fix \( i \in \{1,\ldots,k\} \), and let \( v_i \) denote the normalized eigenvector of \( A \) corresponding to the eigenvalue \( \lambda_i(A) \). Then, as \( N \to \infty \), the vector
\[
\left( \sqrt{N} \theta \left( \underline{g}_j' v_i - \mathbb{E}[\underline{g}_j' v_i] \right) : 1 \le j \le l \right)
\]
converges in distribution to a multivariate Gaussian vector \( (H_1, \ldots, H_l) \in \mathbb{R}^l \), with mean zero and covariance matrix \( \Sigma = (\mathrm{Cov}(H_p, H_q))_{1 \le p,q \le l} \), where
\begin{small}
\begin{align}
\label{covarianceform}
&\mathrm{Cov}(H_p, H_q) = C_k^{(i)}(g_p, g_q)\\
&=2\alpha_i^2 \left( \int_0^1 g_p(x) h_i(x)\, dx \right) \left( \int_0^1 g_q(x) h_i(x)\, dx \right) \int_{[0,1]^2} h_i^2(x) h_i^2(y) f(x,y)\, dxdy \nonumber\\
\nonumber
& + \alpha_i^2 \int_{[0,1]^2} g_p(x) g_q(x) h_i^2(y) f(x,y)\, dxdy 
+ \alpha_i^2 \int_{[0,1]^2} g_p(x) h_i(x) g_q(y) h_i(y) f(x,y)\, dxdy \\
\nonumber
& - 2\alpha_i^2 \left( \int_0^1 g_q(x) h_i(x)\, dx \right) \int_{[0,1]^2} h_i(x) g_p(x) h_i^2(y) f(x,y)\, dxdy \\
\nonumber
& - 2\alpha_i^2 \left( \int_0^1 g_p(x) h_i(x)\, dx \right) \int_{[0,1]^2} h_i(x) g_q(x) h_i^2(y) f(x,y)\, dxdy \\
\nonumber
& + \sum_{\substack{1 \le l \le k \\ l \ne i}} \frac{\alpha_i^2 \sqrt{\alpha_l}}{\alpha_i - \alpha_l} 
\left( \int_0^1 g_q(x) h_l(x)\, dx \right)
\Bigg[
\int_{[0,1]^2} g_p(x) h_l(x) h_i^2(y) f(x,y)\, dxdy \\
\nonumber
& \qquad + \int_{[0,1]^2} g_p(x) h_i(x) h_i(y) h_l(y) f(x,y)\, dxdy
\Bigg] \\
\nonumber
& - 2\sum_{\substack{1 \le l \le k \\ l \ne i}} \frac{\alpha_i^2 \sqrt{\alpha_l}}{\alpha_i - \alpha_l} 
\left( \int_0^1 g_q(x) h_l(x)\, dx \right)
\left( \int_0^1 g_p(x) h_i(x)\, dx \right)
\int_{[0,1]^2} h_l(x) h_i(x) h_i^2(y) f(x,y)\, dxdy \\
\nonumber
& + \sum_{\substack{1 \le l \le k \\ l \ne i}} \frac{\alpha_i^2 \sqrt{\alpha_l}}{\alpha_i - \alpha_l} 
\left( \int_0^1 g_p(x) h_l(x)\, dx \right)
\Bigg[
\int_{[0,1]^2} g_q(x) h_l(x) h_i^2(y) f(x,y)\, dxdy \\
\nonumber
& \qquad + \int_{[0,1]^2} g_q(x) h_i(x) h_i(y) h_l(y) f(x,y)\, dxdy
\Bigg] \\
\nonumber
& - 2\sum_{\substack{1 \le l \le k \\ l \ne i}} \frac{\alpha_i^2 \sqrt{\alpha_l}}{\alpha_i - \alpha_l} 
\left( \int_0^1 g_p(x) h_l(x)\, dx \right)
\left( \int_0^1 g_q(x) h_i(x)\, dx \right)
\int_{[0,1]^2} h_l(x) h_i(x) h_i^2(y) f(x,y)\, dxdy \\
\nonumber
& + \sum_{\substack{1 \le l_1, l_2 \le k \\ l_1, l_2 \ne i}} 
\frac{\alpha_i^2 \sqrt{\alpha_{l_1}} \sqrt{\alpha_{l_2}}}{(\alpha_i - \alpha_{l_1})(\alpha_i - \alpha_{l_2})}
\left( \int_0^1 g_p(x) h_{l_1}(x)\, dx \right)
\left( \int_0^1 g_q(x) h_{l_2}(x)\, dx \right) \\
\nonumber
& \qquad \times \Bigg[
\int_{[0,1]^2} h_{l_1}(x) h_{l_2}(x) h_i^2(y) f(x,y)\, dxdy 
+ \int_{[0,1]^2} h_{l_1}(x) h_i(x) h_i(y) h_{l_2}(y) f(x,y)\, dxdy
\Bigg].
\end{align}
\end{small}

Fix any $\varepsilon > 0$ and $d > \frac{1}{2}$. The finite-dimensional convergence above ensures that for any $l \in \mathbb{N}$ and test functions $g_1, g_2, \ldots, g_l \in \mathcal H^d_0$, we have weak convergence of the finite-dimensional projections of $\mathcal V_N$. To conclude the proof of Theorem~\ref{EVecProMain}, it remains to verify:
\begin{enumerate}
  \item tightness of the sequence $\mathcal V_N$ in $\mathcal{H}^{-d}$;
  \item uniqueness of the limiting distribution.
\end{enumerate}
\noindent\textbf{Step 4: Tightness.}
Tightness depends on the following lemma.

\begin{lemma}[Uniform bound for Fourier modes]\label{EVecproPreTight}
There exists a constant \( C > 0 \) such that
\[
\limsup_{N \to \infty} \sup_{j \ge 1} \mathbb{E}\left( \langle \mathcal V^{(i)}_N , \phi_j \rangle^2 \right) < C.
\]
\end{lemma}

\begin{proof}[Proof of Lemma \ref{EVecproPreTight}]
We use the representation from Step 1 with \( k = 1 \) and \( \alpha_1 = 1 \). Fix \( g = \phi_j \in C^\infty[0,1] \) for some \( j \ge 1 \), and define \( \underline{g} \in \mathbb{R}^N \) as in \eqref{def:lowerg}. Then, from Step 1 we obtain:
\begin{align}
\mathbb{E} \langle \mathcal V^{(i)}_N, \phi_j \rangle^2 
&= \mathbb{E}\left( \theta \sqrt{N} (\underline{g}'v_i - \mathbb{E}[\underline{g}'v_i]) \right)^2 
= N\theta^2 \cdot \mathbb{E} \left( \underline{g}'v_i - \mathbb{E}[\underline{g}'v_i] \right)^2 \notag \\
&= \mathbb{E}\left( \mathbb{E}(e^\prime v_i)\frac{\theta}{\mu^2} \underline{g}'We 
- \mathbb{E}(e^\prime v_i)\frac{\theta^2}{\mu^3}\underline{g}'e \cdot e'We 
+ o_{hp}\left(\frac{1}{\sqrt{N}\theta}\right) \right)^2 N\theta^2. \label{EVecproTight1.1}
\end{align}

Using the uniform bound \( \sup_{j \ge 1} \sup_{x \in [0,1]} |\phi_j(x)| \le \sqrt{2} \), we note that the error terms denoted by \( \mathbf{R} \) and \( \mathbf{R}^2 \) in Step 1 are independent of \( j \), uniformly over all \( j \ge 1 \).

Now applying the variance expressions from \eqref{EVecproTh1.5}--\eqref{EVecproTh1.7}, and the error control from \eqref{EVecproTh1.8} and $\E[\mathbf R^2]$, we find that for all \( j \ge 1 \),
\begin{align}
\mathbb{E} \langle \mathcal V^{(i)}_N, \phi_j \rangle^2 
&\le 2 \left[ \frac{1}{N} \sum_{i_1=1}^N h\left( \frac{i_1}{N} \right) \right]^2 \cdot \mathbb{E}^2(e'v) \cdot \frac{\theta^6}{\bar{\mu}^6} \cdot \mathbb{E}(e'We)^2 \notag \\
&\quad + \frac{\theta^4}{\bar{\mu}^4} \mathbb{E}^2(e'v) \cdot \frac{2}{N^2} \left[ \sum_{i_1,i_2=1}^N h^2\left( \frac{i_2}{N} \right) f\left( \frac{i_1}{N}, \frac{i_2}{N} \right) + \sum_{i_1,i_2=1}^N h\left( \frac{i_1}{N} \right) h\left( \frac{i_2}{N} \right) f\left( \frac{i_1}{N}, \frac{i_2}{N} \right) \right] \notag \\
&\quad + 2 \left[ \frac{1}{N} \sum_{i=1}^N h\left( \frac{i}{N} \right) \right] \cdot \mathbb{E}^2(e'v) \cdot \frac{\theta^5}{\bar{\mu}^5} \cdot \frac{1}{N^2} \left[ \sum_{i_1,i_2=1}^N h\left( \frac{i_1}{N} \right) h^2\left( \frac{i_2}{N} \right) f\left( \frac{i_1}{N}, \frac{i_2}{N} \right) \right. \notag \\
&\hspace{4em} \left. + \sum_{i_1,i_2=1}^N h^2\left( \frac{i_1}{N} \right) h\left( \frac{i_2}{N} \right) f\left( \frac{i_1}{N}, \frac{i_2}{N} \right) \right] + o\left( \frac{1}{N\theta^2} \right). \label{EVecproTight1.11}
\end{align}

Taking the limit as \( N \to \infty \), we conclude:
\begin{align}
\limsup_{N \to \infty} \mathbb{E} \langle \mathcal V^{(i)}_N, \phi_j \rangle^2 
&\le 4 \left( \int_0^1 h(x)\,dx \right)^2 \cdot \int_0^1 \int_0^1 h^2(x) h^2(y) f(x,y)\,dxdy \notag \\
&\quad + 2 \left[ \int_0^1 h^2(y) f(x,y)\,dxdy + \int_0^1 h(x) h(y) f(x,y)\,dxdy \right] \notag \\
&\quad + 4 \left( \int_0^1 h(x)\,dx \right) \cdot \int_0^1 h(x) h^2(y) f(x,y)\,dxdy. \label{EVecproTight1.2}
\end{align}

All terms are bounded uniformly in \( j \), and hence we conclude:
\[
\sup_{N \ge 1} \sup_{j \ge 1} \mathbb{E} \langle \mathcal V^{(i)}_N, \phi_j \rangle^2 < \infty.
\]
This completes the proof.
\end{proof}

Using Lemma~\ref{EVecproPreTight}, the Monotone Convergence Theorem, and the definition of the $\mathcal{H}^{-d}$-norm via the orthonormal basis $\{\phi_j\}_{j \ge 1}$, we obtain:
\begin{align}
\label{EVecProMain.1}
\mathbb{E}\|\mathcal V_N^{(i)}\|_{\mathcal{H}^{-d}}^2 = \sum_{j \ge 1} \frac{ \mathbb{E} \langle \mathcal V_N^{(i)}, \phi_j \rangle^2 }{j^{2d}} \le C \sum_{j \ge 1} \frac{1}{j^{2d}} < \infty.
\end{align}

To conclude tightness, we construct a compact subset of $\mathcal{H}^{-d}$ with high probability content. The following compact embedding is standard (see \cite{cipriani2019scaling})
\begin{fact}[Rellich's Theorem]
For $c_1 < c_2$, the embedding $\mathcal{H}^{-c_1} \hookrightarrow \mathcal{H}^{-c_2}$ is compact. In particular, for any $R > 0$, the closed ball $\overline{B_{\mathcal{H}^{-c_1}}(0,R)}$ is compact in $\mathcal{H}^{-c_2}$.
\end{fact}

Applying this with $d > d' > \frac{1}{2}$ and using Markov's inequality, we get:
\begin{align}
\label{EVecProMain.2}
\mathbb{P}\left( \mathcal V_N^{(i)} \notin \overline{B_{\mathcal{H}^{-d'}}(0, R)} \right) \le \frac{ \mathbb{E} \|\mathcal V_N^{(i)}\|_{\mathcal{H}^{-d'}}^2 }{R^2}.
\end{align}

Choosing $R = \sqrt{ \frac{C \sum_{j \ge 1} j^{-2d'}}{\varepsilon} }$ ensures that the above probability is at most $\varepsilon$, thereby establishing tightness of $\mathcal V_N$ in $\mathcal{H}^{-d}$ for any $d > \frac{1}{2}$.

\noindent\textbf{Step 6: Uniqueness of the limit}

Let $\mathcal{E}_g : \mathcal{H}^{-1} \to \mathbb{R}$ denote the evaluation functional defined by $\mathcal{E}_g(h) := \langle h, g \rangle$ for $g \in \mathcal{H}_0^1$. For $m \in \mathbb{N}$ and test functions $f_1, \ldots, f_m \in \mathcal{H}_0^1$, define the vector-valued map:
\[
\mathcal{E}_{f_1, \ldots, f_m}(h) := \left( \langle h, f_1 \rangle, \ldots, \langle h, f_m \rangle \right).
\]

Since $\mathcal{H}_0^1$ is a separable Hilbert space, let $\{g_n\}_{n \ge 1}$ be a countable dense subset (we may assume $g_n \in C_0^\infty(0,1)$ without loss of generality). The metric topology on $\mathcal{H}^{-1}$ is generated by sets of the form:
\[
B_{\mathcal{H}^{-1}}(h_0, \varepsilon) = \bigcup_{k \ge 1} \bigcap_{n \ge 1} \mathcal{E}_{g_n}^{-1} \left( -\infty, \langle h_0, g_n \rangle + \varepsilon - \frac{1}{k} \right].
\]
This implies that the Borel $\sigma$-algebra is generated by evaluations against smooth test functions:
\begin{align}
\label{BorelSobolevDual}
\mathcal{B}_{\mathcal{H}^{-1}} = \sigma\left( \mathcal{E}_{f_1, \ldots, f_m} : f_1, \ldots, f_m \in C_0^\infty(0,1),\ m \ge 1 \right).
\end{align}

Now, suppose $P$ is any weak limit of the sequence of laws $P_N$ of $\mathcal V_N$ in $\mathcal{H}^{-d}$. Then, there exists a subsequence $N_m \to \infty$ such that
\[
P_{N_m} \Rightarrow P \quad \text{in } \mathcal{H}^{-d}.
\]
By the Continuous Mapping Theorem, for any finite collection $f_1, \ldots, f_l \in C_0^\infty(0,1)$,
\[
P_{N_m} \circ \mathcal{E}_{f_1, \ldots, f_l}^{-1} \Rightarrow P \circ \mathcal{E}_{f_1, \ldots, f_l}^{-1}.
\]
From the previous steps we have shown that the left-hand side converges to a multivariate Gaussian:
\[
P_{N_m} \circ \mathcal{E}_{f_1, \ldots, f_l}^{-1} \Rightarrow \mathcal{N}(0, \Sigma),
\]
where $\Sigma$ is the covariance structure given explicitly in \eqref{covarianceform}. Therefore, we conclude:
\begin{align}
\label{EVecprolim1}
P \circ \mathcal{E}_{f_1, \ldots, f_l}^{-1} \sim \mathcal{N}(0, \Sigma).
\end{align}

The same argument holds for any other limit point $Q$, so:
\begin{align}
\label{EVecprolim2}
Q \circ \mathcal{E}_{f_1, \ldots, f_l}^{-1} \sim \mathcal{N}(0, \Sigma).
\end{align}

Finally, since the Borel $\sigma$-algebra is generated by $\{ \mathcal{E}_g : g \in C_0^\infty(0,1) \}$ (see \eqref{BorelSobolevDual}) and all finite-dimensional marginals agree, the $\pi$-$\lambda$ Theorem implies that $P = Q$. Hence, the limiting distribution is unique, and the sequence $\mathcal V_N$ converges in law to a mean-zero Gaussian element of $\mathcal{H}^{-d}$.

This completes the proof.

\end{proof}

\begin{remark}\label{EVecproTh1.GaussianLimit}
The analysis remains valid for any test function \( g \in C[0,1] \). Moreover, Theorem \ref{EVecProMain} implies that when \( k = 1 \) and \( g \in C[0,1] \) is fixed, we have the following central limit theorem:
\[
\sqrt{N}\theta\left(\underline{g}^\prime v_1 - \mathbb{E}(\underline{g}^\prime v_1)\right) \Rightarrow \mathcal{N}(0, \sigma^2),
\]
where the asymptotic variance \( \sigma^2 \) is given, using \eqref{EVecproTh1.5}, \eqref{EVecproTh1.6}, and \eqref{EVecproTh1.7}, by:
\begin{align*}
\sigma^2 =\;& 2\left(\int_0^1 g(x)h(x)\,dx\right)^2 \int_0^1\int_0^1 h^2(x)h^2(y)f(x,y)\,dx\,dy \\
& + \int_0^1\int_0^1 g^2(x) h^2(y) f(x,y)\,dx\,dy \\
& + \int_0^1\int_0^1 g(x)g(y)h(x)h(y)f(x,y)\,dx\,dy \\
& - 4\left(\int_0^1 g(x)h(x)\,dx\right)\left(\int_0^1\int_0^1 h(x)g(x)h^2(y)f(x,y)\,dx\,dy\right).
\end{align*}

In particular, in the special case where \( h \equiv 1 \) and \( f \equiv 1 \), this simplifies to:
\[
\sigma^2 = \int_0^1 g^2(x)\,dx - \left(\int_0^1 g(x)\,dx\right)^2,
\]
which is simply the variance of \( g \) under the uniform measure on \([0,1]\).
\end{remark}

\section{Martingale CLT}\label{section:MartingaleCLT}
The next result establishes a central limit theorem for the quadratic form $e_i^\prime W^2 e_i$, which plays a crucial role in the proof of eigenvector alignment. Beyond its immediate application, this theorem is of independent interest: it provides a clean example of the martingale CLT applied to structured quadratic forms of random matrices with inhomogeneous variance profiles. Moreover, the argument generalizes naturally to higher matrix powers $W^m$, and we expect this technique to have further uses in analyzing functionals of inhomogeneous random graphs and related spectral statistics.

\begin{theorem}\label{Th.MartingaleCLT}
\textbf{Martingale Central Limit Theorem.} For all $i$ such that $1 \leq i \leq k$, we have
\[
\frac{e_i^\prime W^2 e_i - \mathbb{E}(e_i^\prime W^2 e_i)}{\sqrt{N}} \Rightarrow  N\bigg(0,\sigma_2^2\bigg)\,,
\]
where 
\[
\sigma_2^2 = 2\int_{x,y \in [0,1]^3} h_i^2(x) f(x,y) f(y,z) h_i^2(z)\,dx\,dy\, dz.
\]
\end{theorem}

\begin{proof}[Proof of Theorem \ref{Th.MartingaleCLT}]

We fix an index $i$ with $1 \le i \le k$, and for notational simplicity denote $e = e_i$ and $h = h_i$.  Our goal is to study the fluctuations of the quadratic form
\[
T_N \;=\; e' W^2 e \;=\; e_i' W^2 e_i.
\]
We begin by decomposing $T_N$ into sums that reveal a martingale structure.  Expanding $e^\prime W^2 e$ in terms of the matrix entries yields
\[
\begin{aligned}
T_N 
 &= \frac{2}{N}\sum_{r=1}^N\sum_{\substack{p<r\\ q\neq p,r}} W_{pq}W_{qr}\,h\Bigl(\frac{p}{N}\Bigr)\,h\Bigl(\frac{r}{N}\Bigr) \\
 &\quad + \frac{2}{N}\sum_{p=1}^N\sum_{\substack{r=1 \\ r\neq p}}^N W_{pp}W_{pr}\,h\Bigl(\frac{p}{N}\Bigr)\,h\Bigl(\frac{r}{N}\Bigr) \\
 &\quad + \frac{1}{N}\sum_{p,q=1}^N W_{pq}^2\,h^2\Bigl(\frac{p}{N}\Bigr).
\end{aligned}
\]
Here we have used symmetry of the summation to combine the terms $W_{pq}W_{qr}$ for $p<r$ with those for $r<p$, hence the factor $2$.  We denote these three sums by $M_N$, $A_N$, and $B_N$ respectively, so that
\begin{equation}\label{eq:T_N}
T_N =\frac{2}{N}M_N + \frac{2}{N}A_N + \frac{1}{N}B_N.
\end{equation}
Our strategy is to show that $M_N$ is the dominant term contributing to Gaussian fluctuations, while the terms $A_N$ and $B_N$ are of smaller order.  Thus we focus on analyzing $M_N$ first.

\paragraph{\bf Martingale Decomposition of \(M_N\)}

Recall that
\[
M_N = \sum_{r=1}^N \sum_{\substack{p<r\\ q\neq p,r}} W_{pq}W_{qr}\,h\Bigl(\frac{p}{N}\Bigr)\,h\Bigl(\frac{r}{N}\Bigr).
\]
To expose its martingale structure, we introduce partial sums.  For each $r=0,1,\dots,N$, define
\[
M_{N,r} = \sum_{u=1}^r \sum_{\substack{p<u\\ q\neq p,u}} W_{pq}W_{qu}\,h\Bigl(\frac{p}{N}\Bigr)\,h\Bigl(\frac{u}{N}\Bigr),
\]
with the convention $M_{N,0}=0$.  Notice that $M_{N,r}$ accumulates the contributions up to the index $u=r$.  In particular, $M_{N,N} = M_N$.  Writing $M_N$ as a telescoping sum,
\[
M_N = M_{N,N} = \sum_{r=1}^N (M_{N,r} - M_{N,r-1}).
\]
We will show that the sequence $(M_{N,r})_{r=0}^N$ is a martingale with respect to an increasing filtration.

 Let $\mathcal{F}_r = \sigma(W_{pq}: 1 \le p \le q \le r)$ be the sigma-field generated by the matrix entries in the first $r$ rows and columns.  In particular, $\mathcal{F}_{r-1}$ contains all randomness up to index $r-1$.
Observe that $M_{N, r}-M_{N, r-1}$ can be written as
$$
M_{N, r}-M_{N, r-1}=\sum_{p<r} \sum_{q \leq r-1, q \neq p} W_{p q} W_{q r} h(p / N) h(r / N) .
$$
Here, for each term, $W_{p q}$ (with $p, q<r$ ) is $\mathcal{F}_{r-1}$-measurable, whereas $W_{q r}$ (which involves the new index $r$ ) is independent of $\mathcal{F}_{r-1}$ and has mean zero. Therefore $\mathbb{E}\left[M_{N, r}-M_{N, r-1} \mid \mathcal{F}_{r-1}\right]=0$, so $\left(M_{N, r}, \mathcal{F}_r\right)_{r=0}^N$ is a martingale.

\paragraph{\bf Conditional Variance Computation.}

We next compute the conditional second moment of the martingale increments.  For each $r=1,\dots,N$, consider
\[
\mathbb{E}\bigl[(M_{N,r} - M_{N,r-1})^2 \,\big|\; \mathcal{F}_{r-1}\bigr].
\]
Expanding the square and noting that the $W_{pq}$ are independent with mean zero, we obtain
\begin{equation}\label{eq:condvar}
\mathbb{E}[(M_{N,r} - M_{N,r-1})^2 \mid \mathcal{F}_{r-1}]
 = \mathrm{Var}\Bigl(\sum_{\substack{p<r\\ q\neq p,r}} W_{pq}W_{qr}h\Bigl(\frac{p}{N}\Bigr)h\Bigl(\frac{r}{N}\Bigr)\Bigm| \mathcal{F}_{r-1}\Bigr).
\end{equation}
This variance can be split into two parts depending on whether the index $q$ is at most $r-1$ or strictly greater than $r-1$.  Decomposing, we get
\[
\begin{aligned}
\eqref{eq:condvar}
 &= \mathrm{Var}\Bigl(\sum_{\substack{p<r,\,q\neq p\\ q\le r-1}}W_{pq}W_{qr}h\bigl(\tfrac{p}{N}\bigr)h\bigl(\tfrac{r}{N}\bigr)\Bigm|\mathcal{F}_{r-1}\Bigr)\\
 &\quad + \mathrm{Var}\Bigl(\sum_{\substack{p<r,\,q\neq p\\ q> r-1}}W_{pq}W_{qr}h\bigl(\tfrac{p}{N}\bigr)h\bigl(\tfrac{r}{N}\bigr)\Bigm|\mathcal{F}_{r-1}\Bigr)\\
 &\quad + 2\,\mathrm{Cov}\Bigl(\sum_{\substack{p<r,\,q\neq p\\ q\le r-1}}W_{pq}W_{qr}h\bigl(\tfrac{p}{N}\bigr)h\bigl(\tfrac{r}{N}\bigr),\,
      \sum_{\substack{p<r,\,q\neq p\\ q> r-1}}W_{pq}W_{qr}h\bigl(\tfrac{p}{N}\bigr)h\bigl(\tfrac{r}{N}\bigr)\Bigm|\mathcal{F}_{r-1}\Bigr).
\end{aligned}
\]

Splitting the sum into the cases $q<r$ and $q>r$ (where for $q<r$, $W_{pq}$ is $\mathcal F_{r-1}$–measurable and $W_{qr}$ is independent, and for $q>r$, both $W_{pq}$ and $W_{qr}$ are independent of $\mathcal F_{r-1}$) and computing variances, one obtains:
$$
\begin{aligned}
\mathbb{E}\left[\left(M_{N, r}-M_{N, r-1}\right)^2 \mid \mathcal{F}_{r-1}\right]=h^2 & \left(\frac{r}{N}\right)\left[\sum_{\substack{q=1 \\
r-1}}^{r-1} \sum_{\substack{p=1 \\
p \neq q}}^{r-1} W_{p q}^2 h^2\left(\frac{p}{N}\right) f\left(\frac{q}{N}, \frac{r}{N}\right)\right. \\
& +\sum_{q=1}^{r-1} \sum_{\substack{p, p^{\prime}=1 \\
p, p^{\prime} \neq q}}^{r-1} W_{p q} W_{p^{\prime} q} h\left(\frac{p}{N}\right) h\left(\frac{p^{\prime}}{N}\right) f\left(\frac{q}{N}, \frac{r}{N}\right) \\
& \left.+\sum_{p=1}^{r-1} \sum_{q=r+1}^N h^2\left(\frac{p}{N}\right) f\left(\frac{p}{N}, \frac{q}{N}\right) f\left(\frac{q}{N}, \frac{r}{N}\right)\right] .
\end{aligned}
$$
For convenience, define
\[
V_N \;=\; N^{-3}\sum_{r=1}^N \mathbb{E}\bigl[(M_{N,r} - M_{N,r-1})^2 \mid \mathcal{F}_{r-1}\bigr].
\]
We will show that $V_N$ converges in probability to a constant, which will be the limiting variance of the martingale.

\paragraph{\bf Asymptotic Variance of the Martingale.}
We first compute the expected value of $V_N$.  Taking expectation and noting that the $W_{pq}$ in $V_N$ appear in squared terms, one obtains
\[
\begin{aligned}
\mathbb{E}[V_N] &= N^{-3} \sum_{r=1}^N \Bigl[
   \sum_{\substack{p<r\\ q<r,\,q\neq p}} h^2\Bigl(\tfrac{p}{N}\Bigr)h^2\Bigl(\tfrac{r}{N}\Bigr)\,f\Bigl(\tfrac{p}{N},\tfrac{q}{N}\Bigr)\,f\Bigl(\tfrac{q}{N},\tfrac{r}{N}\Bigr) \\
&\quad + \sum_{\substack{p<p'<r\\ q<r,\,q\neq p,p'}} h\Bigl(\tfrac{p}{N}\Bigr)h\Bigl(\tfrac{p'}{N}\Bigr)h^2\Bigl(\tfrac{r}{N}\Bigr)\,f\Bigl(\tfrac{p}{N},\tfrac{q}{N}\Bigr)\,f\Bigl(\tfrac{q}{N},\tfrac{r}{N}\Bigr) \\
&\quad + \sum_{\substack{p<r\\ q>r}} h^2\Bigl(\tfrac{p}{N}\Bigr)h^2\Bigl(\tfrac{r}{N}\Bigr)\,f\Bigl(\tfrac{p}{N},\tfrac{q}{N}\Bigr)\,f\Bigl(\tfrac{q}{N},\tfrac{r}{N}\Bigr) \Bigr].
\end{aligned}
\]
As $N\to\infty$, Riemann sum approximations show that this sum converges to an integral
\[
\mathbb{E}[V_N] \;\longrightarrow\; \frac12 \int_{[0,1]^3} h^2(x)\,f(x,y)\,f(y,z)\,h^2(z)\,dx\,dy\,dz.
\]
The factor $1/2$ arises from accounting for the two symmetric sums involving $p<r<q$ and $p<q<r$.

We now verify that $\Var(V_N)$ tends to zero. First, note that for $p_1<q_1$ and $p_2<p_3$ with $q_2\neq p_2,p_3$, the covariance
\[
\Cov\bigl(W_{p_1q_1}^2,\;W_{p_2q_2}W_{p_3q_2}\bigr) = 0,
\]
since distinct entries of $W$ are independent. Hence, by expanding $\Var(V_N)$ one obtains
\begin{align*}
\frac{N^6}{4}\Var(V_N)
&= \Var\Bigl(\sum_{r=1}^N \sum_{p<q<r} W_{pq}^2 \,h^2\bigl(\tfrac{p}{N}\bigr)\,h^2\bigl(\tfrac{r}{N}\bigr)\,f\bigl(\tfrac{q}{N},\tfrac{r}{N}\bigr)\Bigr)\\
&+ \Var\Bigl(\sum_{r=1}^N \sum_{\substack{p<p'<r\\ q<r,\,q\neq p,p'}} W_{pq}W_{p'q}\,h\bigl(\tfrac{p}{N}\bigr)h\bigl(\tfrac{p'}{N}\bigr)\,h^2\bigl(\tfrac{r}{N}\bigr)\,f\bigl(\tfrac{q}{N},\tfrac{r}{N}\bigr)\Bigr).
\end{align*}
We bound each term separately. For the first term, re-indexing gives
\[
\Var\Bigl(\sum_{r=1}^N\sum_{p<q<r} W_{pq}^2 \,h^2\bigl(\tfrac{p}{N}\bigr)\,h^2\bigl(\tfrac{r}{N}\bigr)\,f\bigl(\tfrac{q}{N},\tfrac{r}{N}\bigr)\Bigr)
= \Var\Bigl(\sum_{1\le p<q<N} h^2\bigl(\tfrac{p}{N}\bigr)\,W_{pq}^2\Bigl(\sum_{r=q+1}^N h^2\bigl(\tfrac{r}{N}\bigr)f\bigl(\tfrac{q}{N},\tfrac{r}{N}\bigr)\Bigr)\Bigr).
\]
Using $\E[W_{pq}^4]\le 4^{4B}$ and boundedness of $f,h$, we obtain
\begin{align}
&\Var\Bigl(\sum_{r=1}^N\sum_{p<q<r} W_{pq}^2 \,h^2\bigl(\tfrac{p}{N}\bigr)\,h^2\bigl(\tfrac{r}{N}\bigr)\,f\bigl(\tfrac{q}{N},\tfrac{r}{N}\bigr)\Bigr)\nonumber\\
&\quad\le 4^{4B}\|f\|_{\infty}^4\|h\|_{\infty}^8\sum_{1\le p<q<N}(N-q)^2
= 4^{4B}\|f\|_{\infty}^4\|h\|_{\infty}^8\sum_{q=1}^{N-1}(q-1)(N-q)^2\nonumber\\
&\quad= O(N^4).
\label{MCLT.Var1}
\end{align}
For the second term, one checks that non-zero contributions arise only when indices match suitably. Note that, for $p_1,p_2,p_3,p_4,q,q'$ such that $p_1<p_2$,$p_3 < p_4$, $q \ne p_1,p_2$ and $q' \ne p_3,p_4$, we can have a non-zero contribution from $\Cov(W_{p_1q}W_{p_2q},W_{p_3q'}W_{p_4q'})$ when 
\[
p_1 = p_3, p_2 = p_4, q= q'\, \quad \text{or}\, \quad p_1 = p_4, p_3=p_2 , q=q'\,.
\] Arguing similarly yields
\begin{align}
&\Var\Bigl(\sum_{r=1}^N \sum_{\substack{p<p'<r\\ q<r,\,q\neq p,p'}} W_{pq}W_{p'q}\,h\bigl(\tfrac{p}{N}\bigr)h\bigl(\tfrac{p'}{N}\bigr)\,h^2\bigl(\tfrac{r}{N}\bigr)\,f\bigl(\tfrac{q}{N},\tfrac{r}{N}\bigr)\Bigr)\nonumber\\
&\quad\le 2\,\|f\|_{\infty}^4\|h\|_{\infty}^8 \sum_{p<p'<N,\;q<N}(N-\max\{p',q\})^2\nonumber\\
&\quad= 2\,\|f\|_{\infty}^4\|h\|_{\infty}^8 N^2 \sum_{p<p'<N,\;q<N}\Bigl(1-\frac{p'\vee q}{N}\Bigr)^2\nonumber\\
&\quad\sim N^5 \int_0^1\!\!\int_0^1\!\!\int_0^1 (1-\max\{y,z\})^2\one(x<y)\,dx\,dy\,dz 
= O(N^5).
\label{MCLT.Var2}
\end{align}
Combining \eqref{MCLT.Var1} and \eqref{MCLT.Var2}, we conclude that
\[
\lim_{N\to\infty} \Var(V_N) = 0.
\]


Hence $V_N$ converges in probability to its mean limit:
\[
V_N \;\xrightarrow{P}\; \frac12 \int_{[0,1]^3} h^2(x)\,f(x,y)\,f(y,z)\,h^2(z)\,dx\,dy\,dz.
\]

This identifies the limiting conditional variance of the martingale increments.

\paragraph{\bf Verification of the Lindeberg Condition}

To apply the martingale central limit theorem, we must verify a Lindeberg-type condition.  Set
\[
Y_{N,r} \;=\; N^{-3/2}\,(M_{N,r} - M_{N,r-1}),\quad r=1,\dots,N,
\]
and extend $Y_{N,r}=0$ for $r>N$.  Then
\[
\sum_{r=1}^\infty Y_{N,r} = \frac{M_N}{N^{3/2}}.
\]
We already have from the martingale property that $\mathbb{E}[Y_{N,r}|\mathcal{F}_{r-1}]=0$ and
\begin{align*}
&\sum_{r=1}^\infty \mathbb{E}[Y_{N,r}^2|\mathcal{F}_{r-1}] \\
&= N^{-3}\sum_{r=1}^N \mathbb{E}[(M_{N,r} - M_{N,r-1})^2|\mathcal{F}_{r-1}] 
\xrightarrow{P} 
\frac12\int_{[0,1]^3} h^2(x)\,f(x,y)\,f(y,z)\,h^2(z)\,dx\,dy\,dz.
\end{align*}
Thus the triangular array $(Y_{N,r})$ satisfies the martingale variance condition.  It remains to check the Lindeberg condition:
\[
\sum_{r=1}^\infty \mathbb{E}\bigl[Y_{N,r}^2 \,\mathbf{1}(|Y_{N,r}|>\varepsilon)\bigr] \;\to\; 0 
\quad\text{for each }\varepsilon>0.
\]
Observe that by Minkowski's inequality we have
\[
\mathbb{E}^{1/4}[(M_{N,r}-M_{N,r-1})^4]
\;\le\; \mathbb{E}^{1/4}[Z_{N1}^4] \;+\; \mathbb{E}^{1/4}[Z_{N2}^4],
\]
where we split the sum in $M_{N,r}-M_{N,r-1}$ into the two parts according to $q\le r-1$ or $q>r-1$:
\[
Z_{N1} \;=\; \sum_{\substack{p<r\\ q\le r-1,\,q\neq p}} W_{pq}W_{qr}\,h\Bigl(\tfrac{p}{N}\Bigr)h\Bigl(\tfrac{r}{N}\Bigr),
\quad
Z_{N2} \;=\; \sum_{\substack{p<r\\ q>r-1}} W_{pq}W_{qr}\,h\Bigl(\tfrac{p}{N}\Bigr)h\Bigl(\tfrac{r}{N}\Bigr).
\]
One shows by combinatorial estimates and the bounded moments of $W_{pq}$ that
\[
\mathbb{E}[Z_{N1}^4] = O(N^4), 
\qquad 
\mathbb{E}[Z_{N2}^4] = O(N^4).
\]
It follows that $\mathbb{E}[(M_{N,r}-M_{N,r-1})^4] = O(N^4)$ uniformly in $r$.  Consequently, for any fixed $\varepsilon>0$,
\[
\begin{aligned}
\sum_{r=1}^{N} \mathbb{E}\bigl[Y_{N,r}^2\,\mathbf{1}(|Y_{N,r}|>\varepsilon)\bigr]
&\le \varepsilon^{-2} \sum_{r=1}^{N} \mathbb{E}[Y_{N,r}^4] 
= \frac{\varepsilon^{-2}}{N^6}\sum_{r=1}^N \mathbb{E}[(M_{N,r}-M_{N,r-1})^4] \\
&= \frac{\varepsilon^{-2}}{N^6} \times O(N^5) \;\longrightarrow\; 0
\quad\text{as } N\to\infty.
\end{aligned}
\]
Thus the Lindeberg condition holds, and we can apply the martingale central limit theorem, see for example \cite[Theorem 35.12]{billingsley2017probability}.  We conclude that
\[
\sum_{r=1}^N Y_{N,r} \;\;=\;\; \frac{M_N}{N^{3/2}}
\;\xrightarrow{d}\; N\Bigl(0,\sigma'^2\Bigr),
\]
where
\[
\sigma'^2 = \frac{1}{2}\int_{[0,1]^3} h^2(x)\,f(x,y)\,f(y,z)\,h^2(z)\,dx\,dy\,dz.
\]
This establishes that the scaled martingale term $M_N$ converges in distribution to a mean-zero normal with variance $\sigma'^2$.

\paragraph{\bf Negligibility of Remaining Terms.}

It remains to show that the other terms $A_N$ and $B_N$ in the decomposition \eqref{eq:T_N} are of smaller order than $M_N$ and do not affect the limit.  Recall
\[
A_N = \sum_{p=1}^N\sum_{r\neq p}W_{pp}W_{pr}\,h\Bigl(\tfrac{p}{N}\Bigr)h\Bigl(\tfrac{r}{N}\Bigr), 
\qquad
B_N = \sum_{p,q=1}^N W_{pq}^2\,h^2\Bigl(\tfrac{p}{N}\Bigr).
\]
One can directly compute (or bound) their variances.  Using the boundedness of moments and summing over the index ranges yields
\[
\mathrm{Var}(A_N) = O(N^2),
\qquad 
\mathrm{Var}(B_N) = O(N^2).
\]
Since we are normalizing by $N^{3/2}$ for the central limit theorem, these terms are negligible: $\mathrm{Var}(A_N/N^{3/2}) = O(N^{-1}) \to 0$ and similarly for $B_N$.  Hence
\[
\frac{2}{N}A_N + \frac{1}{N}B_N
= o_{P}(N^{1/2}),
\]
in the sense that it vanishes after scaling by $\sqrt{N}$.

Putting everything together, from the decomposition we have
\[
\frac{T_N - \mathbb{E}[T_N]}{\sqrt{N}}
= \frac{2M_N}{N^{3/2}} + \frac{2(A_N - \mathbb{E}[A_N])}{N^{3/2}} + \frac{B_N - \mathbb{E}[B_N]}{N^{3/2}}.
\]
We have shown that $2M_N/N^{3/2}$ converges in distribution to $N(0,4\sigma'^2)$, while the remaining terms vanish in probability.  Therefore, by Slutsky’s theorem,
\[
\frac{T_N - \mathbb{E}[T_N]}{\sqrt{N}}
\;\xrightarrow{d}\; N\Bigl(0,4\sigma'^2\Bigr).
\]
Recalling the definition of $\sigma'^2$, we see that $4\sigma'^2$ equals
\[
\sigma_2^2 = 2\int_{[0,1]^3} h^2(x)\,f(x,y)\,f(y,z)\,h^2(z)\,dx\,dy\,dz,
\]
which matches the variance stated in Theorem \ref{Th.MartingaleCLT}.  This completes the proof of the Martingale Central Limit Theorem for the quadratic form $e_i' W^2 e_i$.

\end{proof}

\bibliographystyle{apalike}
\bibliography{Bibliography}

\appendix 
\section{Eigenvalue CLT: Proof of \ref{Th.EigenValueCLT}} \label{appendixA}

The proof of Theorem \ref{Th.EigenValueCLT} follows along the lines of \cite{chakrabarty2020eigenvalues} so we provide the details in the appendix. Some of the estimates are crucial in the results about the eigenvectors so we need to reprove some of the crucial lemmas from \cite{chakrabarty2020eigenvalues}.
Let us denote
\[
Z_{j,l,n}:=e_j^\prime W^n e_l\,,1\le j,l\le k\,,n\in \mathbb N \cup \{0\}\,,
\]
Also, let $Y_n$ be the $k\times k$ matrix with entries
\begin{equation}\label{Defn.Yn}
Y_n(j,l) = \sqrt{\alpha_j\alpha_l}\theta Z_{j,l,n}, \qquad 1\le j,l\le k,.
\end{equation}
\begin{remark}\label{EVa.YN.normbound}
An immediate observation from the definition of $Y_n$ is that
\[
    \|Y_n\| = O_{hp}(\theta N^{\frac{n}{2}})\,.
\]
\end{remark}

We now proceed towards the proof of Theorem \ref{Th.EigenValueCLT} through several steps, organized as Lemmas \ref{Le.S1}–\ref{Le.S5}. The key idea is to break the infinite sum in \eqref{K:inf} at $L$, where $L=\lfloor\log N\rfloor$.

\begin{lemma}\label{Le.S1}
\[
\mu=\lambda_i\left(\sum_{n=0}^L\mu^{-n}Y_n\right)+o_{hp}(1)\,.
\]  
\end{lemma}

\begin{proof}
Using the definitions of $K$ and $Y_n$ from \eqref{eq:VMatrix} and \eqref{Defn.Yn}, respectively, it is immediate that
\[
K=\one(\|W\|<\mu)\sum_{n=0}^\infty\mu^{-n}Y_n\,.
\]
Lemma \ref{EVa.combined} gives that with high probability,
\[\mu = \lambda_i(V) = \lambda_i(\sum_{n=0}^\infty\mu^{-n}Y_n).\]
It then suffices to show
\begin{equation}\label{Le.S1.tail}
\Big\|\sum_{n=L+1}^\infty \mu^{-n}Y_n\Big\| = o_{hp}(1),.
\end{equation}
Consequently, Remark \ref{EVa.YN.normbound} implies that there exist constants $C, C'>0$ such that
\begin{align*}
\Big\|\sum_{n=L+1}^\infty \mu^{-n}Y_n\Big\| &\le \sum_{n=L+1}^\infty |\mu|^{-n}\,\|Y_n\| \\
&= O_{hp}\Big(\theta \sum_{n=L+1}^{\infty}\Big(\frac{C\sqrt{N}}{\theta}\Big)^n\Big) \\
&= O_{hp}\Big(\frac{(C')^{L+1}\sqrt{N}}{(\log N)^{L\xi}}\Big) \\
&= o_{hp}(1),
\end{align*}
where $\xi>0$ is as in Assumption \ref{assump:A1}. This verifies \eqref{Le.S1.tail} and completes the proof.
\end{proof}

In the next step, we approximate $Y_n$ by $\E(Y_n)$ for $2\leq n \leq L$.

\begin{lemma}\label{Le.S2}
\[
\mu=\lambda_i\left(Y_0 + \frac{Y_1}{\mu} +  \sum_{n=2}^L\mu^{-n}\mathbb E(Y_n)\right)+o_{hp}(1)\,.
\]
\end{lemma}

\begin{proof}
In view of Lemma \ref{Le.S1}, it suffices to show that
\begin{equation}\label{Le.S2.Target}
\sum_{n=2}^L \mu^{-n}\,\|Y_n - \E(Y_n)\| = o_{hp}(1).
\end{equation}
To this end, define the following events (depending on $N$):
\[E = \bigcap_{2\le n\le L,1\le j,l\le k}\left[\left|Z_{j,l,n}-\mathbb E(Z_{j,l,n})\right|\le N^{(n-1)/2}(\log N)^{n\xi/4}\right],\] 
\[E' = \left[\frac{\mu}{\theta} \ge \frac{\alpha_i}{2} \right],\]
\[ E'' = \left[\sum_{n=2}^L\mu^{-n}\|Y_n-\mathbb E(Y_n)\|\le C_2\frac{\sqrt{N}}{\theta}(\log N)^{\xi/2}\right],\] 

where $C_2>0$ is a constant. We will show that for all large $N$,
\[E \cap E' \subset E''\,.\]

Assume that the event $E \cap E'$ occurs. Then using the definitions above, we have
\begin{align*}
\sum_{n=2}^L \mu^{-n}\|Y_n - \E(Y_n)\|
&\le \theta\alpha_1 \sum_{n=2}^L \mu^{-n} \sum_{1 \le j,l \le k} \big|Z_{j,l,n} - \E(Z_{j,l,n})\big|
&&\text{(\tt{spectral norm $\le$ sum of entries})}\\
&\le \frac{\alpha_1}{\alpha_i}k^2 \sum_{n=2}^L 2^n (\theta\alpha_i)^{-n+1} N^{(n-1)/2} (\log N)^{n\xi/4}
&&\text{(\tt{by the definition of $E$})}\\
&\le \frac{\alpha_1}{2\alpha_i}k^2 \sum_{n=2}^\infty \Big(\frac{2\sqrt{N}}{\theta\alpha_i}\Big)^{n-1} (\log N)^{n\xi/4} \\
&\le\frac{\alpha_1}{2\alpha_i}k^2 \left[1 - \frac{2\sqrt{N}}{\theta\alpha_i}(\log N)^{\xi/4}\right]^{-1} \frac{2\sqrt{N}}{\theta\alpha_i} (\log N)^{\xi/2}.
\end{align*}
In view of Assumption \ref{assump:A1}, the above inclusion $E\cap E' \subset E''$ holds for all large $N$ (by choosing $C_2$ appropriately).

Moreover, Lemma \ref{ConcentrationBound} (a concentration bound for sums of independent entries of $W$) yields that for all $1 \le j,l \le k$ and $2 \le n \le L$, uniformly in $N$,
\[ P\left(\left|Z_{j,l,n}-\mathbb E(Z_{j,l,n})\right|>N^{(n-1)/2}(\log N)^{n\xi/4}\right) = O\left(e^{-(\log N)^{\eta_1}}\right)
\]
for some $\eta_1>0$. Using the complement of the event inclusion and Theorem \ref{Th.probedge}, we thus obtain an $\eta_2>1$ such that
\begin{align*}
&P\Big(\sum_{n=2}^L \mu^{-n}\|Y_n-\E(Y_n)\| > C_2\frac{\sqrt{N}}{\theta}(\log N)^{\xi/2}\Big)\\\
=& \sum_{n=2}^L P\Big(\big|Z_{j,l,n}-\E(Z_{j,l,n})\big| > N^{(n-1)/2}(\log N)^{n\xi/4}\Big) + P\Big( \frac{\mu}{\theta} < \frac{\alpha_i}{2}\Big)\\
=& O\Big(\log N  e^{-(\log N)^{\eta_1}}\Big) + O\Big(e^{-(\log N)^{\eta_2}}\Big)\\
=& o\Big(e^{-(\log N)^{(1+\eta_1)/2}}\Big) + o\Big(e^{-(\log N)^{(1+\eta_2)/2}}\Big).
\end{align*}
Hence \eqref{Le.S2.Target} is established. This completes the proof.
\end{proof}

In the next step, we remove the randomness entirely from
\[
\sum_{n=2}^L\mu^{-n}\mathbb E(Y_n)\,.
\]
by replacing the factor $\mu$ suitably with a deterministic value.

\begin{lemma}\label{Le.S3}
For $N$ large, the deterministic equation
\begin{equation}\label{Le.S3.deteqn}
x = \lambda_i\Big(\sum_{n=0}^L x^{-n}\E(Y_n)\Big), \qquad x>0,,
\end{equation}
has a solution $\tilde{\mu}$, and moreover
\begin{equation}\label{Le.S3.detlim}
\lim_{N \to \infty}\frac{\tilde{\mu}}{\theta} = \alpha_i.
\end{equation}
\end{lemma}

\begin{proof}
Consider the function (for each fixed $N$) $h:(0,\infty)\to\bbr$ by
\[
h(x)=\lambda_i\left(\sum_{n=0}^Lx^{-n}\mathbb E(Y_n)\right)\,.
\]
We first show that for every fixed $x>0$,
\begin{equation}\label{Le.S3.detlim.interm}
\lim_{N\to\infty}\frac{h(x\theta)}{\theta} = \alpha_i.
\end{equation}
To see this, note that $\E(Y_1) = 0$ (since $\E[W]=0$ under Assumption \ref{assump:A1}). Thus
\begin{align*}
h\left(x\theta\right)=\lambda_i\left(\mathbb E(Y_0)+\sum_{n=2}^L(x\theta)^{-n}\mathbb E(Y_n)\right)\,.
\end{align*}

By definition of $Y_0$,
\[
Y_0(j,l)=\theta\sqrt{\alpha_j\alpha_l}\,e_j^\prime e_l\,,1\le j,l\le k\,.
\]
Then Assumption \ref{assump:A2} implies
\begin{equation}\label{Le.S3.output1}
\lim_{N\to\infty}\theta^{-1}\E(Y_0) = \mathrm{Diag}(\alpha_1,\alpha_2,\ldots,\alpha_k).
\end{equation}
Moreover, by a moment bound (Lemma \ref{ExpectationBound}), there exists a constant $D>0$ (not depending on $N$) such that for all $2 \le n \le L$,
\begin{equation}\label{EYn.norm}
\|\E(Y_n)\| \le \theta\big(D\sqrt{N}\big)^{n}\,.
\end{equation}
Using the above bound, we obtain
\begin{align*}
\|\sum_{n=2}^L (x\theta)^{-n}\E(Y_n)\| &\le \theta \sum_{n=2}^\infty \Big(\frac{D\sqrt{N}}{x\theta} \Big)^n \\
&=o(1) \qquad \text{as } N \to \infty,
\end{align*}
uniformly for $x>0$. The last line, along with \eqref{Le.S3.output1}, implies
\[
\lim _{N \rightarrow \infty} \frac{1}{\theta}\left(\E\left(Y_0\right)+\sum_{n=2}^L(x \theta)^{-n} \E\left(Y_n\right)\right)=\operatorname{Diag}\left(\alpha_1, \ldots, \alpha_k\right),\]
from which \eqref{Le.S3.detlim.interm} follows. With \eqref{Le.S3.detlim.interm} at our disposal, fix any $0 < \delta < \alpha_i$. Then
\[
\lim _{N \rightarrow \infty} \frac{\theta\left(\alpha_i+\delta\right)-h\left(\theta\left(\alpha_i+\delta\right)\right)}{\theta}=\delta,
\]
and hence for all sufficiently large $N$,
\[
h\left(\theta\left(\alpha_i+\delta\right)\right)<\theta\left(\alpha_i+\delta\right).
\]

Similarly, one finds for large $N$ that $h(\theta(\alpha_i - \delta)) > \theta(\alpha_i - \delta)$. This means that for all sufficiently large $N$, the equation $h(x) = x$ (i.e. \eqref{Le.S3.deteqn}) has a solution $x=\tilde{\mu}$ in the interval $[\theta(\alpha_i-\delta)\,,\theta(\alpha_i+\delta)]$. In particular, $\tilde{\mu}$ exists and satisfies
$\theta\left(\alpha_i-\delta\right) \leq \tilde{\mu} \leq \theta\left(\alpha_i+\delta\right)$
for large $N$. Since $\delta>0$ was arbitrary, this yields \eqref{Le.S3.detlim} (uniqueness of the solution $\tilde{\mu}$ is not needed and not asserted).
\end{proof}

In the next step, we estimate the error incurred when approximating $\mu$ by the deterministic $\tilde{\mu}$ obtained in Lemma \ref{Le.S3}.

\begin{lemma}\label{Le.S4}
\[
\mu-\tilde\mu=O_{hp}\left(\frac{\|Y_1\|}{\theta}+1\right)\,.
\]
\end{lemma}

\begin{proof}
Lemmas \ref{Le.S2} and \ref{Le.S3} jointly imply
\begin{align*}
|\mu-\tilde{\mu}|
&= \Big|\lambda_i\Big(Y_0+\mu^{-1}Y_1+\sum_{n=2}^L\mu^{-n}\E(Y_n)\Big) - \lambda_i\Big(\sum_{n=0}^L \tilde{\mu}^{-n}\E(Y_n)\Big)\Big| + o_{hp}(1)\\
&\le \|\mu^{-1}Y_1\| + |\mu-\tilde{\mu}| \sum_{n=2}^L \mu^{-n}\tilde{\mu}^{-n}\|\E(Y_n)\| \sum_{j=0}^{n-1} \mu^j \tilde{\mu}^{n-1-j} + o_{hp}(1)\\
&= |\mu-\tilde{\mu}| \sum_{n=2}^L \mu^{-n}\tilde{\mu}^{-n}\|\E(Y_n)\| \sum_{j=0}^{n-1} \mu^j \tilde{\mu}^{n-1-j} + O_{hp}\Big(\frac{\|Y_1\|}{\theta} + 1\Big)\,.
\end{align*}
Clearly, to obtain the desired result, it suffices to show that the first term on the RHS above is $o_{hp}(1)$. The above display simplifies to
\begin{equation}\label{Le.S4.Target'}
|\mu-\tilde{\mu}|\Bigg[1 - \sum_{n=2}^L \mu^{-n}\tilde{\mu}^{-n}\|\E(Y_n)\| \sum_{j=0}^{n-1} \mu^j \tilde{\mu}^{n-1-j}\Bigg]
= O_{hp}\Big(\frac{\|Y_1\|}{\theta} + 1\Big).
\end{equation}

As in the proof of Lemma \ref{Le.S3}, Lemma \ref{ExpectationBound} ensures that for all $2 \le n \le L$,
\[
\mathbb \|E(Y_n)\| \le \theta\left(D\sqrt{N}\right)^{n}\,,
\]

for some $D>0$. Also, by \eqref{Le.S3.detlim} (from Lemma \ref{Le.S3}), for all sufficiently large $N$,
$$\frac{\alpha_i}{2}<\frac{\tilde{\mu}}{\theta}<\frac{3 \alpha_i}{2}.$$

By Theorem \ref{Th.probedge}, we similarly have w.h.p.
$$\frac{\alpha_i}{2}<\frac{\mu}{\theta}<\frac{3 \alpha_i}{2}.$$

Thus for large $N$, with high probability,
\begin{align*}
\Big|\mu^{-n}\tilde{\mu}^{-n}\sum_{j=0}^{n-1}\mu^j\tilde{\mu}^{n-1-j}\Big|
&\le \theta^{-2n}\Big(\frac{2}{\alpha_i}\Big)^{2n} \sum_{j=0}^{n-1} \theta^{n-1}\Big(\frac{3\alpha_i}{2}\Big)^j \Big(\frac{3\alpha_i}{2}\Big)^{n-1-j} \\
&\le n\Big(\frac{6}{\theta\alpha_i}\Big)^{n+1}.
\end{align*}
Using this and the estimate on $\|\E(Y_n)\|$ provided by Lemma \ref{ExpectationBound}, we get with high probability
\begin{align*}
\sum_{n=2}^L \mu^{-n}\tilde{\mu}^{-n}\|\E(Y_n)\| \sum_{j=0}^{n-1}\mu^j\tilde{\mu}^{n-1-j}
&\le \sum_{n=2}^L n\Big(\frac{6}{\theta\alpha_i}\Big)^{n+1} \theta\big(D\sqrt{N}\big)^{n} \\
&=\Big(\frac{6}{\alpha_i}\Big)^2 \frac{D\sqrt{N}}{\theta} \sum_{n=2}^\infty n \Big(\frac{6D\sqrt{N}}{\theta\alpha_i}\Big)^{n-1} \\
&= o(1),
\end{align*}
as $N \to \infty$. This in conjunction with \eqref{Le.S4.Target'} completes the proof.
\end{proof}


The next step is the most crucial one and will propel us to the penultimate step in proving Theorem \ref{Th.EigenValueCLT}.

\begin{lemma}\label{Le.S5}
$\exists$ a deterministic constant $\bar\mu = \bar\mu_N,$ such that, 
\[
\mu=\bar\mu+\mu^{-1}Y_1(i,i)+o_{hp}\left(\frac{\|Y_1\|}{\theta}+1 \right)\,.
\]    
\end{lemma}

\begin{proof}
With $\tilde{\mu}$ as in Lemma \ref{Le.S3}, consider the deterministic $k\times k$ matrix
\begin{equation}\label{eq:LeS5Xmatrix}
X := \sum_{n=0}^L \tilde{\mu}^{-n}\E(Y_n).
\end{equation}
Using the argument and final estimate from the proof of Lemma \ref{Le.S4}, we have
\begin{align*}
\Big\|X - \sum_{n=0}^L \mu^{-n}\E(Y_n)\Big\|
&\le |\mu-\tilde{\mu}| \sum_{n=2}^L \mu^{-n}\tilde{\mu}^{-n}\|\E(Y_n)\| \sum_{j=0}^{n-1}\mu^j\tilde{\mu}^{,n-1-j} \\
&= o_{hp}\big(|\mu-\tilde{\mu}|\big) \\
&= o_{hp}\Big(\frac{\|Y_1\|}{\theta} + 1\Big).
\end{align*}
Then an application of Lemma \ref{Le.S2} along with Weyl’s inequality gives
\begin{align}
\nonumber
\mu &= \lambda_i\Big(Y_0 + \frac{Y_1}{\mu} + \sum_{n=2}^L \mu^{-n}\E(Y_n)\Big) + o_{hp}(1)\\
\label{Le.S5.approx1}
&=\lambda_i\Big(\frac{Y_1}{\mu} + X\Big) + o_{hp}\Big(\frac{\|Y_1\|}{\theta} + 1\Big).
\end{align}

Define 
\[
H_1=X-X(i,i)I\,,
\]and
\[
H_2=X+\mu^{-1}Y_1-\left(X(i,i)+\mu^{-1}Y_1(i,i)\right)I\,,
\]and
\[
\bar\mu=\lambda_i(X)\,.
\]

By Lemma \ref{Le.S3} (see \eqref{Le.S3.detlim}), $\bar{\mu}$ coincides with $\tilde{\mu}$ and hence
\[
\lim_{N \to \infty}\frac{\bar\mu}{\theta} = \alpha_i\,.
\]
Also, we have
\begin{align*}
\lambda_i\big(\mu^{-1}Y_1 + X\big)
&= \lambda_i(H_2) + X(i,i) + \mu^{-1}Y_1(i,i) \\
&= \lambda_i(H_2) - \lambda_i(H_1) + \bar{\mu} + \mu^{-1}Y_1(i,i).
\end{align*}
Then \eqref{Le.S5.approx1} implies
\[
\mu = \bar\mu + \frac{Y_1(i,i)}{\mu} + \lambda_i(H_2)-\lambda_i(H_1) + o_{hp}\left(\frac{\|Y_1\|}{\theta}+1 \right)
\]
The proof will follow if we can show that
\begin{equation}\label{Le.S5.target}
\lambda_i(H_1) - \lambda_i(H_2) = o_{hp}\Big(\frac{\|Y_1\|}{\theta}\Big).
\end{equation}

If $k=1$, then necessarily $i=1$ and $H_1 = H_2 = 0$, so \eqref{Le.S5.target} holds trivially. For the remainder of the proof, assume $k\ge 2$ (so that there are “other” spike indices $j\neq i$ to consider).

To this end, let
\[
U_1 = \theta^{-1}H_1\,\,\,\,and\,\,\,U_2 = \theta^{-1}H_2
\]
Clearly, using the equivalence of norms in finite dimensions, there exists some constant $\tilde{D} > 0$ such that
\[
\| H_1 -H_2\| = \|\mu^{-1}(Y_1 - Y_1(i,i)I)\| \le \tilde D \left\|\frac{Y_1}{\mu}\right\|
\]
Then, from Theorem \ref{Th.probedge} and Remark \ref{EVa.YN.normbound} we get
\begin{equation}\label{Le.S5.interm2}
|U_1 - U_2| = O_{hp}\Big(\frac{\|Y_1\|}{\theta^2}\Big) = o_{hp}(1),.
\end{equation}
Combining this with a bound on $\E(Y_n)$ (Lemma \ref{ExpectationBound}), we find that for $m=1,2$ and all $1 \le j,l \le k$,
\begin{equation}\label{Le.S5.interm3}
U_m(j,l) = (\alpha_j - \alpha_i)\one{j=l} + o_{hp}(1) \qquad (N \to \infty).
\end{equation}
In particular, as $N\to\infty$, both $U_1$ and $U_2$ converge (with high probability) to the diagonal matrix $\mathrm{Diag}(\alpha_1-\alpha_i,\ldots,\alpha_k-\alpha_i)$. Consequently,
\begin{equation}\label{Le.S5.interm4}
\lambda_i(U_m) = o_{hp}(1),, \qquad m=1,2.
\end{equation}

Let $\tilde{U}_m$ (for $m=1,2$) be the $(k-1)\times(k-1)$ matrix obtained by deleting the $i$th row and column of $U_m$, and let $\tilde{u}_m$ be the $(k-1)\times1$ vector obtained from the $i$th column of $U_m$ by deleting its $i$th entry. It is worth noting that
\begin{align*}
\|\tilde u_m\|=o_{hp}(1)\,,m=1,2
\end{align*}
and
\begin{equation}\label{Le.S5.interm5}
|\tilde{u}_1-\tilde{u}_2| = O_{hp}\Big(\frac{\|Y_1\|}{\theta^2}\Big),
\end{equation}
both of which follow from \eqref{Le.S5.interm3}. Also, \eqref{Le.S5.interm3} and \eqref{Le.S5.interm4} together imply that the matrix $\tilde{U}_m - \lambda_i(U_m)I_{k-1}$ converges (w.h.p.) to
\[
\mathrm{Diag}(\alpha_1-\alpha_i,\ldots,\alpha_{i-1}-\alpha_i,\alpha_{i+1}-\alpha_i,\alpha_k-\alpha_i)\,.
\]
Define $(k-1)$-dimensional vectors
\[
\tilde v_m=-\left[\tilde U_m-\lambda_i(U_m)I_{k-1}\right]^{-1}\tilde u_m\,,m=1,2\,,
\]
and extend each $\tilde{v}_m$ to a $k$-vector $v_m$ by inserting a $1$ in the $i$th position:
\[
v_m=\left[\tilde v_m(1),\ldots,\tilde v_m(i-1),\,1,\,\tilde v_m(i),\ldots,\tilde v_m(k-1)\right]^\prime\,,m=1,2\,.
\]
Now, proceeding exactly along the lines of the proof of Lemma 5.8 in \cite{chakrabarty2020eigenvalues}, we get
\[\
\|v_1-v_2\|=O_{hp}\left(\frac{\|Y_1\|}{\theta^2}\right)\, \text{ and } \,\, \|\tilde v_m\|=o_{hp}(1)\,,m=1,2\,.
\]
(Note also that $U_m(i,i)=0$ and by construction $v_m(i)=1$ for $m=1,2$.)

Therefore, we finally have
\begin{align*}
\big|\lambda_i(U_1)-\lambda_i(U_2)\big|
&= \Big|\sum_{j\ne i} U_1(i,j)v_1(j)- \sum_{j\ne i} U_2(i,j)v_2(j)\Big|\\
&\le \sum_{j\ne i} |U_1(i,j)||v_1(j)-v_2(j)| + \sum_{j\ne i} |U_1(i,j)-U_2(i,j)||v_2(j)|\\
&= O_{hp}\Big(\|\tilde{u}_1\|\|v_1-v_2\| + \|U_1-U_2\|\|\tilde{v}_2\|\Big)\\
&= o_{hp}\Big(\frac{\|Y_1\|}{\theta^2}\Big).
\end{align*}
Thus \eqref{Le.S5.target} is established, and the proof is complete.
\end{proof}

Now we are in a position to complete the proof of Theorem \ref{Th.EigenValueCLT}.

\begin{proof}[Proof of Theorem \ref{Th.EigenValueCLT}]
The proof will be complete if we can show that
\begin{equation}
\label{Th.EigenValueCLT.goal}
\mu - \E(\mu) = \mu^{-1}Y_1(i,i) + o_p(1),.
\end{equation}
Hence, replacing $\bar{\mu}$ by $\E(\mu)$ in Lemma \ref{Le.S5} (which is justified since $\bar{\mu}-\E(\mu)=o(1)$ as we will see) will serve our purpose.

Recalling that $Y_1(i,i) = \alpha_i\theta e_i'W e_i$, Lemma \ref{Le.S5} and Remark \ref{EVa.YN.normbound} together imply
\begin{align}
\nonumber
\mu - \bar{\mu} &= \mu^{-1}Y_1(i,i) + o_{hp}\Big(\frac{\|Y_1\|}{\theta} + 1\Big) \\
\label{Th.EigenValueCLT.interm1}
&= O_{hp}\Big(\frac{\|Y_1\|}{\theta} + 1\Big).
\end{align}
From the proof of Lemma \ref{Le.S5} we also know that

\[
\lim _{N \rightarrow \infty} \frac{\bar{\mu}}{\theta}=\alpha_i.
\]
This ensures
\begin{align}
\Big|\frac{1}{\bar{\mu}}Y_1(i,i) - \frac{1}{\mu}Y_1(i,i)\Big|
&= O_{hp}\Big(|\mu-\bar{\mu}|\frac{\|Y_1\|}{\theta^2}\Big) \nonumber \\
&= o_{hp}\big(|\mu-\bar{\mu}|\big) \nonumber \\
&= o_{hp}\Big(\frac{\|Y_1\|}{\theta} + 1\Big) \label{Th.EigenValueCLT.interm2}\\
&= o_p(1), \label{Th.EigenValueCLT.interm3}
\end{align}
where we used Remark \ref{EVa.YN.normbound} in the second step and \eqref{Th.EigenValueCLT.interm1} in the third.

Using Lemma \ref{Le.S5} and \eqref{Th.EigenValueCLT.interm2}, we get
\begin{equation}\label{Th.EigenValueCLT.interm4}
\mu = \bar{\mu} + \frac{1}{\bar{\mu}}Y_1(i,i) + o_{hp}\Big(\frac{\|Y_1\|}{\theta} + 1\Big).
\end{equation}
Now define
\[
R=\mu-\bar\mu-\frac{Y_1(i,i)}{\bar\mu}\,.
\]
Clearly,
\[
\mathbb E (R)=\mathbb E (\mu)-\bar\mu\,,
\]
and \eqref{Th.EigenValueCLT.interm4} implies that for any $\delta>0$ there exists $\eta>1$ and $N_0$, such that for all $N>N_0$
\begin{align*}
\E|R|
&\le \delta\Big(1 + \frac{\E\|Y_1\|}{\theta}\Big) + \E^{1/2}\Big(\mu-\bar{\mu}-\frac{1}{\bar{\mu}}Y_1(i,i)\Big)^2 O\big(e^{-(\log N)^\eta}\big).
\end{align*}
For the first term on the RHS, recalling the definition of $Y_1$, we note that
\begin{equation}\label{normY1}
\|Y_1\| \le \alpha_1\theta \sum_{j,l=1}^N |e_j'W e_l|,
\end{equation}
and hence
\[
\mathbb E(\|Y_1\|) =  O(\theta\sum_{j,l =1}^k \mathbb E|e_j^\prime W e_l|) \le O(\theta\sum_{j,l =1}^k \Var(e_j^\prime W e_l)) = O(\theta)\\     
\]

Thus
\[
\mathbb E|R|\le o(1)+\mathbb E^{1/2}\left(\mu-\bar\mu-\frac1{\bar\mu}Y_1(i,i)\right)^2 O\left(e^{-(\log N)^\eta}\right)\,.
\]
Let $H := \max_{1 \le j \le k} \sup_{x \in[0,1]}|h_j(x)|$. (Clearly $H<\infty$ by Assumption A1.) Since $|\mu| \le \|A\|$ and $\theta=O(N)$, we have
\[
|\mu|^2 \le \left(\sum_{j,l =1}^N|W(j,l)| + \alpha_1 K H^2N^2\right)^2\,.
\]
Then it follows that
\begin{align*}
\mathbb E\left(\sum_{j,l =1}^N|W(j,l)|\right)^2 &= \mathbb E\left(\sum_{j,l,j^\prime,l^\prime}|W(j,l)||W(j^\prime,l^\prime)|\right)\\
&\le \sum_{j,l,j^\prime,l^\prime}\sqrt{\mathbb E [W(j,l)^2]}\sqrt{\mathbb E [W(j^\prime,l^\prime)^2]} = O(N^4)\,
\end{align*}
and similarly $\E\left[\sum_{j,l=1}^N |W(j,l)|\right] = O(N^2)$. Therefore,
\begin{align}\label{Th.EigenValueCLT.interm5}
\E(\mu^2) = O(N^4).
\end{align}
Along with \eqref{Le.S3.detlim}, this gives
\begin{align}\label{Th.EigenValueCLT.interm6}
\frac{\E[Y_1(i,i)^2]}{\bar{\mu}^2} = O(1).
\end{align}
Using \eqref{Le.S3.detlim} and Assumption \ref{assump:A1} again, we have for large $N$:
\begin{align}\label{Th.EigenValueCLT.interm7}
\E(\mu\bar{\mu}) \le \bar{\mu}\E^{1/2}(\mu^2) = O(N^3),
\end{align}
and similarly
\begin{align}\label{Th.EigenValueCLT.inter8}
\E\Big(\mu\frac{Y_1(i,i)}{\bar{\mu}}\Big) = O(N^2) \qquad \text{ and } \qquad \E\big(\mu Y_1(i,i)\big) = O(N^3).
\end{align}
Therefore, \eqref{Th.EigenValueCLT.interm5}–\eqref{Th.EigenValueCLT.inter8} together imply
\[
\mathbb E^{1/2}\left(\mu-\bar\mu-\frac1{\bar\mu}Y_1(i,i)\right)^2 = O(N^2) = o\left(e^{(\log N)^\eta}\right)
\]
Hence we get
\[
\mathbb E|R|=o(1)\,,
\]
and therefore
\begin{equation}\label{lim.mu0-mubar}
\E(\mu) = \bar{\mu} + o(1).
\end{equation}
Then \eqref{Th.EigenValueCLT.interm2} and \eqref{Th.EigenValueCLT.interm3} imply that
\begin{align*}
\mu &= \E(\mu) + \frac{1}{\bar{\mu}}Y_1(i,i) + o_p\Big(\frac{\|Y_1\|}{\theta} + 1\Big)\\
&= \E(\mu) + \frac{1}{\mu}Y_1(i,i) + o_p(1).
\end{align*}
This proves \eqref{Th.EigenValueCLT.goal} and completes the proof of Theorem \ref{Th.EigenValueCLT}.
\end{proof}

\begin{proof}[Proof of Theorem \ref{Th.EigenValueMean}]
Henceforth we will use Assumption \ref{assump:A3}, i.e., the functions \( h_1,\ldots,h_k \) are Lipschitz. Along with this Assumption \ref{assump:A4} is also required. We aim to analyze the asymptotic behavior of the  expectation of eigenvalue \( \mu \), given by
\[
\mu = \lambda_i\left(\sum_{n=0}^\infty \mu^{-n} Y_n\right),
\]
as described in Theorem \ref{Th.EigenValueCLT}.

By Lemma \ref{Le.S2}, we have:
\[
\mu = \lambda_i\left(\sum_{n=0}^3 \mu^{-n} \mathbb{E}(Y_n)\right) + O_{hp}\left(\mu^{-1}\|Y_1\| + \sum_{n=4}^L \mu^{-n} \|\mathbb{E}(Y_n)\|\right) + o_{p}(1).
\]
Using Theorems\ref{Th.EigenValueCLT} and \ref{Th.probedge}, we get
\begin{align}
\label{EigenValueMean.0}\mu = \lambda_i\left(\sum_{n=0}^3 \mu^{-n} \mathbb{E}(Y_n)\right) + O_{p}\left(1 + \sum_{n=4}^L \mu^{-n} \|\mathbb{E}(Y_n)\|\right).
\end{align}
Also, from \eqref{EYn.norm} we deduce along lines of the proof of Lemma \ref{Le.S1},
\[
\sum_{n=4}^L \mu^{-n} \|\mathbb{E}(Y_n)\| = O_p\left(\frac{N^{2}}{\theta^3}\right)=o_p(1),
\]by Assumption \ref{assump:A4}.
Hence,
\begin{equation*}
\mu = \lambda_i\left(\sum_{n=0}^3 \mu^{-n} \mathbb{E}(Y_n)\right) + O_p\left(1 + \frac{N^2}{\theta^3}\right).
\end{equation*}
Using the fact that $$\mathbb E(e_j'W^3e_l) \sim N\quad \forall j,l$$ we get,
\begin{equation}
\label{Eq.EigMean.1} \mu = \lambda_i\left(\sum_{n=0}^2 \mu^{-n} \mathbb{E}(Y_n)\right) + O_p\left(1 + \frac{N^2}{\theta^3}\right).
\end{equation}
To simplify further, we again use the bound on \( \mathbb E \|Y_2\| \) from \ref{Le.S2} to get:
\[
\mu = \lambda_i(Y_0) + O_p\left(1 + \frac{N}{\theta} + \frac{N^2}{\theta^3}\right).
\]

We now analyze \( \lambda_i(Y_0) \). Using Fact 5.1 from Lemma 5.2 in \cite{chakrabarty2020eigenvalues}, we know that
\[
|\lambda_i(Y_0) - Y_0(i,i)| \le \sum_{j \ne i} |Y_0(i,j)|.
\]
Since the functions \( h_1,\ldots,h_k \) are Lipschitz, we have:
\[
e_i^\prime e_j = \one(i = j) + O\left(\frac{1}{N}\right).
\]
Using the definition of \( Y_0 \) from \eqref{Defn.Yn}, we obtain:
\[
Y_0(i,j) = 
\begin{cases}
\theta\alpha_i + O\left(\frac{\theta}{N}\right) & \text{if } i = j, \\
O\left(\frac{\theta}{N}\right) & \text{if } i \ne j.
\end{cases}
\]
Therefore,
\begin{equation}
\label{EigenValueMean.6}\lambda_i(Y_0) = \theta\alpha_i + O\left(\frac{\theta}{N}\right).
\end{equation}
Substituting this into the previous equation yields:
\[
\mu = \theta\alpha_i + O_p\left(\frac{\theta}{N} + 1 + \frac{N}{\theta} + \frac{N^2}{\theta^3}\right) = \theta\alpha_i + O_p\left(\frac{N}{\theta}\right).
\]

We now refine the expansion by reintroducing the \( \mathbb{E}(Y_2) \) term. Using the triangle inequality, we estimate
\[
\left\| \mu^{-2} \mathbb{E}(Y_2) - (\theta\alpha_i)^{-2} \mathbb{E}(Y_2) \right\| = O_p\left(\frac{N}{\theta^4} \|\mathbb{E}(Y_2)\|\right) = O_p\left(\frac{N^2}{\theta^3}\right),
\]
where we used that \( \mu^{-2} - (\theta\alpha_i)^{-2} = O_p(N/\theta^4) \), and \( \|\mathbb{E}(Y_2)\| = O(N\theta) \), where the later bound follows from Lemma \ref{ExpectationBound}.

Substituting this into \eqref{Eq.EigMean.1}, we finally get:
\begin{align}
\label{EigenValueMean.5}\mu = \lambda_i\left(Y_0 + \frac{\mathbb{E}(Y_2)}{(\theta\alpha_i)^2}\right) + O_p\left(1 + \frac{N^2}{\theta^3}\right).
\end{align}

Thus we have derived the required approximation for \( \mu \). It remains to connect this to the mean of the eigenvalue. From Theorem \ref{Th.EigenValueCLT}, we have:
\[
\mu - \mathbb{E}(\mu) = O_p(1).
\]
It follows from \eqref{EigenValueMean.5}
\[
\lambda_i(B) - \mu = O_p\left(1 + \frac{N^2}{\theta^3}\right),
\]
where \( B := Y_0 + \frac{\mathbb{E}(Y_2)}{(\theta\alpha_i)^2} \) is the deterministic approximation matrix.

Combining the two bounds, we conclude:
\[
\lambda_i(B) - \mathbb{E}(\mu) = O\left(1 + \frac{N^2}{\theta^3}\right).
\]
Now, appeal to equations \eqref{lim.mu0-mubar}, \eqref{EigenValueMean.6} and \eqref{EigenValueMean.5} completes the proof.
\end{proof}
\begin{remark}
Observe that, if we use Lemma \ref{ConcentrationBound} in \eqref{EigenValueMean.0}, then, \eqref{EigenValueMean.5} provides a more precise approximation of closeness of $\mu$ to the edge of the spectrum than, Theorem \ref{Th.probedge}, i.e 
\begin{equation}\label{approx.mu/theta}
\frac{\mu}{\theta\alpha_i} = 1+ o_{hp}\left(\frac{\sqrt{N}}{\theta}\right).  
\end{equation}
This is because assumption \ref{assump:A3} was not plugged in while proving Theorem \ref{Th.probedge}.
\end{remark}

\section{Technical lemmas} \label{appendix:B}
\begin{proof}[Proof of Lemma \ref{lemma:NormBound}]
Let \( \lambda_1 \) and \( \lambda_n \) denote the largest and smallest eigenvalues of \( W_N \), respectively. By our assumptions, the entries of \( W_N \) satisfy the moment conditions required by \cite[Theorem 2.1.22]{anderson2010introduction}. Using the notation from \cite{anderson2010introduction}, define the normalized moment of the empirical spectral distribution as
\[
\langle L_N, x^k \rangle := \frac{1}{N^{k/2 + 1}} \mathbb{E}(\mathrm{tr}(W_N^k)).
\]
Then, by the derivation in \cite[Theorem 2.1.22]{anderson2010introduction}, for some constant \( \tilde{B} = \tilde{B}(B) \), we have:
\[
\langle L_N, x^{2n} \rangle \le (2\sqrt{\|f\|_\infty})^{2n} \sum_{j=0}^n \left( \frac{n^{\tilde{B}}}{N} \right)^j.
\]

Now define \( n = n(N) := \lfloor (\log N)^3 \rfloor \). Observe that as \( N \to \infty \),
\[
\frac{n(N)}{\log N} \to \infty \quad \text{and} \quad \frac{n(N)}{N} \to 0.
\]
Fix any constant \( \tilde{C} > 2\sqrt{\|f\|_\infty} \). By Markov's inequality, we get:
\begin{align*}
\mathbb{P}(\lambda_1 > \tilde{C}) 
&= \mathbb{P}\left( \|W_N\| > \tilde{C} \right) 
\le \frac{\mathbb{E}(\mathrm{tr}(W_N^{2n}))}{\tilde{C}^{2n}} 
= N \cdot \frac{\langle L_N, x^{2n} \rangle}{\tilde{C}^{2n}} \cdot N^{n},\\
&\le N \cdot \left( \frac{2\sqrt{\|f\|_\infty}}{\tilde{C}} \right)^{2n} \sum_{j=0}^{n} \left( \frac{n^{\tilde{B}}}{N} \right)^j.
\end{align*}

Since \( n = o(N) \), the sum is \( O(1) \), and hence:
\begin{align*}
\mathbb{P}(\lambda_1 > \tilde{C}) 
&\le C_1 N \left( \frac{2\sqrt{\|f\|_\infty}}{\tilde{C}} \right)^{2n} \\
&= C_1 \exp \left( \log N + 2n \log \left( \frac{2\sqrt{\|f\|_\infty}}{\tilde{C}} \right) \right) \\
&= C_1 \exp \left( -a (\log N)^3 \right)
\end{align*}
for some constant \( a > 0 \), provided \( \tilde{C} > 2\sqrt{\|f\|_\infty} \) is fixed.

Therefore,
\begin{equation} \label{norm.upper}
\mathbb{P} \left( \lambda_1 > \tilde{C} \right) \le C_1 \exp \left( -a (\log N)^3 \right).
\end{equation}

A symmetric argument applied to \( -W_N \) yields a similar bound for the smallest eigenvalue:
\begin{equation} \label{norm.lower}
\mathbb{P} \left( \lambda_n < -\bar{C} \right) \le C_2 \exp \left( -b (\log N)^3 \right),
\end{equation}
for some constants \( \bar{C} > 2\sqrt{\|f\|_\infty} \) and \( b > 0 \).

Now let \( C := \max(\tilde{C}, \bar{C}) \) and \( c := \min(a, b) \). Combining \eqref{norm.upper} and \eqref{norm.lower}, we get:
\[
\mathbb{P} \left( \left\| \frac{W_N}{\sqrt{N}} \right\| > C \right) 
= \mathbb{P} \left( \lambda_1 > C \sqrt{N} \,\, \text{or} \,\, \lambda_n < -C \sqrt{N} \right) 
= O\left( \exp(-c (\log N)^3) \right),
\]
which completes the proof.
\end{proof}

\begin{proof}[Proof of Lemma \ref{ExpectationBound}]
As in the proof of the previous lemma, we consider the high-probability event
\[
E_N := \left\{ \|W_N\| \le C\sqrt{N} \right\}
\]
for some constant \( C > 0 \). By Lemma~\ref{lemma:NormBound}, we have
\[
\mathbb{P}(E_N^c) \le e^{-a (\log N)^3}
\]
for some constant \( a > 0 \).

On the event \( E_N \), we trivially have
\[
|\mathbb{E}(e_1^\prime W^n e_2 \cdot \mathbf{1}_{E_N})| \le (C\sqrt{N})^n.
\]

It remains to show that the contribution from the complement event \( E_N^c \) is negligible. For this, we estimate the second moment \( \mathbb{E}(e_1^\prime W^n e_2)^2 \).

Write out the full expansion:
\[
(e_1^\prime W^n e_2)^2 = \sum_{i_1, \ldots, i_{n+1}} \sum_{j_1, \ldots, j_{n+1}} 
\left[ e_1(i_1) \left( \prod_{l=1}^{n} W(i_l, i_{l+1}) \right) e_2(i_{n+1}) \right]
\left[ e_1(j_1) \left( \prod_{l=1}^{n} W(j_l, j_{l+1}) \right) e_2(j_{n+1}) \right].
\]

Let \( S \) denote a typical summand in the above expression. Using Hölder's inequality and the moment assumptions on \( W \), we get:
\[
\mathbb{E}(|S|^{2n}) \le \frac{H^{8n}}{N^{4n}} (2n)^{4Bn^2},
\]
where \( H \) and \( B \) are constants depending on the uniform moment bound on the entries of \( W \).

Hence, for \( 2 \le n \le L \) and \( N \) sufficiently large, we can find a constant \( C_1 > 0 \) such that
\begin{align*}
\mathbb{E}(e_1^\prime W^n e_2)^2 
&\le H^4 \left( N (2n)^B \right)^{2n} 
\le H^4 \cdot (N^{B+1})^{2n} 
= N^{n C_1}.
\end{align*}

Now using Cauchy–Schwarz inequality, we estimate:
\begin{align*}
\mathbb{E}(e_1^\prime W^n e_2 \cdot \mathbf{1}_{E_N^c}) 
&\le \left[ \mathbb{E}(e_1^\prime W^n e_2)^2 \right]^{1/2} \cdot \left[ \mathbb{P}(E_N^c) \right]^{1/2} \\
&\le \left( N^{n C_1} \right)^{1/2} \cdot e^{- \frac{a}{2} (\log N)^3 } \\
&= \exp\left( n C_1 \log N / 2 - \frac{a}{2} (\log N)^3 \right).
\end{align*}

Since \( n \le L = \lfloor \log N \rfloor \), we have:
\[
\mathbb{E}(e_1^\prime W^n e_2 \cdot \mathbf{1}_{E_N^c}) 
\le \exp\left( - (\log N)^3 \left[ \frac{a}{2} - \frac{C_1}{2 \log N} \right] \right),
\]
which vanishes as \( N \to \infty \). Thus, for \( N \) sufficiently large,
\begin{equation}\label{ExpBound.Negligible}
\mathbb{E}(e_1^\prime W^n e_2 \cdot \mathbf{1}_{E_N^c}) \le \exp\left( -(\log N)^3 \left[ \frac{a}{2} - o(1) \right] \right) \to 0.
\end{equation}

Combining the bounds on \( \mathbb{E}(e_1^\prime W^n e_2 \cdot \mathbf{1}_{E_N}) \) and \( \mathbb{E}(e_1^\prime W^n e_2 \cdot \mathbf{1}_{E_N^c}) \), we conclude that:
\[
\mathbb{E}(e_1^\prime W^n e_2) \le (C\sqrt{N})^n + o(1).
\]

This completes the proof.
\end{proof}

\begin{proof}[Proof of Lemma \ref{Align.bounds}]
We prove the claims separately for the cases \( j \ne i \) and \( j = i \).

\smallskip
\noindent\textbf{Case 1: Off-diagonal alignment, \( j \ne i \).}  
From the master equation \eqref{eq.master1}, we have:
\[
e_j^\prime v = \frac{\theta}{\mu} \sum_{l=1}^{k} \alpha_l \cdot e_j^\prime \left(I - \frac{W}{\mu}\right)^{-1} e_l.
\]
Isolating the \( l = j \) term, this becomes:
\[
e_j^\prime v = \frac{\theta \alpha_j}{\mu} \cdot e_j^\prime \left(I - \frac{W}{\mu}\right)^{-1} e_j + \frac{\theta}{\mu} \sum_{l \ne j} \alpha_l \cdot e_j^\prime \left(I - \frac{W}{\mu}\right)^{-1} e_l.
\]

Multiplying both sides by the bracketed expression, we rearrange:
\[
e_j^\prime v \left(1 - \frac{\theta \alpha_j}{\mu} \cdot e_j^\prime \left(I - \frac{W}{\mu}\right)^{-1} e_j \right) 
= \frac{\theta}{\mu} \sum_{l \ne j} \alpha_l (e_l^\prime v) \cdot e_j^\prime \left(I - \frac{W}{\mu}\right)^{-1} e_l.
\]

Now, using
\begin{itemize}
\item[(a)]The approximation \( \mu/\theta \sim \alpha_i \) from \eqref{approx.mu/theta},
\item[(b)] The expansion of \( e_j^\prime \left(I - \frac{W}{\mu} \right)^{-1} e_j = 1/\alpha_j + o_{hp}(\sqrt{N}/\theta) \) from Lemma~\ref{EigenVector.1},
\item[(c)] The fact that \( e_j^\prime v \) is bounded,
\end{itemize}
we obtain:
\begin{equation} \label{align.ortho.final}
e_j^\prime v \left( 1 - \frac{\alpha_j}{\alpha_i} + o_{hp}\left( \frac{\sqrt{N}}{\theta} \right) \right) 
= o_{hp}\left( \frac{\sqrt{N}}{\theta} \right).
\end{equation}

Thus, for \( j \ne i \), we conclude:
\[
e_j^\prime v = o_{hp}\left( \frac{\sqrt{N}}{\theta} \right).
\]

\smallskip
\noindent\textbf{Case 2: Diagonal alignment, \( j = i \).}  
From the second master equation \eqref{eq.master2}, we have:
\[
\sum_{l,m=1}^k \alpha_l \alpha_m (e_l^\prime v)(e_m^\prime v) \cdot e_l^\prime \left(I - \frac{W}{\mu}\right)^{-2} e_m = \theta^{-2} \mu^2.
\]

Isolating the \( (l,m) = (i,i) \) term and grouping the rest:
\[
\alpha_i^2 (e_i^\prime v)^2 \cdot e_i^\prime \left(I - \frac{W}{\mu}\right)^{-2} e_i 
= \theta^{-2} \mu^2 - \sum_{(l,m) \ne (i,i)} \alpha_l \alpha_m (e_l^\prime v)(e_m^\prime v) \cdot e_l^\prime \left(I - \frac{W}{\mu}\right)^{-2} e_m.
\]

Using the bounds:
\( e_j^\prime v = o_{hp}(\sqrt{N}/\theta) \) for \( j \ne i \),   \( \mu/\theta \sim \alpha_i \), and  Lemma~\ref{EigenVector.1} again for the matrix term,

we get:
\[
\alpha_i^2 (e_i^\prime v)^2 \cdot e_i^\prime \left(I - \frac{W}{\mu}\right)^{-2} e_i 
= \alpha_i^2 \left(1 + o_{hp}\left( \frac{\sqrt{N}}{\theta} \right)\right) + o_{hp}\left( \frac{N}{\theta^2} \right).
\]

Using the bound \( e_i^\prime \left(I - \frac{W}{\mu} \right)^{-2} e_i = 1 + o_{hp}(1) \) from Lemma~\ref{EigenVector.1}, we conclude:
\[
(e_i^\prime v)^2 = 1 + o_{hp}\left( \frac{\sqrt{N}}{\theta} \right),
\]
and hence
\[
e_i^\prime v = 1 + o_{hp}\left( \frac{\sqrt{N}}{\theta} \right),
\]
by the Taylor expansion \( (1 + x)^{1/2} = 1 + O(x) \) for small \( x \).

\smallskip
This completes the proof.
\end{proof}

\begin{proof}[Proof of Lemma \ref{ConcentrationBound}]
We divide the proof into two parts: \( n \ge 2 \) and \( n = 1 \).

\subsubsection*{Step 1: Case \( 2 \le n \le L \)}

We begin by recalling the moment bound assumption from Section~1. For \( k \ge 2 \),
\[
\mathbb{E}\left( \left| \frac{W_{ij}}{\sqrt{N}} \right|^k \right) \le \frac{k^{Bk}}{N^{k/2}} \le \frac{k^{Bk}}{N}.
\]
With a slight abuse of notation, we continue writing \( W_{ij} \) for the normalized entries \( \frac{W_{ij}}{\sqrt{N}} \).

Let us consider the centered linear statistic:
\begin{align}
\label{eq:centered}
e_1^\prime W^n e_2 - \mathbb{E}(e_1^\prime W^n e_2) 
= \sum_{\mathbf{i} \in \{1, \ldots, N\}^{n+1}} e_1(i_1) e_2(i_{n+1}) 
\left( \prod_{\ell = 1}^n W(i_\ell, i_{\ell+1}) - \mathbb{E} \prod_{\ell = 1}^n W(i_\ell, i_{\ell+1}) \right).
\end{align}

To exploit independence, we symmetrize \( W \) as \( W = W' + W'' \), where:
\[
W'(i,j) := W(i,j)\cdot \mathbf{1}_{i \le j}, \quad 
W''(i,j) := W(i,j)\cdot \mathbf{1}_{i > j}.
\]
This allows us to expand the product into \( 2^n \) terms. Fix one such term:
\begin{align*}
P_n := \sum_{\mathbf{i}} e_1(i_1) e_2(i_{n+1}) 
\left( \prod_{\ell=1}^n W'(i_\ell, i_{\ell+1}) - \mathbb{E} \prod_{\ell=1}^n W'(i_\ell, i_{\ell+1}) \right).
\end{align*}

Let \( H := \max_{1 \le j \le k} \sup_{x \in [0,1]} |h_j(x)| \), and assume without loss of generality that \( H, B > 1 \). Following the argument in Lemma~6.5 of \cite{erdHos2013spectral}, and restoring the scaling \( W_{ij} = \frac{W_{ij}}{\sqrt{N}} \), we obtain:
\begin{align}
\label{Expc.bound}
\mathbb{E}(|P_n|^p) \le \frac{(2H^2 n p)^{np} (np)^{Bnp} N^{np/2}}{N^{p/2}}.
\end{align}

Since there are at most \( 2^n \) such terms, all with the same bound, we apply Markov's inequality:
\begin{align*}
\mathbb{P}\left( \left| e_1^\prime W^n e_2 - \mathbb{E}(e_1^\prime W^n e_2) \right| > N^{(n-1)/2} (\log N)^{n\xi/4} \right)
&\le \frac{(4H^2 (np)^{B+1})^{np} N^{np/2}}{N^{p/2} N^{(n-1)p/2} (\log N)^{pn\xi/4}} \\
&= \frac{(\Sigma np)^{(B+1)np}}{(\log N)^{pn\xi/4}},
\end{align*}
where \( \Sigma := 4H^2 \). Choose \( \eta \in \left(1, \frac{\xi}{4(B+1)}\right) \) and define
\[
p := \frac{(\log N)^{\eta}}{\Sigma n}.
\]
Then \( p > 1 \) for large \( N \), and we obtain:
\begin{align*}
\mathbb{P}\left( \left| e_1^\prime W^n e_2 - \mathbb{E}(e_1^\prime W^n e_2) \right| > N^{(n-1)/2} (\log N)^{n\xi/4} \right) 
\le \exp\left( - \frac{(\log N)^{\eta}}{\Sigma} \left( \frac{\xi}{4} - \eta(B+1) \right) \log \log N \right),
\end{align*}
which decays faster than any inverse polynomial as \( N \to \infty \), uniformly for \( 2 \le n \le L \).

\subsubsection*{Step 2: Case \( n = 1 \)}

We use the decomposition:
\begin{align}
\label{Baisc}
e_1^\prime W e_2 = \sum_{i \le j} e_1(i) W(i,j) e_2(j) + \sum_{i > j} e_1(i) W(i,j) e_2(j) =: A + B.
\end{align}

By the proof of Lemma~3.8 in \cite{erdHos2013spectral}, we get (as in \eqref{Expc.bound}) for \( p \in 2\mathbb{N} \),
\[
\mathbb{E}(|A|^p) \le p^{(B+1)p} \left[ A_1 + \left( \frac{ \sum_{i \le j} |e_1(i) e_2(j)|^2 }{N} \right)^{1/2} \right]^p N^{p/2},
\]
where \( A_1 := \max_{i,j} |e_1(i) e_2(j)| = O\left( \frac{1}{N} \right) \). Thus, the quantity inside brackets is \( O\left( \frac{1}{\sqrt{N}} \right) \).

Let \( C_1 > 1 \), then:
\begin{align*}
\mathbb{P}(|A| > 4(\log N)^{\xi/4}) 
&\le \mathbb{P}\left( |A| > C_1 \sqrt{N} (\log N)^{\xi/4} \cdot \left[ A_1 + \left( \frac{ \sum_{i \le j} |e_1(i) e_2(j)|^2 }{N} \right)^{1/2} \right] \right) \\
&\le \frac{p^{(B+1)p}}{(\log N)^{p\xi/4}}.
\end{align*}

Now take \( \eta \in \left(1, \frac{\xi}{4(B+1)}\right) \) and \( p := (\log N)^{\eta} \), which yields:
\[
\mathbb{P}(|A| > 4(\log N)^{\xi/4}) \le \exp\left( - (\log N)^{\eta} (\tfrac{\xi}{4} - \eta(B+1)) \log \log N \right).
\]

The same argument holds for \( B \), and so by \eqref{Baisc}, we conclude:
\[
e_1^\prime W e_2 = O_{hp}((\log N)^{\xi/4}).
\]

Putting both parts together, we conclude that for all \( 1 \le n \le L \),
\[
e_1^\prime W^n e_2 - \mathbb{E}(e_1^\prime W^n e_2) = O_{hp}\left( N^{(n-1)/2} (\log N)^{n\xi/4} \right).
\]
This completes the proof.
\end{proof}

\begin{proof}[Proof of Lemma \ref{Deloc.est}]
The proof is similar to Lemma 7.10 in \cite{erdHos2013spectral}, with minor modifications tailored to our setup.

Fix \( j \in \{1,2,\ldots,k\} \) and \( i \in \{1,2,\ldots,N\} \). Then by definition,
\[
(W^n e_j)(i) = \sum_{l_1, l_2, \ldots, l_n} W(i, l_1) W(l_1, l_2) \cdots W(l_{n-1}, l_n) \cdot \frac{h_j\left(\frac{l_n}{N}\right)}{\sqrt{N}}.
\]

Let \( H := \max_{1 \le j \le k} \sup_{x \in [0,1]} |h_j(x)| \), and note that \( H < \infty \) since \( h_j \in C[0,1] \). Proceeding as in Lemma 7.10 of \cite{erdHos2013spectral}, we obtain the moment bound:
\[
\mathbb{E}\left( |(W^n e_j)(i)|^p \right) \le (H n p)^{n p} \cdot N^{\frac{(n-1)p}{2}}.
\]

Now apply Markov's inequality. For any \( \varepsilon > 0 \),
\begin{align*}
\mathbb{P}\left( |(W^n e_j)(i)| > N^{\frac{n-1}{2}} (\log N)^{\frac{n\xi}{4}} \right) 
&\le \frac{ \mathbb{E} \left( |(W^n e_j)(i)|^p \right) }{ N^{\frac{(n-1)p}{2}} (\log N)^{\frac{n\xi p}{4}} } \\
&\le \frac{ (H n p)^{n p} }{ (\log N)^{\frac{n \xi p}{4}} }.
\end{align*}

Choose \( \eta \in \left(1, \frac{\xi}{4} \right) \), and let
\[
p = \frac{ (\log N)^\eta }{H n }.
\]
This choice ensures \( p > 1 \) for large \( N \), and plugging it in gives:
\begin{align*}
\mathbb{P}\left( |(W^n e_j)(i)| > N^{\frac{n-1}{2}} (\log N)^{\frac{n\xi}{4}} \right)
&\le \exp\left( - (\log N)^\eta \left( \frac{\xi}{4} - \eta \right) \log \log N \right).
\end{align*}

This bound decays faster than any inverse polynomial, and it holds uniformly for all \( 2 \le n \le L \).

For the case \( n = 1 \), an identical argument as in the proof of Lemma~\ref{ConcentrationBound} applies, yielding the same high-probability estimate.

\smallskip
This completes the proof.
\end{proof}

\end{document}